\documentclass[letterpaper,10pt]{article}

\usepackage[margin=0.5in]{geometry}

\usepackage{fullpage}
\usepackage{enumerate}
\usepackage{authblk}
\usepackage{verbatim}
\usepackage{section, amsthm, textcase, setspace, amssymb, amscd, lineno, amsmath, amsfonts, latexsym, fancyhdr, longtable, ulem, mathtools}
\usepackage{epsfig, graphicx, pstricks,pst-grad,pst-text,tikz,colortbl}
\usepackage{epsf}
\usepackage{fancybox, color}
\usepackage{float}
\usepackage[rflt]{floatflt}
\usepackage{latexsym,enumitem}
\usetikzlibrary{fit,matrix,positioning}
\usepackage{pdflscape}
\usetikzlibrary{decorations.pathreplacing}
\usepackage[most]{tcolorbox}
\usepackage{lipsum}
\usetikzlibrary{decorations.markings}
\usetikzlibrary{patterns}

\usepackage{mathrsfs}

\newcommand{\precdot}{\prec\mathrel{\mkern-5mu}\mathrel{\cdot}}

\newtheorem{theorem}{Theorem}[section]
\newtheorem{lemma}[theorem]{Lemma}
\newtheorem{corollary}[theorem]{Corollary}

\newtheorem{conj}[theorem]{Conjecture}

\newtheorem{definition}[theorem]{Definition}

\newtheorem{example}[theorem]{Example}

\newtheorem{prop}[theorem]{Proposition}
\newtheorem*{theorem*}{Theorem}
\newtheorem{remark}[theorem]{Remark}

\newcommand{\col}[3][]{\text{Col}^{#1}(#2;#3)}
\newcommand{\ldownarrow}{\Big\downarrow}

\usepackage{xargs}
\usepackage[colorinlistoftodos,prependcaption,textsize=tiny,linecolor=red,backgroundcolor=red!25,bordercolor=red]{todonotes}
\newcommandx{\nick}[2][1=]{\todo[linecolor=blue,backgroundcolor=blue!25,bordercolor=blue,#1]{#2 ---Nick}}
\newcommandx{\laura}[2][1=]{\todo[linecolor=orange,backgroundcolor=orange!25,bordercolor=orange,#1]{#2 ---Laura}}
\newcommandx{\felix}[2][1=]{\todo[linecolor=purple,backgroundcolor=purple!25,bordercolor=purple,#1]{#2 ---Felix}}
\newcommandx{\etienne}[2][1=]{\todo[linecolor=red,backgroundcolor=red!25,bordercolor=red,#1]{#2 ---Etienne}}

\usepackage[backend=bibtex]{biblatex}
\begin{document}

\title{On ranked and bounded Kohnert posets}

\author[*]{Laura Colmenarejo}
\author[*]{Felix Hutchins}
\author[*]{Nicholas Mayers}
\author[*]{Etienne Phillips}

\affil[*]{Department of Mathematics, North Carolina State University, Raleigh, NC, 27605}

\maketitle

\bigskip
\begin{abstract} 
\noindent
In this paper, we explore combinatorial properties of the posets associated with Kohnert polynomials. In particular, we determine a sufficient condition guaranteeing when such ``Kohnert posets'' are bounded and two necessary conditions for when they are ranked. Moreover, we apply the aforementioned conditions to find complete characterizations of when Kohnert posets are bounded and when they are ranked in special cases, including those associated with Demazure characters.
\end{abstract}

\section{Introduction}

In his thesis \cite{Kohnert}, Kohnert showed that Demazure characters can be viewed as generating functions of certain diagrams of unit cells distributed in the first quadrant. Moreover, it was conjectured, and subsequently proven \cite{AssafSchu, Winkel2, Winkel1}, that a similar result applied to Schubert polynomials. Motivated by this, Assaf and Searles \cite{KP1} generalized the aforementioned construction and defined the notion of ``Kohnert polynomials''. Such polynomials are generating functions of general configurations of cells in the first quadrant which are related to a ``seed'' diagram by a sequence of moves, called ``Kohnert moves''. The effect of applying a Kohnert move is that at most one rightmost cell of a row in the given diagram moves to the first empty position below in the same column (see Section~\ref{sec:KDprelim}). In particular, starting from a diagram $D$, one forms a collection of diagrams $KD(D)$ by applying sequences of Kohnert moves to $D$. Each diagram in $KD(D)$ corresponds to a monomial by encoding the number of cells in each row. The Kohnert polynomial associated with a diagram $D$ is then the generating function for the diagrams of $KD(D)$.

In \cite{KP2}, the author notes that for a diagram $D$ one can also naturally define a poset structure on the collection of diagrams $KD(D)$, which we refer to as its ``Kohnert poset'' (see also \cite{KP3}). Moreover, the author illustrates that Kohnert posets are not usually ``well-behaved''. For instance, in general, they are not bounded\footnote{In \cite{KP2}, the author noted that in general Kohnert posets do not have a unique minimal element. As Kohnert posets always have a unique maximal element, having a unique minimal element is equivalent to being bounded.} nor ranked. The goal of this paper is to initiate a more in-depth investigation into Kohnert posets, with a focus of identifying such posets with those two properties. Kohnert posets are hard to consider in a general form and their ``bad behavior'' drives us to focus on a few families of diagrams. Undertaking a systematic approach, we initially study families of diagrams where the number of cells allowed in each column of a diagram increases from one family to the next. However, the study becomes challenging already for the case of diagrams with cells only in two rows. Consequently, we move to considering families of diagrams for which the arrangements of cells in nonempty rows is more structured. Among the families of diagrams considered in this last part of our study are ``key diagrams'' whose Kohnert polynomials are Demazure characters.

The outline of the paper is as follows. In Section~\ref{sec:prelims} we introduce Kohnert diagrams and present a brief introduction to poset theory with the concepts and results needed for our study. Also included in Section~\ref{sec:prelims} is a subsection devoted to Kohnert posets in which, along with their definition, we facilitate the discussion by introducing some language and notation. In Section~\ref{sec:sres} we establish several general results presenting conditions for a diagram to be associated with a Kohnert poset which is either bounded or ranked. More concretely, in Proposition~\ref{prop:nummingen} we find that a sufficient condition for a diagram to generate a bounded Kohnert poset can be given in terms of the sequence of the number of cells per nonempty column; and in Theorems~\ref{thm:ranked1} and~\ref{thm:ranked2} we find two necessary conditions for a diagram to generate a ranked Kohnert poset in terms of avoiding certain ``subdiagrams''. In the following four sections, we apply the results of Section~\ref{sec:sres} to distinct families of diagrams. Specifically, Section~\ref{sec:obpc} is dedicated to diagrams with at most one cell in each column, Section~\ref{sec:2r} to diagrams with the cells concentrated in two rows, Section~\ref{sec:cd} to key diagrams, and Section~\ref{sec:chd} to checkered diagrams. For each family, we formally define the diagrams and establish characterizations of when the associated Kohnert posets are bounded and when they are ranked. Finally, in the conclusion, we discuss ongoing work as well as open problems.

\vspace{1cm}

\section{Preliminaries}\label{sec:prelims}

In this section, we cover the preliminaries necessary to define Kohnert posets as well as to state our research questions. To start, we introduce the elements of Kohnert posets.

\subsection{Kohnert diagrams}\label{sec:KDprelim}
As mentioned in the introduction, the underlying sets of Kohnert posets are certain collections of diagrams. Formally, a \textbf{diagram} $D$ is an array of finitely many cells in $\mathbb{N}\times\mathbb{N}$ (see Figure~\ref{fig:diagram}). We may also think of a diagram as the set of row-column coordinates of the cells defining it, where rows are labeled from bottom to top and columns from left to right. Therefore, we write $(r,c)\in D$ if the cell in row $r$ and column $c$ is in the diagram $D$. We say that a diagram $D'$ is a \textbf{subdiagram of $D$} if $D'\subseteq D$ as sets of cells. 

Given a diagram $D$, we denote by $|D|$ the total number of cells contained in the diagram and by $rowsum(D)$ the sum of the row coordinates of the cells in the diagram; that is, $$|D| = |\{ (r,c) \ |\ (r,c)\in D\}| \quad \text{ and }\quad rowsum(D) =\sum_{(r,c)\in D}r.$$

\begin{example}
In Figure~\ref{fig:diagram}, we illustrate the diagram $D=\{(1,3),(2,1),(2,2),(3,2)\}$. For this diagram, $|D|=4$ and $rowsum(D)=1+2+2+3=8$.

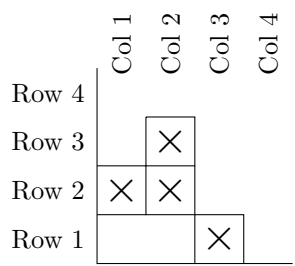
\begin{figure}[H]
    \centering
    $$\begin{tikzpicture}[scale=0.65]
  \node at (1.5, 2.5) {$\bigtimes$};
  \node at (0.5, 1.5) {$\bigtimes$};
  \node at (1.5, 1.5) {$\bigtimes$};
  \node at (2.5, 0.5) {$\bigtimes$};
  \node at (-1,0.5) {Row 1};
  \node at (-1,1.5) {Row 2};
  \node at (-1,2.5) {Row 3};
  \node at (-1,3.5) {Row 4};
  \node at (0.5, 4.5) {\rotatebox{90}{Col 1}};
  \node at (1.5, 4.5) {\rotatebox{90}{Col 2}};
  \node at (2.5, 4.5) {\rotatebox{90}{Col 3}};
  \node at (3.5, 4.5) {\rotatebox{90}{Col 4}};
  \draw (0,4)--(0,0)--(4,0);
  \draw (1,3)--(2,3)--(2,2)--(1,2)--(1,3);
  \draw (2,2)--(2,1)--(1,1)--(1,2)--(0,2);
  \draw (0,1)--(1,1);
  \draw (2,0)--(2,1)--(3,1)--(3,0);
\end{tikzpicture}$$
    \caption{Diagram $D=\{(1,3),(2,1),(2,2),(3,2)\}$}
    \label{fig:diagram}
\end{figure}
\end{example}

\begin{remark}\label{rem:diagconv}
    Ongoing, when illustrating diagrams with a particular form, it will prove helpful to decorate regions to indicate a particular structure. Other than describing the properties of regions in words \textup(usually in parentheses\textup), regions shaded gray represent empty regions containing no cells, and regions shaded with diagonal lines will represent regions that are arbitrary, i.e., the placement of cells can be arbitrary.
\end{remark}

Now, to any diagram $D$ we can apply what are called ``Kohnert moves'' defined as follows. For $r>0$, applying a \textbf{Kohnert move} at row $r$ of $D$ results in moving the rightmost cell in row $r$ to the first \textit{empty} position below in the same column (if such a position exists), jumping over other cells as needed. If applying a Kohnert move at row $r$ causes the cell in position $(r,c)\in D$ to move down to position $(r',c)$, forming the diagram $D'$, then we write $$D'=D\ldownarrow^{(r,c)}_{(r',c)}$$ and refer to this Kohnert move as moving the cell $(r,c)$ to position $(r',c)$. If there is no empty position below the cell $(r,c)$ in the same column, then the Kohnert move does nothing, leaving the diagram fixed. In the case that a Kohnert move does not fix the diagram, we say that it is a \textbf{nontrivial Kohnert move}. We let $KD(D)$ denote the set of all diagrams that can be obtained from $D$ by applying a sequence of Kohnert moves. Sets of the form $KD(D)$ for a diagram $D$ are the underlying sets of Kohnert posets. Note that although some diagrams can be formed in multiple ways from an initial diagram $D$, which will be relevant when discussing poset structure, $KD(D)$ is a set and does not contain repeated elements.

\begin{remark}\label{rem:rempty}
Kohnert moves only affect cells within the same column, which implies that empty columns in a diagram $D$ do not interfere with nor are affected by Kohnert moves. Therefore, empty columns are not relevant to Kohnert moves nor the poset structure introduced later in the paper. Thus, ongoing, we assume that the diagrams do not contain empty columns to the left of nonempty columns. 
\end{remark}

\begin{example}\label{ex:KDD}
In Figure~\ref{fig:KD}, we illustrate the diagrams of $KD(D_0)$ for the diagram $D_0$ of Figure~\ref{fig:diagram}, and included here as the first diagram. The other diagrams in Figure~\ref{fig:KD} are obtained by applying sequences of Kohnert moves to $D_0$. For instance, $D_1$ is obtained by applying a Kohnert move at row 2 of $D_0$, $D_2$ by applying two Kohnert moves at row 2 of $D_0$, $D_3$ by applying a Kohnert move at row 3 of $D_0$, and $D_4$ by applying two Kohnert moves at row 2 followed by another one at row 3 of $D_0$. Note that $D_3$ could also be formed by applying a Kohnert move at row 2 followed by another one at row 3 of $D_0$. 
    \begin{figure}[H]
    \centering
    $$\begin{tikzpicture}[scale=0.65]
  \node at (1.5, 2.5) {$\bigtimes$};
  \node at (0.5, 1.5) {$\bigtimes$};
  \node at (1.5, 1.5) {$\bigtimes$};
  \node at (2.5, 0.5) {$\bigtimes$};
  \node at (2,-0.5) {$D_0$};
  \draw (0,4)--(0,0)--(4,0);
  \draw (1,3)--(2,3)--(2,2)--(1,2)--(1,3);
  \draw (2,2)--(2,1)--(1,1)--(1,2)--(0,2);
  \draw (0,1)--(1,1);
  \draw (2,0)--(2,1)--(3,1)--(3,0);
\end{tikzpicture}\quad\quad\begin{tikzpicture}[scale=0.65]
  \node at (1.5, 2.5) {$\bigtimes$};
  \node at (0.5, 1.5) {$\bigtimes$};
  \node at (1.5, 0.5) {$\bigtimes$};
  \node at (2.5, 0.5) {$\bigtimes$};
    \node at (2,-0.5) {$D_1$};
  \draw (0,4)--(0,0)--(4,0);
  \draw (1,3)--(2,3)--(2,2)--(1,2)--(1,3);
  \draw (0,1)--(1,1);
  \draw (2,0)--(2,1)--(3,1)--(3,0);
  \draw (0,2)--(1,2)--(1,0);
  \draw (1,1)--(2,1);
\end{tikzpicture}\quad\quad\begin{tikzpicture}[scale=0.65]
  \node at (1.5, 2.5) {$\bigtimes$};
  \node at (0.5, 0.5) {$\bigtimes$};
  \node at (1.5, 0.5) {$\bigtimes$};
  \node at (2.5, 0.5) {$\bigtimes$};
    \node at (2,-0.5) {$D_2$};
  \draw (0,4)--(0,0)--(4,0);
  \draw (1,3)--(2,3)--(2,2)--(1,2)--(1,3);
  \draw (0,1)--(1,1);
  \draw (2,0)--(2,1)--(3,1)--(3,0);
  \draw (1,0)--(1,1)--(2,1);
\end{tikzpicture}\quad\quad\begin{tikzpicture}[scale=0.65]
  \node at (1.5, 1.5) {$\bigtimes$};
  \node at (0.5, 1.5) {$\bigtimes$};
  \node at (1.5, 0.5) {$\bigtimes$};
  \node at (2.5, 0.5) {$\bigtimes$};
    \node at (2,-0.5) {$D_3$};
  \draw (0,4)--(0,0)--(4,0);
  \draw (0,1)--(1,1);
  \draw (2,0)--(2,1)--(3,1)--(3,0);
  \draw (0,2)--(1,2)--(1,0);
  \draw (1,1)--(2,1);
  \draw (1,2)--(2,2)--(2,1);
\end{tikzpicture}\quad\quad\begin{tikzpicture}[scale=0.65]
  \node at (1.5, 1.5) {$\bigtimes$};
  \node at (0.5, 0.5) {$\bigtimes$};
  \node at (1.5, 0.5) {$\bigtimes$};
  \node at (2.5, 0.5) {$\bigtimes$};
    \node at (2,-0.5) {$D_4$};
  \draw (0,4)--(0,0)--(4,0);
  \draw (1,2)--(2,2)--(2,1)--(1,1)--(1,2);
  \draw (0,1)--(1,1);
  \draw (2,0)--(2,1)--(3,1)--(3,0);
  \draw (1,0)--(1,1)--(2,1);
\end{tikzpicture}$$
    \caption{The set $KD(D_0)$ for $D_0=\{(1,3),(2,1),(2,2),(3,2)\}$}
    \label{fig:KD}
\end{figure}
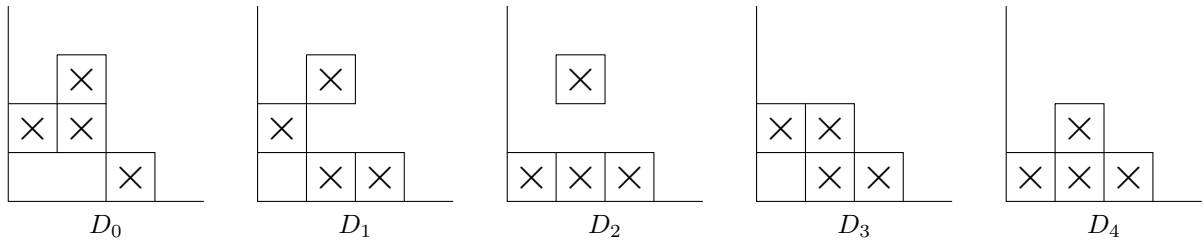
\end{example}

\subsection{Poset theory}\label{subsec:KPos}

In this section, we cover the needed preliminaries from the theory of posets. For more details, see~\cite{EC1}.

A \textbf{finite poset} $(\mathcal{P},\preceq)$ consists of a finite underlying set $\mathcal{P}$ along with a binary relation $\preceq$, also called an \textbf{order}, between the elements of $\mathcal{P}$ which is reflexive, anti-symmetric, and transitive. When no confusion will arise, we simply denote a poset $(\mathcal{P}, \preceq)$ by $\mathcal{P}$, and we let $\le$ denote the natural order on $\mathbb{Z}$. 

Let $\mathcal{P}$ be a finite poset and take $x,y\in\mathcal{P}$. 
\begin{itemize}
    \item We write $x\prec y$ to denote a \textbf{strict relation}; that is, if $x\preceq y$ and $x\neq y$.
    
    \item If $x\preceq y$, we set $[x,y]=\{z\in\mathcal{P}~|~x\preceq z\preceq y\}$ and treat $[x,y]$ as a poset with the order inherited from $\mathcal{P}$; that is, for $z_1,z_2\in [x,y]$, $z_1\prec_{[x,y]}z_2$ if and only if $z_1\prec_{\mathcal{P}}z_2$.
    
    \item We write $x\precdot y$ to denote a \textbf{covering relation}; that is, if $x\prec y$ and there exists no $z\in \mathcal{P}$ satisfying $x\prec z\prec y$.

    \item The \textbf{Hasse diagram} of $\mathcal{P}$ is the graph whose vertices correspond to elements of $\mathcal{P}$ and whose edges correspond to covering relations.

    \item We say $\mathcal{P}$ is \textbf{connected} if the Hasse diagram of $\mathcal{P}$ is a connected graph and say it is \textbf{disconnected} otherwise. 

    \item We say that $x\in\mathcal{P}$ is a \textbf{minimal element} (resp., \textbf{maximal element}) if there exists no $z\in\mathcal{P}$ such that $z\prec x$ (resp., $z\succ x$). If $\mathcal{P}$ has a unique minimal and maximal element, then we say that $\mathcal{P}$ is \textbf{bounded}.

    \item A totally ordered subset $C\subset\mathcal{P}$ is called a \textbf{chain}. The \textbf{length} of a chain is one less than the number of elements that it contains. We refer to a chain $C\subset\mathcal{P}$ as \textbf{maximal} if it is contained in no larger chain of $\mathcal{P}$.

    \item We say $\mathcal{P}$ is \textbf{ranked} if there exists a \textbf{rank function} $\rho:\mathcal{P}\to\mathbb{Z}_{\ge 0}$ such that 
\begin{enumerate}
    \item if $x\prec y$, then $\rho(x)<\rho(y)$, and
    \item if $x$ is covered by $y$, then $\rho(y)=\rho(x)+1$.
\end{enumerate}
    If a poset is ranked, then we can arrange its Hasse diagram into ``levels'' or ``ranks''.
\end{itemize}
In the case that $\mathcal{P}$ is bounded, an alternative definition of being ranked can be given in terms of its chains.
\begin{lemma}\label{lem:bdrank}
    Let $\mathcal{P}$ be a bounded poset. Then $\mathcal{P}$ is ranked if and only if all maximal chains of $\mathcal{P}$ have the same length.
\end{lemma}

Now, if $\mathcal{P}$ is a ranked poset and $x,y\in \mathcal{P}$ satisfy $x\prec y$, then $[x,y]$ is also a ranked poset. Consequently, considering Lemma~\ref{lem:bdrank}, we get the following.
\begin{lemma}\label{lem:ranksl}
    If $\mathcal{P}$ is ranked and $x,y\in \mathcal{P}$ with $x\prec y$, then all maximal chains of $[x,y]$ have the same length.
\end{lemma}

\begin{example}\label{ex:poset}
    Let $\mathcal{P}=\{1,2,3,4,5\}$ be the poset defined by the relations $1\prec 2\prec 3$ and $2\prec 4 \prec 5$. The Hasse diagram of $\mathcal{P}$ is illustrated in Figure~\ref{fig:Hasse}.

    \begin{figure}[H]
        \centering
        \begin{tikzpicture}
	\node (1) at (0, 0) [circle, draw = black, fill = black, inner sep = 0.5mm, label=left:{$1$}]{};
	\node (2) at (0, 1)[circle, draw = black, fill = black, inner sep = 0.5mm, label=left:{$2$}] {};
	\node (3) at (-0.75, 1.75) [circle, draw = black, fill = black, inner sep = 0.5mm, label=left:{$3$}] {};
	\node (4) at (0.75, 1.75) [circle, draw = black, fill = black, inner sep = 0.5mm, label=right:{$4$}] {};
 \node (5) at (0, 2.5) [circle, draw = black, fill = black, inner sep = 0.5mm, label=left:{$5$}] {};
    \draw (1)--(2);
    \draw (2)--(3);
    \draw (2)--(4);
    \draw (3)--(5)--(4);
\end{tikzpicture}
        \caption{Hasse diagram of $\mathcal{P}$}
        \label{fig:Hasse}
    \end{figure}
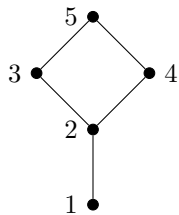
    Note that $\mathcal{P}$ is bounded with 1 and 5 being its unique minimum and maximal elements, respectively.
  Moreover, $\mathcal{P}$ is ranked with rank function $\rho$  given by $\rho(1)=0$, $\rho(2)=1$, $\rho(3)=\rho(4)=2$, and $\rho(5)=3$.
  Finally, $\mathcal{P}$ has exactly two maximal chains, $\{1,2,3,5\}$ and $\{1,2,4,5\}$, both of length 4, and so it also follows from Lemma~\ref{lem:bdrank} that $\mathcal{P}$ is ranked.
\end{example}

\subsection{Kohnert Posets}\label{sec:Kpos}
In this section, we focus on the posets arising from Kohnert diagrams. 

Let $D_0$ be a diagram. The set $KD(D_0)$ has the structure of a poset with the following order. Given $D_1,D_2\in KD(D_0)$, we say $D_2\prec D_1$ if $D_2$ can be obtained from $D_1$ by applying some sequence of Kohnert moves (see \cite{KP3,KP2}). We denote the corresponding poset on $KD(D_0)$ by $\mathcal{P}(D_0)$ and refer to it as the \textbf{Kohnert poset} associated with $D_0$. Moreover, we let $Min(D_0)$ denote the collection of minimal elements of $\mathcal{P}(D_0)$, which are exactly those diagrams that are fixed by all Kohnert moves.

\begin{remark}\label{rem:cover}
    Note that the cover relation $D_2\precdot D_1$ implies that $D_2$ can only be obtained from $D_1$ by applying a single nontrivial Kohnert move. However, the converse is not necessarily true; that is, if $D_2$ is obtained from $D_1$ by applying a single nontrivial Kohnert move, we cannot conclude that $D_2\precdot D_1$ since there could be another way to obtain $D_2$ from $D_1$ by applying more than one nontrivial Kohnert move. For example, let $D_0$ be the diagram of Figure~\ref{fig:diagram} and $D_3$ the diagram obtained from $D_0$ by applying a single Kohnert move at row 3. Then $D_3\prec D_0$ is not a cover relation as can be seen in Example~\ref{ex:KDD} where it is shown that $D_3$ can be formed from $D_0$ by applying two nontrivial Kohnert moves.
\end{remark}

For the purpose of defining rank functions on Kohnert posets in the sections that follow, given a diagram $D_0$, we define $$b(D_0)=\min\{rowsum(T)~|~T\in Min(D_0)\}.$$

\begin{remark}
    In Sections~\ref{sec:obpc},~\ref{sec:2r}, and~\ref{sec:cd}, we will see that for certain diagrams $D_0$, the function $rk(D)=rowsum(D)-b(D_0)$ for $D\in KD(D_0)$ defines a rank function on $\mathcal{P}(D_0)$. The purpose of the constant $b(D_0)$ in this formula is that it forces certain minimal elements of $\mathcal{P}(D_0)$ to have a rank of 0 and all other diagrams to have positive rank.
\end{remark}

\section{General results}\label{sec:sres}

In this section, we establish a sufficient condition for a Kohnert poset to be bounded and two necessary conditions for a Kohnert poset to be ranked.

We start by noticing that since Kohnert posets $\mathcal{P}(D_0)$ have a unique maximal element, i.e., the initial diagram $D_0$, we have the following.

\begin{lemma}\label{lem:bmum}
    Let $D_0$ be a diagram. Then $\mathcal{P}(D_0)$ is bounded if and only if $|Min(D_0)|=1$.
\end{lemma}

The following result shows that if the number of cells in the columns of a diagram weakly decreases from left to right, then the associated Kohnert poset is bounded.

\begin{prop}\label{prop:nummingen}
Let $D_0$ be a diagram with $b_i$ cells in column $i$ for $1\leq i \leq n$. If $b_1\ge b_2\ge \cdots\ge b_n$, then $\mathcal{P}(D_0)$ is bounded.
\end{prop}
\begin{proof}

Let $D_{min}$ be the diagram obtained by bottom justifying the cells of $D_0$; that is, 
$$D_{min}=\bigcup_{i=1}^n\{(j,c_i)~|~1\le j\le b_i\}.$$ 
We show that $D_{min}\preceq \tilde{D}$ for every $\tilde{D}\in KD(D_0)$. Consequently, $Min(D_0)=\{D_{min}\}$ and the result follows from Lemma~\ref{lem:bmum}.

Given $D\in KD(D_0)$, $D_{min}$ is obtained from $D$ by the following general procedure:
\begin{itemize}
    \item[] Step 1: Apply Kohnert moves to bottom-justify the cells of column $n$ in $D$, if necessary.
    \item[] Step $j$: Apply Kohnert moves to bottom justify the cells of column $n-j+1$ in $D$, if necessary.
\end{itemize}
The output of the procedure above is the diagram $D_{min}$ and we are left with showing that each step is achievable. Now, Step 1 refers to the cells in the last column of $D$. Since there are no cells to the right of column $n$, it is always possible to apply Kohnert moves to cells in that column, i.e., Step 1 is achievable. For Step $j>1$, let $D_{j-1}$ be the diagram obtained after completing step $j-1$ for $1\leq j\leq n$. Then there are three possibilities:
\begin{enumerate}
    \item[1)] In $D_{j-1}$, the cells in column $n-j+1$ are already bottom-justified. 
    \item[2)] In $D_{j-1}$, the cells in column $n-j+1$ are not bottom-justified, but there are cells in rows 1 up to $b_{j-1}$. In this case, $b_j>b_{j+1}$ and there are no cells to the right of the cells in rows higher than $b_{j-1}$. Therefore, there exists a sequence of Kohnert moves at those rows that bottom justifies the cells in column $n-j+1$. 
    \item[3)] In $D_{j-1}$, the cells in column $n-j+1$ are not bottom-justified and there are empty cells in rows 1 up to $b_{j-1}$. Again, there are no cells to the right of the cells in rows higher than $b_{j-1}$, and, consequently, there exists a sequence of Kohnert moves at those rows that bottom justifies the cells in column $n-j+1$. 
\end{enumerate}
Thus, Step $j$ is achievable and the result follows. 
\end{proof}

Moving to rankedness, we have two necessary conditions for a Kohnert poset to be ranked. To aid in proving these conditions, we require the following result which essentially states that if two diagrams $D_1$ and $D_2$ with $D_2\prec D_1$ match in a column $c$ above (resp., below) some row $r$, then all the diagrams between them in the poset must also have the same configuration of cells in column $c$ above (resp., below) row $r$.

\begin{lemma}\label{lem:rankhelp}
   Let $D_0$ be a diagram and consider $D_1,D_2\in KD(D_0)$ satisfying $D_2\prec D_1$.
    \begin{enumerate}
        \item[\textup{(a)}] For $c,r\in\mathbb{N}$, if
        $$D_1\cap\{(\tilde{r},c)~|~\tilde{r}\ge r\}=D_2\cap\{(\tilde{r},c)~|~\tilde{r}\ge r\},$$ 
        then for all $D\in[D_2,D_1]$, $$D\cap\{(\tilde{r},c)~|~\tilde{r}\ge r\}=D_1\cap\{(\tilde{r},c)~|~\tilde{r}\ge r\}.$$
        \item[\textup{(b)}] Similarly, for $c,r\in\mathbb{N}$, if
        $$D_1\cap\{(\tilde{r},c)~|~\tilde{r}\le r\}=D_2\cap\{(\tilde{r},c)~|~\tilde{r}\le r\},$$ 
        then for all $D\in[D_2,D_1]$, $$D\cap\{(\tilde{r},c)~|~\tilde{r}\le r\}=D_1\cap\{(\tilde{r},c)~|~\tilde{r}\le r\}.$$
    \end{enumerate} 
\end{lemma}
\begin{proof}
We prove (a), as (b) follows via similar reasoning. Assume that for integers $c,r\ge 1$ we have $$S=D_1\cap\{(\tilde{r},c)~|~\tilde{r}\ge r\}=D_2\cap\{(\tilde{r},c)~|~\tilde{r}\ge r\}.$$ 
Consider $D\in[D_2,D_1]$ and suppose that 
\begin{equation}\label{eq:rankhelp}
    D\cap\{(\tilde{r},c)~|~\tilde{r}\ge r\}\neq S.
\end{equation}
Since $D\prec D_1$, $D$ can be obtained from $D_1$ by applying a sequence of Kohnert moves. Considering (\ref{eq:rankhelp}), there must exist $\hat{r}$ such that $\hat{r}\ge r$, $(\hat{r},c)\in D_1$, and the cell $(\hat{r},c)$ of $D_1$ is moved to a lower row by a Kohnert move in forming $D$. Note that this implies that $D$ has at least one less cell in column $c$ weakly above row $\hat{r}$ than both $D_1$ and $D_2$. As applying Kohnert moves cannot increase the number of cells in a fixed column above a given row, it follows that $D_2\nprec D$, which is a contradiction. Thus, $D\cap\{(\tilde{r},c)~|~\tilde{r}\ge r\}= S$ and the result follows.
\end{proof}

The following results establish two necessary conditions for a Kohnert poset to be ranked. In both cases, we identify a family of subdiagrams that are associated with intervals containing two maximal chains of different lengths. For Theorem~\ref{thm:ranked1}, Figures~\ref{fig:rank1c1},~\ref{fig:rank1c2}, and~\ref{fig:rank1c3} illustrate those subdiagrams, and for Theorem~\ref{thm:ranked2}, an example is illustrated in Figure~\ref{fig:ranklem}(b).
\begin{theorem}\label{thm:ranked1}
Let $D_0$ be a diagram. Suppose there exists a diagram $D\in KD(D_0)$ such that for some $r,r',r_1,r_2,c,c'\in\mathbb{N}$ satisfying $1\le r_1,r_2< r<r'$ and $1\le c<c'$ we have
\begin{enumerate}
        \item[\textup{(i)}] $(\tilde{r},c),(r,c')\in D$ for $r\le \tilde{r}\le r'$, i.e., the cells in rows $r$ through $r'$ in column $c$ are in the diagram as well as the cell in row $r$ and column $c'$;
        \item[\textup{(ii)}] $(r',\tilde{c})\notin D$ for $\tilde{c}>c$, i.e., the cell in row $r'$ and column $c$ is the rightmost cell in its row;
        \item[\textup{(iii)}] for each $\tilde{r}$ satisfying $r<\tilde{r}<r'$ there exists $\tilde{c}>c$ such that $(\tilde{r},\tilde{c})\in D$, i.e., in every row strictly between row $r$ and row $r'$ there is a cell in a column to the right of column $c$;
        \item[\textup{(vi)}] $(r,\tilde{c})\notin D$ for $c<\tilde{c}<c'$ and $\tilde{c}>c'$, i.e., there are no cells in row $r$ strictly between columns $c$ and $c'$ and the cell in column $c$ of row $r$ is the rightmost cell in its row; and
        \item[\textup{(v)}] $(r_1,c),(r_2,c')\notin D$, i.e., there are empty positions below $(r,c)$ and $(r,c')$. 
    \end{enumerate}
Then $\mathcal{P}(D_0)$ is not ranked.
\end{theorem}

\begin{proof}
We show that there exist diagrams $D_{min},D_{max}\in KD(D_0)$ such that $D_{min}\prec D_{max}\preceq D$ and the interval $[D_{min},D_{max}]$ contains two maximal chains of different lengths, one of length 2 and one of length 3. Applying Lemma~\ref{lem:ranksl}, the result will then follow. Note that here, $D_{min}$ and $D_{max}$ are the unique minimal and maximal elements, respectively, of an interval of $\mathcal{P}(D_0)$, but are not necessarily minimal and maximal, respectively, in $\mathcal{P}(D_0)$. Without loss of generality, we assume that $r_1,r_2$ are both the largest values such that $r_1,r_2<r$ and $(r_1,c),(r_2,c')\notin D$. 

We split the proof into three cases depending on the locations of columns $\tilde{c}>c$ which may contain ``movable cells" in rows $\tilde{r}<r$; that is, the locations of columns $\tilde{c}>c$ which may contain cells in rows $\tilde{r}<r$ whose positions can be affected by the application of some sequence of Kohnert moves.
\bigskip

\noindent
\textbf{Case 1:} Suppose that there is a column $\hat{c}$ to the right of column $c'$ which contains a cell in every row strictly below row $r$; that is, suppose there exists $\hat{c}>c'$ such that $(\tilde{r},\hat{c})\in D$ for $1\le \tilde{r}< r$. We depict the general form of such a diagram in Figure~\ref{fig:rank1c1}; see Remark~\ref{rem:diagconv} for a discussion of the pertinent diagram conventions.

\begin{figure}[H]
    \centering
    $$\scalebox{0.9}{\begin{tikzpicture}[scale=0.65]
    \node at (8.5,2.5) {$\bigtimes$};
    \node at (8.5,1.5) {$\vdots$};
    \node at (8.5,0.5) {$\bigtimes$};
    \node at (8.5, -0.5) {$\hat{c}$};
    \node at (5.5,3.5) {$\bigtimes$};
    \node at (1.5,6.5) {$\bigtimes$};
    \node at (-0.5,6.5) {$r'$};
    \node at (-0.5,3.5) {$r$};
    \node at (1.5,5) {$\vdots$};
    \node at (1.5,3.5) {$\bigtimes$};
    \node at (5.5, -0.5) {$c'$};
    \node at (1.5, -0.5) {$c$};
    \node at (6, 5) {(At most one cell per row)};
    \node at (4.5, 1.5) {$\binom{\text{At least one empty position}}{\text{in columns $c$ and $c'$}}$};
    \draw (0,8.5)--(0,0)--(11,0);
    \draw (8,0)--(8,3)--(9,3)--(9,0);
    \draw (8,1)--(9,1);
    \draw (8,2)--(9,2);
    \draw (5,3)--(6,3)--(6,4)--(5,4)--(5,3);
    \filldraw[draw=lightgray,fill=lightgray] (2,3) rectangle (5,4);
    \filldraw[draw=lightgray,fill=lightgray] (6,3) rectangle (10.5,4);
    \filldraw[draw=lightgray,fill=lightgray] (2,6) rectangle (10.5,7);
    \draw (1,3)--(2,3)--(2,7)--(1,7)--(1,3);
    \draw (1,4)--(2,4);
    \draw (1,6)--(2,6);
    \path[pattern=north west lines] (9,0)--(9,3) -- (10.5,3) -- (10.5,0) -- cycle;
    \path[pattern=north west lines] (0,0)--(1,0) -- (1,7) -- (10.5,7) -- (10.5,8) -- (0,8) -- cycle;
\end{tikzpicture}}$$
    \caption{Theorem~\ref{thm:ranked1} Case 1}
    \label{fig:rank1c1}
\end{figure}
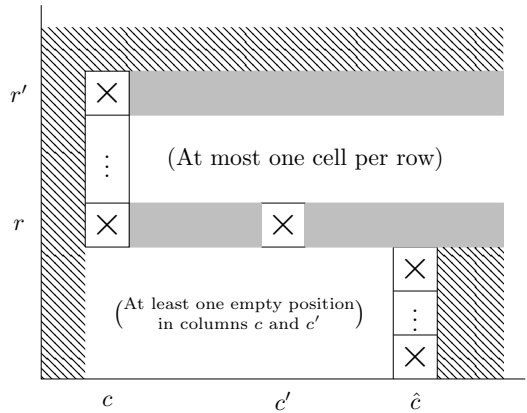

\noindent
In this case, we take $D_{max}=D$ and $D_{min}$ to be the diagram obtained from $D$ by removing cells $(r',c)$ and $(r,c')$ and adding the cells $(r_1,c)$ and $(r_2,c')$ .Consider the following two chains from $D_{min}$ to $D_{max}$, both defined by the sequences of Kohnert moves applied to form $D_{min}$ from $D_{max}$.
\begin{itemize}
    \item[1)] Form the chain $C_1$ by applying two Kohnert moves at row $r$ followed by another at row $r'$, i.e.,
    $$ 
    C_1: \qquad D_{max} \succ D_1:= D_{max}\ldownarrow^{(r,c')}_{(r_2,c')} \succ 
    D_2:= D_1\ldownarrow^{(r,c)}_{(r_1,c)} \succ D_{min} = D_2 \ldownarrow^{(r',c)}_{(r,c)}. 
    $$

    \item[2)] Form the chain $C_2$ by applying one Kohnert move at row $r'$ followed by another at row $r$, i.e.,
      $$ 
    C_2: \qquad D_{max} \succ D_3:= D_{max}\ldownarrow^{(r',c)}_{(r_1,c)} \succ D_{min} = D_3 \ldownarrow^{(r,c')}_{(r_2,c')}. 
    $$
\end{itemize}
We claim that the chains $C_1$ and $C_2$ are maximal in $[D_{min},D_{max}]$. In order to establish the claim, we need to show that the relations defining $C_1$ and $C_2$ are all cover relations. As the arguments for each relation are similar, we prove that $D_3\precdot D_{max}$ and leave the other relations to the reader. Assume that there exists a diagram $D^*$ such that $D_3\prec D^*\prec D_{max}$, where $D^*$ is formed from $D_{max}$ by applying a Kohnert move at row $r^*\neq r$. As $D_{max}$ and $D_3$ only differ in positions $(r',c)$ and $(r_1,c)$, applying Lemma~\ref{lem:rankhelp}, it follows that $r_1<r^*<r'$. Now, note that property (iii) of $D_0$ carries over to $D_{max}$ and $D_3$. Consequently, if $r_1<r^*<r'$, then $D^*$ must differ from $D_{max}$ in some column $\tilde{c}>c$; that is, applying Lemma~\ref{lem:rankhelp} once again, we may conclude that $D^*\notin [D_3,D_{max}]$. Thus, $D_3\precdot D_{max}$. Now, since $[D_{min},D_{max}]$ contains two maximal chains of different lengths, as discussed above, it follows that $\mathcal{P}(D_0)$ is not ranked.
\bigskip

\noindent
\textbf{Case 2:} Suppose that there is a column $\hat{c}$ between columns $c$ and $c'$ which contains a cell in every row strictly below row $r$ and that every column to the right of column $\hat{c}$ contains at least one empty position below row $r$; that is, there exists $c<\hat{c}<c'$ such that $(\hat{r},\hat{c})\in D$ for $1\le \hat{r}< r$, and for all $c^\dagger >\hat{c}$, there exists $r^\dagger<r$ such that $(r^\dagger, c^\dagger)\notin D^*$. We depict the general form of such a diagram in Figure~\ref{fig:rank1c2}.

\begin{figure}[H]
    \centering
    $$\scalebox{0.9}{\begin{tikzpicture}[scale=0.65]
    \node at (8.5,2.5) {$\bigtimes$};
    \node at (8.5,1.5) {$\vdots$};
    \node at (8.5,0.5) {$\bigtimes$};
    \node at (8.5, -0.5) {$\hat{c}$};
    \node at (11.5,3.5) {$\bigtimes$};
    \node at (1.5,6.5) {$\bigtimes$};
    \node at (-0.5,6.5) {$r'$};
    \node at (-0.5,3.5) {$r$};
    \node at (1.5,5) {$\vdots$};
    \node at (1.5,3.5) {$\bigtimes$};
    \node at (11.5, -0.5) {$c'$};
    \node at (1.5, -0.5) {$c$};
    \node at (8, 5) {(At most one cell per row)};
    \node at (4.5, 1.5) {$\binom{\text{At least one empty position}}{\text{in column $c$}}$};
    \node at (12,1.5) {$\binom{\text{At least one empty}}{\text{position per column}}$};
    \draw (0,8.5)--(0,0)--(14,0);
    \draw (8,0)--(8,3)--(9,3)--(9,0);
    \draw (8,1)--(9,1);
    \draw (8,2)--(9,2);
    \draw (11,3)--(12,3)--(12,4)--(11,4)--(11,3);
    \filldraw[draw=lightgray,fill=lightgray] (2,3) rectangle (11,4);
    \filldraw[draw=lightgray,fill=lightgray] (2,6) rectangle (13.5,7);
    \filldraw[draw=lightgray,fill=lightgray] (12,3) rectangle (13.5,4);
    \draw (1,3)--(2,3)--(2,7)--(1,7)--(1,3);
    \draw (1,4)--(2,4);
    \draw (1,6)--(2,6);
    \path[pattern=north west lines] (0,0)--(1,0) -- (1,7) -- (13.5,7) -- (13.5,8) -- (0,8) -- cycle;
\end{tikzpicture}}$$
    \caption{Theorem~\ref{thm:ranked1} Case 2}
    \label{fig:rank1c2}
\end{figure}
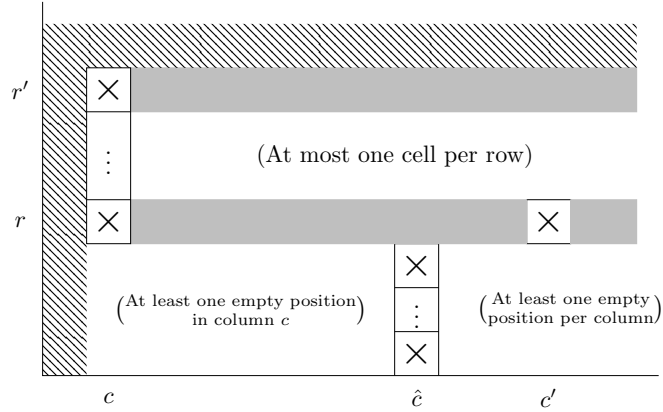

\noindent
In this case, letting $$k=|\{\tilde{c}~|~\tilde{c}\ge c',~(r-1,c')\in D\}|,$$ we take $D_{max}$ to be the diagram obtained from $D$ by applying $k$ Kohnert moves at row $r-1$; that is, $D_{max}$ is the diagram formed from $D$ by applying enough Kohnert moves at row $r-1$ to drop every cell in row $r-1$ and weakly to the right of column $c'$ to a lower row. We then take $D_{min}$ to be the diagram obtained from $D_{max}$ by removing the cells $(r',c)$ and $(r,c')$, and adding the cells $(r_1,c)$ and $(r-1,c')$.  Consider the following two chains from $D_{min}$ to $D_{max}$, both defined by the sequences of Kohnert moves applied to form $D_{min}$ from $D_{max}$.
\begin{itemize}
    \item[1)] Form the chain $C_1$ by applying two Kohnert moves at row $r$ followed by another at row $r'$, i.e.,
    $$ 
    C_1: \qquad D_{max} \succ D_1:= D_{max}\ldownarrow^{(r,c')}_{(r-1,c')} \succ 
    D_2:= D_1\ldownarrow^{(r,c)}_{(r_1,c)} \succ D_{min} = D_2 \ldownarrow^{(r',c)}_{(r,c)}. 
    $$

    \item[2)] Form the chain $C_2$ by applying one Kohnert move at row $r'$ followed by another at row $r$, i.e.,
      $$ 
    C_2: \qquad D_{max} \succ D_3:= D_{max}\ldownarrow^{(r',c)}_{(r_1,c)} \succ D_{min} = D_3 \ldownarrow^{(r,c')}_{(r-1,c')}. 
    $$
\end{itemize}
Using similar reasoning to that found in Case 1, it follows that $C_1$ and $C_2$ are maximal chains of $[D_{min},D_{max}]$ and, consequently, $\mathcal{P}(D_0)$ is not ranked.
\bigskip

\noindent
\textbf{Case 3:} Suppose that in every column weakly to the right of column $c$ there is an empty position below row $r$; that is, for all $\tilde{c}\ge c$, there exists $\tilde{r}<r$ such that $(\tilde{r},\tilde{c})\notin D$. We depict the general form of such a diagram in Figure~\ref{fig:rank1c3}.

\begin{figure}[H]
    \centering
    $$\scalebox{0.9}{\begin{tikzpicture}[scale=0.65]
    \node at (8.5, -0.5) {$c'$};
    \node at (8.5,3.5) {$\bigtimes$};
    \node at (1.5,6.5) {$\bigtimes$};
    \node at (-0.5,6.5) {$r'$};
    \node at (-0.5,3.5) {$r$};
    \node at (1.5,5) {$\vdots$};
    \node at (1.5,3.5) {$\bigtimes$};
    \node at (1.5, -0.5) {$c$};
    \node at (6, 5) {(At most one cell per row)};
    \node at (5.5, 1.5) {$\binom{\text{At least one empty position}}{\text{in columns $\tilde{c}\ge c$}}$};
    \draw (0,8.5)--(0,0)--(11,0);
    \draw (8,3)--(9,3)--(9,4)--(8,4)--(8,3);
    \filldraw[draw=lightgray,fill=lightgray] (2,3) rectangle (8,4);
    \filldraw[draw=lightgray,fill=lightgray] (9,3) rectangle (10.5,4);
    \filldraw[draw=lightgray,fill=lightgray] (2,6) rectangle (10.5,7);
    \draw (1,3)--(2,3)--(2,7)--(1,7)--(1,3);
    \draw (1,4)--(2,4);
    \draw (1,6)--(2,6);
    \path[pattern=north west lines] (0,0)--(1,0) -- (1,7) -- (10.5,7) -- (10.5,8) -- (0,8) -- cycle;
\end{tikzpicture}}$$
    \caption{Theorem~\ref{thm:ranked1} Case 3}
    \label{fig:rank1c3}
\end{figure}
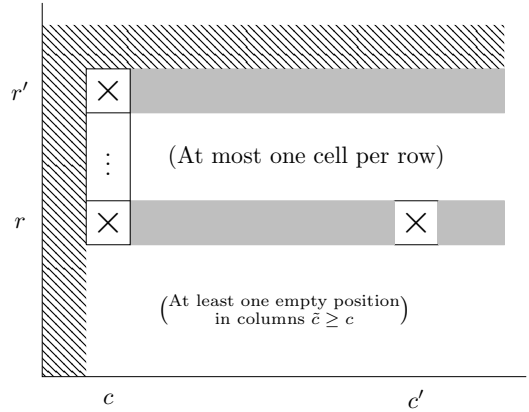

\noindent
In this case, letting $$k=\{\tilde{c}~|~\tilde{c}\ge c,~(r-1,\tilde{c})\in D\},$$ we take $D_{max}$ to be the diagram obtained from $D$ by applying $k$ Kohnert moves at row $r-1$; that is, $D_{max}$ is the diagram formed from $D$ by applying enough Kohnert moves at row $r-1$ to drop every cell in row $r-1$ and weakly to the right of column $c$ to a lower row. We then take $D_{min}$ to be the diagram obtained from $D_{max}$ by removing the cells $(r',c)$ and $(r,c')$, and adding the cells $(r-1,c)$ and $(r-1,c')$. Consider the following two chains from $D_{min}$ to $D_{max}$, both defined by the sequences of Kohnert moves applied to form $D_{min}$ from $D_{max}$.
\begin{itemize}
    \item[1)] Form the chain $C_1$ by applying two Kohnert moves at row $r$ followed by another at row $r'$, i.e.,
    $$ 
    C_1: \qquad D_{max} \succ D_1:= D_{max}\ldownarrow^{(r,c')}_{(r-1,c')} \succ 
    D_2:= D_1\ldownarrow^{(r,c)}_{(r-1,c)} \succ D_{min} = D_2 \ldownarrow^{(r',c)}_{(r,c)}. 
    $$

    \item[2)] Form the chain $C_2$ by applying one Kohnert move at row $r'$ followed by another at row $r$, i.e.,
      $$ 
    C_2: \qquad D_{max} \succ D_3:= D_{max}\ldownarrow^{(r',c)}_{(r-1,c)} \succ D_{min} = D_3 \ldownarrow^{(r,c')}_{(r-1,c')}. 
    $$
\end{itemize}
Using similar reasoning to that found in Case 1, it follows that $C_1$ and $C_2$ are maximal chains of $[D_{min},D_{max}]$ and, consequently, $\mathcal{P}(D_0)$ is not ranked.
\end{proof}

\begin{theorem}\label{thm:ranked2}
Let $D_0$ be a diagram. Suppose that there exists $D\in KD(D_0)$ such that for some $r,r',c,c'\in\mathbb{N}$ satisfying $1\le c< c'$ and $1\le r<r'-1$ we have
\begin{enumerate}
    \item[\textup{(i)}] $(r',c')\in D$ and $(r', \tilde{c})\notin D$ for $\tilde{c}>c'$, i.e., the cell in row $r'$ and column $c'$ is the rightmost cell in its row; 
    \item[\textup{(ii)}] $(\tilde{r},c)\in D$ for $r < \tilde{r}\le r'$, i.e., the cells between rows $r$ and $r'$ in column $c$ are in the diagram as well as the cell in row $r'$ and column $c$;
    \item[\textup{(iii)}] $(r,c)\notin D$, 
    \item[\textup{(iv)}] for each $\tilde{r}$ satisfying $r<\tilde{r}<r'-1$, there exists $\tilde{c}>c$ such that $(\tilde{r},\tilde{c})\in D$, i.e., in every row strictly between rows $r$ and $r'-1$ there is a cell in a column strictly to the right of column $c$; and
    \item[\textup{(v)}] for each $\tilde{c}>c$, there exists $\tilde{r}$ such that $1\le \tilde{r}<r'$ and $(\tilde{r},\tilde{c})\notin D$, i.e., in every column to the right of column $c$ there is an empty position strictly below row $r'$. 
\end{enumerate}
Then $\mathcal{P}(D_0)$ is not ranked.
\end{theorem}

\begin{proof}
We follow the same strategy as in the proof of Theorem~\ref{thm:ranked1}: We find two diagrams $D_{max},D_{min}\in KD(D_0)$ such that $D_{min}\prec D_{max}\preceq D$ and the interval $[D_{min},D_{max}]$ contains two maximal chains of different lengths. The result then follows as a consequence of Lemma~\ref{lem:ranksl}. As in the proof of Theorem~\ref{thm:ranked1}, $D_{min}$ and $D_{max}$ are the unique minimal and maximal elements, respectively, of an interval of $\mathcal{P}(D_0)$, but are not necessarily minimal and maximal, respectively, in $\mathcal{P}(D_0)$.

Letting $k=|\{\tilde{c}~|~\tilde{c}>c,~(r'-1,\tilde{c})\in D\}|$, we take $D_{max}$ to be the diagram obtained from $D$ by applying $k$ Kohnert moves at row $r'-1$; that is, $D_{max}$ is the diagram formed from $D$ by applying enough Kohnert moves at row $r'-1$ to drop every cell in row $r'-1$ and strictly to the right of column $c$ to a lower row. Now, $D_{min}$ and the distinct lengths of the two maximal chains depend on the cells of $D_{max}$ contained in row $r'$ and columns $c$ through $c'$. For this reason, we define the set $C$, say of size $m$, as
$$C= \left\{\tilde{c}~|~ c\leq \tilde{c} \leq c'\ \text{ and } \ (r',\tilde{c})\in D_{max} \right\} = \left\{ c=c_0 < c_1 < \cdots < c_{m-1} = c'\right\}.$$
Then, $D_{min}$ is the diagram obtained from $D_{max}$ by applying $m$ Kohnert moves at row $r'$; that is, $D_{min}$ is the diagram obtained from $D_{max}$ by moving the rightmost $m-1$ cells in row $r'$ down to row $r'-1$, and moving the $m^{th}$ cell from right to left in row $r'$ down to row $r$.  Consider the following two chains from $D_{min}$ to $D_{max}$, both defined by the sequences of Kohnert moves applied to form $D_{min}$ from $D_{max}$.
\begin{itemize}
    \item[1)] Form the chain $C_1$ by applying $m$ Kohnert moves at row $r'$, i.e.,
    $$ 
    C_1: \qquad D_{max} \succ D_1  \succ 
    D_2 \succ \cdots \succ D_{m-1} \succ D_{min},
    $$
    where 
    $D_1 := D_{max} \ldownarrow^{(r',c_{m-1})}_{(r'-1,c_{m-1})}$, $D_{i+1}:= D_{i} \ldownarrow^{(r',c_{m-i-1})}_{(r'-1,c_{m-i-1})}$ for $1\leq i \leq m-2$, and $D_{min} = D_{m-1} \ldownarrow^{(r',c_0)}_{(r,c_0)}$.

    \item[2)] Form the chain $C_2$ by applying one Kohnert move at row $r'-1$ followed by $m$ Kohnert moves at row $r'$, i.e.,
      $$ 
    C_2:\qquad D_{max} \succ D_1'  \succ 
    D_2' \succ \cdots \succ D_{m}' \succ D_{min},
    $$
    where
 $D_1' := D_{max} \ldownarrow^{(r'-1,c_0)}_{(r,c_0)}$, $D_{i+1}':= D_{i}' \ldownarrow^{(r',c_{m-i})}_{(r'-1,c_{m-i})}$ for $1\leq i \leq m-1$, and $D_{min} = D_{m}'\ldownarrow^{(r',c_0)}_{(r'-1,c_0)}$.
    
\end{itemize}
Using similar reasoning to that found in Case 1 of the proof of Theorem~\ref{thm:ranked1}, it follows that $C_1$ and $C_2$ are maximal chains of $[D_{min},D_{max}]$. As $C_1$ has length $m$ and $C_2$ length $m+1$, it follows that $\mathcal{P}(D_0)$ is not ranked.
\end{proof}

Ongoing, we will not require the full strength of Theorems~\ref{thm:ranked1} and~\ref{thm:ranked2}. Instead, we will make use of the following corollary where (a) corresponds to a special case of Theorem~\ref{thm:ranked1} and (b) a special case of Theorem~\ref{thm:ranked2}.

\begin{corollary}\label{cor:ranked}
Let $D_0$ be a diagram. Suppose that there exists $D\in KD(D_0)$ such that for some $r^*,c_1,c_2\in\mathbb{N}$ satisfying $1\le c_1<c_2$ we have either
\begin{enumerate}
    \item[\textup(a\textup)] $D$ satisfies
    \begin{itemize}
        \item[\textup{(i)}] $(r^*+1,c_1),(r^*+2,c_1),(r^*+1,c_2)\in D$,
        \item[\textup{(ii)}] $(r^*+2,\tilde{c})\notin D^*$ for $\tilde{c}>c_1$,
        \item[\textup{(iii)}] $(r^*+1,\tilde{c})\notin D$ for $c_1<\tilde{c}<c_2$ and $\tilde{c}>c_2$, and
        \item[\textup{(iv)}] $(r^*,c_1),(r^*,c_2)\notin D$; or
    \end{itemize}
    \item[\textup(b\textup)] $D$ satisfies
    \begin{itemize}
        \item[\textup{(i)}] $(r^*+1,c_1),(r^*+2,c_1),(r^*,c_2),(r^*+2,c_2)\in D$,
        \item[\textup{(ii)}] $(r^*+2,\tilde{c})\notin D$ for $\tilde{c}>c_2$,
        \item[\textup{(iii)}] $(r^*+1,\tilde{c})\notin D$ for $\tilde{c}>c_1$, and
        \item[\textup{(iv)}] $(r^*,c_1)\notin D$.
    \end{itemize}
     Then $\mathcal{P}(D_0)$ is not ranked. 
\end{enumerate}
\end{corollary}
In Example~\ref{ex:rankconfig}, we depict the general forms of the diagrams described in Corollary~\ref{cor:ranked}.
\begin{proof}
(a) corresponds to Theorem~\ref{thm:ranked1} with $r_1=r_2=r^*$, $r=r^*+1$, $r'=r^*+2$, $c=c_1$, and $c'=c_2$. (b) corresponds to Theorem~\ref{thm:ranked2} with $r=r^*$, $r'=r^*+2$, $c=c_1$, and $c'=c_2$. 
\end{proof}

\begin{example}\label{ex:rankconfig}
    In Figure~\ref{fig:ranklem}, \textup{(a)} illustrates the subdiagram described in Corollary~\ref{cor:ranked} \textup{(a)}, and \textup{(b)} the subdiagram described in Corollary~\ref{cor:ranked} \textup{(b)}.
    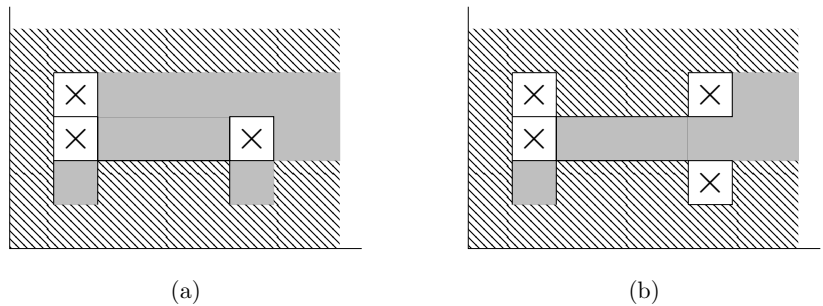
\begin{figure}[H]
    \centering
    $$\scalebox{0.9}{\begin{tikzpicture}[scale=0.65]
    \filldraw[fill=lightgray] (1,1) rectangle (4,2);
    \filldraw[draw=lightgray,fill=lightgray] (1,2) rectangle (4,3);
    \filldraw[draw=lightgray,fill=lightgray] (4,2) rectangle (5,3);
    \filldraw[draw=lightgray,fill=lightgray] (0,0) rectangle (1,1);
    \filldraw[draw=lightgray,fill=lightgray] (4,0) rectangle (5,1);
    \filldraw[draw=lightgray,fill=lightgray] (5,1) rectangle (6.5,3);
    \node at (0.5, 1.5) {$\bigtimes$};
    \node at (0.5, 2.5) {$\bigtimes$};
  \node at (4.5, 1.5) {$\bigtimes$};
  \draw (-1,4.5)--(-1,-1)--(7,-1);
  \draw (0,3)--(1,3)--(1,2);
  \draw (0,1)--(1,1)--(1,2)--(0,2);
  \draw (0,3)--(0,0);
  \draw (4,1)--(5,1)--(5,2)--(4,2)--(4,1);
  \draw (1,0)--(1,1)--(2,1);
  \draw (2,1)--(3,1);
    \draw (3,1)--(4,1)--(4,0);
    \path[pattern=north west lines] (-1,-1)--(6.5,-1) -- (6.5,1) -- (5,1) -- (5,0) -- (4,0)--(4,1)--(1,1)--(1,0)--(0,0)--(0,3)--(6.5,3)--(6.5, 4)--(-1,4)-- cycle;
  \node at (3, -2) {(a)};
\end{tikzpicture}}\quad\quad\quad\quad\scalebox{0.9}{\begin{tikzpicture}[scale=0.65]
    \filldraw[fill=lightgray] (1,1) rectangle (4,2);
    \filldraw[draw=lightgray,fill=lightgray] (4,1) rectangle (5,2);
    \filldraw[draw=lightgray,fill=lightgray] (0,0) rectangle (1,1);
    
    \filldraw[draw=lightgray,fill=lightgray] (5,1) rectangle (6.5,3);
    \node at (0.5, 1.5) {$\bigtimes$};
    \node at (0.5, 2.5) {$\bigtimes$};
  \node at (4.5, 2.5) {$\bigtimes$};
  \node at (4.5, 0.5) {$\bigtimes$};
  \draw (-1,4.5)--(-1,-1)--(7,-1);
  \draw (0,3)--(0,0);
  \draw (0,3)--(1,3)--(1,2);
  \draw (0,1)--(1,1)--(1,2)--(0,2);
  \draw (4,2)--(5,2)--(5,3)--(4,3)--(4,2);
  \draw (4,0)--(5,0)--(5,1)--(4,1)--(4,0);
  \draw (1,0)--(1,1)--(2,1);
  \draw (2,1)--(3,1);
  \path[pattern=north west lines] (-1,-1)--(6.5,-1) -- (6.5,1) -- (5,1) -- (5,0) -- (4,0)--(4,1)--(1,1)--(1,0)--(0,0)--(0,3)--(1,3)--(1,2)--(4,2)--(4,3)--(6.5,3)--(6.5, 4)--(-1,4)-- cycle;
  \node at (3, -2) {(b)};
\end{tikzpicture}}$$
    \caption{Subdiagrams described in Corollary~\ref{cor:ranked}}
    \label{fig:ranklem}
\end{figure}
\end{example}

\section{At most one cell in each column}\label{sec:obpc}

In this section, we consider Kohnert posets associated with diagrams that contain at most one cell per column. More specifically, we show that the associated Kohnert posets are always bounded and ranked.

\begin{example}
Consider the diagram $D_0=\{(2,1),(3,2),(2,3)\}$. The Hasse diagram of $\mathcal{P}(D_0)$ is illustrated in Figure~\ref{fig:dobechasse}. 

\begin{figure}[H]
    \centering
    $$\scalebox{0.8}{\begin{tikzpicture}
` \node (1) at (0,0) {\begin{tikzpicture}[scale=0.4]
  \node at (0.5, 0.5) {$\bigtimes$};
  \node at (1.5, 0.5) {$\bigtimes$};
  \node at (2.5, 0.5) {$\bigtimes$};
  \draw (0,4)--(0,0)--(4,0);
  \draw (0,0)--(0,1)--(1,1)--(1,0);
  \draw (1,1)--(2,1);
  \draw (2,0)--(2,1)--(3,1)--(3,0);
\end{tikzpicture}};
\node (2) at (-2,3) {\begin{tikzpicture}[scale=0.4]
  \node at (0.5, 1.5) {$\bigtimes$};
  \node at (1.5, 0.5) {$\bigtimes$};
  \node at (2.5, 0.5) {$\bigtimes$};
  \draw (0,4)--(0,0)--(4,0);
  \draw (0,1)--(0,2)--(1,2)--(1,1)--(0,1);
  \draw (1,0)--(1,1)--(2,1);
  \draw (2,0)--(2,1)--(3,1)--(3,0);
\end{tikzpicture}};
\node (3) at (2,3) {\begin{tikzpicture}[scale=0.4]
  \node at (0.5, 0.5) {$\bigtimes$};
  \node at (1.5, 1.5) {$\bigtimes$};
  \node at (2.5, 0.5) {$\bigtimes$};
  \draw (0,4)--(0,0)--(4,0);
  \draw (0,0)--(0,1)--(1,1)--(1,0);
  \draw (1,1)--(1,2)--(2,2)--(2,1)--(1,1);
  \draw (2,0)--(2,1)--(3,1)--(3,0);
\end{tikzpicture}};
\node (4) at (-2,6) {\begin{tikzpicture}[scale=0.4]
  \node at (0.5, 1.5) {$\bigtimes$};
  \node at (1.5, 1.5) {$\bigtimes$};
  \node at (2.5, 0.5) {$\bigtimes$};
  \draw (0,4)--(0,0)--(4,0);
  \draw (0,1)--(0,2)--(1,2)--(1,1)--(0,1);
  \draw (1,2)--(2,2)--(2,1)--(1,1);
  \draw (2,0)--(2,1)--(3,1)--(3,0);
\end{tikzpicture}};
\node (5) at (2,6) {\begin{tikzpicture}[scale=0.4]
  \node at (0.5, 0.5) {$\bigtimes$};
  \node at (1.5, 2.5) {$\bigtimes$};
  \node at (2.5, 0.5) {$\bigtimes$};
  \draw (0,4)--(0,0)--(4,0);
  \draw (0,0)--(0,1)--(1,1)--(1,0);
  \draw (1,2)--(1,3)--(2,3)--(2,2)--(1,2);
  \draw (2,0)--(2,1)--(3,1)--(3,0);
\end{tikzpicture}};
\node (6) at (-2,9) {\begin{tikzpicture}[scale=0.4]
  \node at (0.5, 1.5) {$\bigtimes$};
  \node at (1.5, 1.5) {$\bigtimes$};
  \node at (2.5, 1.5) {$\bigtimes$};
  \draw (0,4)--(0,0)--(4,0);
  \draw (0,1)--(0,2)--(1,2)--(1,1)--(0,1);
  \draw (1,2)--(2,2)--(2,1)--(1,1);
  \draw (2,2)--(3,2)--(3,1)--(2,1);
\end{tikzpicture}};
\node (7) at (2,9) {\begin{tikzpicture}[scale=0.4]
  \node at (0.5, 1.5) {$\bigtimes$};
  \node at (1.5, 2.5) {$\bigtimes$};
  \node at (2.5, 0.5) {$\bigtimes$};
  \draw (0,4)--(0,0)--(4,0);
  \draw (0,1)--(0,2)--(1,2)--(1,1)--(0,1);
  \draw (1,2)--(1,3)--(2,3)--(2,2)--(1,2);
  \draw (2,0)--(2,1)--(3,1)--(3,0);
\end{tikzpicture}};
\node (8) at (0,12) {\begin{tikzpicture}[scale=0.4]
  \node at (0.5, 1.5) {$\bigtimes$};
  \node at (1.5, 2.5) {$\bigtimes$};
  \node at (2.5, 1.5) {$\bigtimes$};
  \draw (0,4)--(0,0)--(4,0);
  \draw (0,1)--(0,2)--(1,2)--(1,1)--(0,1);
  \draw (1,2)--(1,3)--(2,3)--(2,2)--(1,2);
  \draw (2,2)--(3,2)--(3,1)--(2,1)--(2,2);
\end{tikzpicture}};
\draw (2)--(1)--(3);
\draw (2)--(4);
\draw (3)--(5);
\draw (4)--(6);
\draw (4)--(7)--(5);
\draw (6)--(8)--(7);
\end{tikzpicture}}$$
    \caption{Hasse diagram of $\mathcal{P}(D_0)$}
    \label{fig:dobechasse}
\end{figure}
\end{example}

The following result is a direct consequence of Proposition~\ref{prop:nummingen} and keeping in mind our convention for empty columns discussed in Remark~\ref{rem:rempty}.

\begin{theorem}\label{thm:obpcmin}
 If a diagram $D_0$ has at most one cell in each column, then $\mathcal{P}(D_0)$ is bounded. In fact, if $|D_0|=m$, then the unique minimal element of $\mathcal{P}(D_0)$ is given by $D_{min}=
 \{(1,1), (1,2), \ldots, (1,m)\}$.
\end{theorem}

\begin{remark}
    Note that, as a consequence of Theorem~\ref{thm:obpcmin}, for a diagram $D_0$ in which every column contains at most one cell, we have $b(D_0)=|D_0|$.
\end{remark}

Next, we show that such Kohnert posets are always ranked.

\begin{theorem}\label{thm:obpcrank}
Let $D_0$ be a diagram with at most one cell in each column. Then $\mathcal{P}(D_0)$ is a ranked poset with a rank function $rk(D)$ defined for $D\in KD(D_0)$ to be the number of empty boxes contained below the cells in $D$; that is, 
$rk(D) =rowsum(D)-b(D_0)$.
\end{theorem}
\begin{proof}
    Let $D_1, D_2\in \mathcal{P}(D_0)$ with  $D_2 \precdot D_1$. Then $D_2$ is obtained by applying exactly one nontrivial Kohnert move to $D_1$. Since there is at most one cell in each column, a nontrivial Kohnert move takes a cell $(r,c)$ to $(r-1,c)$. Thus, $rk(D_2)=rk(D_1)-1$ and the result follows.
\end{proof}

\section{Two-row diagrams}\label{sec:2r}

In this section, we consider Kohnert posets associated with diagrams for which there are only two nonempty rows; we refer to such diagrams as \textbf{two-row diagrams}. Note that we do not assume that the two nonempty rows of a two-row diagram are consecutive.

Letting $$D=\{(r_1,c^1_i)~|~1\le i\le n_1\}\cup \{(r_2,c^2_i)~|~1\le i\le n_2\}$$ be a two-row diagram with cells in rows $r_1$ and $r_2$ with $r_1<r_2$, we encode $D$ more compactly as $$D=(r_1|c_1^1,\hdots,c_{n_1}^1)\cup(r_2|c_1^2,\hdots,c_{n_2}^2).$$ 
Moreover, we make use of the following notation corresponding to a partitioning of the nonempty columns of $D$.
\begin{itemize}
    \item $\col[]{D}{r_1,r_2}$ denotes the set of columns of $D$ containing cells in both rows, i.e.,
    $$\col[]{D}{r_1,r_2}=\{c_1^1,\hdots,c_{n_1}^1\}\cap \{c_1^2,\hdots,c_{n_2}^2\}.$$
    In general, for $D'\in KD(D)$, we refer to the two cells in a column of $\col[]{D}{r_1,r_2}$ (not necessarily in consecutive rows) as a \textbf{block}.
    \item $\col[]{D}{r_i}$ denotes the set of columns of $D$ containing cells only in row $r_i$ for $i=1,2$, i.e.,  for $i=1,2$,
   $$ \col[]{D}{r_i}=\{c_1^i,\hdots,c_{n_i}^i\}\backslash\col[]{D}{r_1,r_2}.$$
    \item $\col[\leftarrow]{D}{r_i}$ denotes the set of columns of $D$ containing cells only in row $r_i$ for $i=1,2$ which are to the left of the rightmost column in $\col[]{D}{r_1,r_2}$, i.e., for $i=1,2$,
    $$\col[\leftarrow]{D}{r_i}=\{c\in \col[]{D}{r_i}~|~c<\max ~\col[]{D}{r_1,r_2}\}.$$
    \item $\col[\rightarrow]{D}{r_i}$ denotes the set of columns of $D$ containing boxes only in row $r_i$ for $i=1,2$ which are to the right of the rightmost column in $\col[]{D}{r_1,r_2}$,  i.e.,  for $i=1,2$,
    $$\col[\rightarrow]{D}{r_i}=\{c\in \col[]{D}{r_i}~|~c>\max ~\col[]{D}{r_1,r_2}\}.$$
\end{itemize}
Note that $\col[\leftarrow]{D}{r_i}$ and $\col[\rightarrow]{D}{r_i}$ refer to columns containing cells to the left and right, respectively, of the largest column in $\col[]{D}{r_1,r_2}$, and that their intersection with $\col[]{D}{r_1,r_2}$ is empty. The following example illustrates these notions. 
\begin{example}
Consider the diagram $D=(2|2,3,4,6,8,11,13)\cup (4|1,4,5,7,8,9,10,12)$ illustrated in Figure~\ref{fig:2r}. Then $r_1=2$, $r_2=4$, and
$$
\begin{array}{lll}
\col[]{D}{r_1,r_2}=\{4,8\} & & \\
\col[]{D}{r_1}=\{2,3,6,11,13\} & \qquad \col[\rightarrow]{D}{r_1}=\{11,13\} & \qquad \col[\leftarrow]{D}{r_1}=\{2,3,6\} \\
\col[]{D}{r_2}=\{1,5,7,9,10,12\} & \qquad \col[\rightarrow]{D}{r_2}=\{9,10,12\} & \qquad \col[\leftarrow]{D}{r_2}=\{1,5,7\}
\end{array}
$$
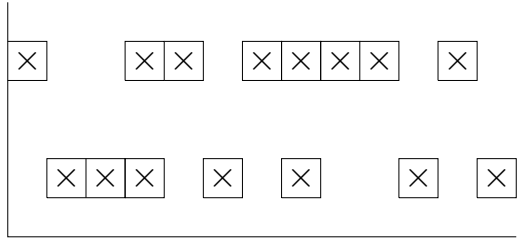
\begin{figure}[H]
    \centering
    $$\scalebox{0.8}{\begin{tikzpicture}[scale=0.65]
  \node at (0.5, 4.5) {$\bigtimes$};
  \node at (1.5, 1.5) {$\bigtimes$};
  \node at (2.5, 1.5) {$\bigtimes$};
  \node at (3.5, 1.5) {$\bigtimes$};
  \node at (3.5, 4.5) {$\bigtimes$};
  \node at (4.5, 4.5) {$\bigtimes$};
  \node at (5.5, 1.5) {$\bigtimes$};
  \node at (6.5, 4.5) {$\bigtimes$};
  \node at (7.5, 4.5) {$\bigtimes$};
  \node at (7.5, 1.5) {$\bigtimes$};
  \node at (8.5, 4.5) {$\bigtimes$};
  \node at (9.5, 4.5) {$\bigtimes$};
  \node at (10.5, 1.5) {$\bigtimes$};
  \node at (11.5, 4.5) {$\bigtimes$};
  \node at (12.5, 1.5) {$\bigtimes$};
  \draw (0,6)--(0,0)--(13,0);
  \draw (0,4)--(1,4)--(1,5)--(0,5);
  \draw (1,1)--(3,1)--(3,2)--(1,2)--(1,1);
  \draw (2,1)--(2,2);
  \draw (3,1)--(4,1)--(4,2)--(3,2)--(3,1);
  \draw (3,4)--(4,4)--(4,5)--(3,5)--(3,4);
  \draw (4,4)--(5,4)--(5,5)--(4,5);
  \draw (5,1)--(6,1)--(6,2)--(5,2)--(5,1);
  \draw (6,4)--(7,4)--(7,5)--(6,5)--(6,4);
  \draw (7,4)--(8,4)--(8,5)--(7,5)--(7,4);
  \draw (7,1)--(8,1)--(8,2)--(7,2)--(7,1);
  \draw (8,4)--(9,4)--(9,5)--(8,5)--(8,4);
  \draw (9,4)--(10,4)--(10,5)--(9,5);
  \draw (10,1)--(11,1)--(11,2)--(10,2)--(10,1);
  \draw (11,4)--(12,4)--(12,5)--(11,5)--(11,4);
  
  \draw (12,1)--(13,1)--(13,2)--(12,2)--(12,1);
\end{tikzpicture}}$$
    \caption{$D=(2|2,3,4,6,8,11,13)\cup (4|1,4,5,7,8,9,10,12)$}
    \label{fig:2r}
\end{figure}
\end{example}

In studying the Kohnert posets associated with two-row diagrams, the following observation proves useful. 
\begin{lemma}\label{lem:2rbase}
Let $D_0=(r_1|c_1^1,\hdots,c_{n_1}^1)\cup(r_2|c_1^2,\hdots,c_{n_2}^2)$ with $1<r_1<r_2$. Then $D=(2|c_1^1,\hdots,c_{n_1}^1)\cup(3|c_1^2,\hdots,c_{n_2}^2)\in KD(D_0)$. 
\end{lemma}
\begin{proof}
Intuitively, we can form $D$ from $D_0$ by first applying a sequence of Kohnert moves which drop the cells of row $r_1$ down to row 2 and then applying a sequence of Kohnert moves which drop the cells of row $r_2$ down to row 3. More formally, we can form $D$ from $D_0$ as follows.
\begin{itemize}
    \item If $r_1>2$, then successively apply $n_1$ Kohnert moves at rows $r_1$ down to 3; and do nothing otherwise. 
    \item If $r_2>3$, then successively apply $n_2$ Kohnert moves at rows $r_2$ down to 4; and do nothing otherwise.
\end{itemize}
\end{proof}

Starting with boundedness, unlike in the case of diagrams with at most one cell per column, two-row diagrams can generate posets which are not bounded. In particular, we have the following.

\begin{theorem}\label{thm:2rnummin}
Let $D_0=(r_1|c_1^1,\hdots,c_{n_1}^1)\cup(r_2|c_1^2,\hdots,c_{n_2}^2)$ with $r_1<r_2$. Then
$$|Min(D_0)|=\begin{cases}
|\col[\leftarrow]{D_0}{r_1}|+1, & r_1>1; \\
1, & r_1=1.
\end{cases}$$
\end{theorem}

In order to prove Theorem~\ref{thm:2rnummin}, given a two-row diagram $D$, we first show that elements of $Min(D)$ are completely determined by the positions of the cells in columns of $\col[\leftarrow]{D}{r_1}$.

\begin{lemma}\label{lem:2rminrest}
Consider $D_0=(r_1|c_1^1,\hdots,c_{n_1}^1)\cup(r_2|c_1^2,\hdots,c_{n_2}^2)$ with $r_1<r_2$ and let $D\in KD(D_0)$. Then $D\in Min(D_0)$ if and only if the following conditions hold:
\begin{enumerate}
    \item[\textup{(a)}] all the cells of $D$ are contained in rows 1 and 2;
    \item[\textup{(b)}] if $(r,c)\in D$ with $c\in \col[\rightarrow]{D_0}{r_1}\cup \col[\rightarrow]{D_0}{r_2}$, then $r=1$;
    \item[\textup{(c)}] if $(r,c)\in D$ with $c\in \col[\leftarrow]{D_0}{r_2}$, then $r=2$; and
    \item[\textup{(d)}] if $\col[\leftarrow]{D_0}{r_1}=\{c_1<\hdots<c_n\}$ and $\{(a_1,c_1),\hdots,(a_n,c_n)\}\subset D$, then $2\ge a_1\ge a_2\ge\cdots\ge a_n\ge 1$.
\end{enumerate}
In the case $\col[]{D_0}{r_1,r_2}=\emptyset$, we have $|Min(D_0)|=1$ and the unique minimal diagram has all cells in the first row.
\end{lemma}
\begin{proof} We first show that if $D\in Min(D_0)$, then it satisfies the properties stated in the lemma.
\begin{enumerate} 
    \item[\textup{(a)}] By contradiction, assume that there exists $D_{min}\in KD(D_0)$ and $r>2$ such that $(r,c)\in D_{min}$. Without loss of generality, assume that $c$ is maximal so that for all $\tilde{c}>c$, any cells in column $\tilde{c}$ of $D_{min}$ are contained in rows 1 or 2. Consequently, $(r,c)$ is the rightmost cell in row $r$ of $D_{min}$. Moreover, since there are at most two cells in column $c$ of $D_{min}$, there must exist an empty position below $(r,c)$ in $D_{min}$. Therefore, $D_{min}$ is not fixed by the Kohnert move at row $r$, i.e., $D_{min}$ is not minimal, which is a contradiction. The result follows.
    
    \item[\textup{(b)}] By part (a), we know that if $(r,c)\in D$ with $c\in \col[\rightarrow]{D_0}{r_1}\cup \col[\rightarrow]{D_0}{r_2}$, then $r=1$ or $2$. Assume that there exists $c\in \col[\rightarrow]{D_0}{r_1}\cup \col[\rightarrow]{D_0}{r_2}$ such that $(2,c)\in D$. Moreover, assume that $(2,c)$ is rightmost in row 2. Since $c\in \col[\rightarrow]{D_0}{r_1}\cup \col[\rightarrow]{D_0}{r_2}$, it follows that $(1,c)\notin D$. However, this implies that we can apply a Kohnert move at row $2$ and obtain a new diagram $D'=D\ldownarrow^{(2,c)}_{(1,c)}$. Thus,  $D\notin Min(D_0)$, which is a contradiction. 
    
    \item[\textup{(c)}] Assume that there exists $(r,c)\in D$ such that $c\in \col[\leftarrow]{D_0}{r_2}$ and $r=1$. Let $c^*=\max \col[]{D_0}{r_1,r_2}$. Then $c<c^*$. Note that $(r_2,c),(r_1,c^*),(r_2,c^*)\in D_0$. Consequently, there must exist two diagrams $D_1,D_2\in KD(D_0)$ such that $D\preceq D_2\precdot D_1\preceq D_0$, $(r',c),(r_1',c^*),(r_2',c^*)\in D_1$ with $r'\ge r_2'>r_1'$, and $(r'',c),(r_1'',c^*),(r_2'',c^*)\in D_2$ with $r''<r_2''>r_1''$. Since diagrams of $KD(D_0)$ have exactly one cell in column $c$ and Kohnert moves result in at most one cell moving to a lower row, it must be the case that $r_2'=r_2''$, $D_2$ is formed from $D_1$ by applying a Kohnert move at row $r'$, and $D_2=D_1\ldownarrow^{(r',c)}_{(r'-1,c)}$. However, since $r''=r'-1<r''_2=r_2'\le r'$, we have $r_2'=r'$ so that $(r',c^*)\in D_1$; but this implies that applying a Kohnert move at row $r'$ of $D_1$ cannot affect the cell $(r',c)$, which is a contradiction, and we conclude that no such column $c$ can exist.
    
    \item[\textup{(d)}] Assume that there exist columns $c_1,c_2\in \col[\leftarrow]{D_0}{r_1}$ with $c_1<c_2$ such that $(a_1,c_1),(a_2,c_2)\in D$ and $a_1<a_2$. Since $(r_1,c_1),(r_1,c_2)\in D_0$,  there must exist diagrams $D_1$ and $D_2$ such that $D \preceq D_2\precdot D_1\preceq D_0$, $(r'_1,c_1),(r'_2,c_2)\in D_1$ with $r'_1\ge r_2'$, and $(r''_1,c_1),(r''_2,c_2)\in D_2$ with $r''_1< r_2''$. Arguing as in part (c), we get a contradiction and conclude that no such pair of columns $c_1$ and $c_2$ can exist.
\end{enumerate}

For the other direction, consider a diagram $D\in KD(D_0)$ satisfying the properties of the lemma. Then all cells of $D$ are in rows 1 and 2, and any cells in row 2 either have cells to their right in the same row or have a cell below in row 1. Moreover, the rightmost cell in row 2 of $D$ has a cell below in row 1. Consequently, no cells of $D$ are affected by applying a Kohnert move at row 2. Thus, $D\in Min(D_0)$.
\end{proof}

We are now ready to prove Theorem~\ref{thm:2rnummin}. 

\begin{proof}[Proof of Theorem~\ref{thm:2rnummin}]
By Lemma~\ref{lem:2rminrest}, any element of $Min(D_0)$ is completely determined by the rows occupied by the cells in columns of $\col[\leftarrow]{D_0}{r_1}$. Thus, the case $r_1=1$ follows immediately. 

Now, let's assume that $r_1>1$. Applying Lemma~\ref{lem:2rminrest}, for $D\in Min(D_0)$, the cells in columns of $\col[\leftarrow]{D_0}{r_1}$ of $D$ are in rows 1 and 2, and appear in weakly decreasing rows from left to right. Consequently, $|Min(D_0)|\le |\col[\leftarrow]{D_0}{r_1}|+1$. Thus, it suffices to show that $|Min(D_0)|\ge |\col[\leftarrow]{D_0}{r_1}|+1$.  We do so by constructing a diagram $D_{t+1}\in Min(D_0)$ with exactly $t$ cells in columns of $\col[\leftarrow]{D_0}{r_1}$ occupying row 1 for $0\le t\le |\col[\leftarrow]{D_0}{r_1}|+1$.

By Lemma~\ref{lem:2rbase}, we have that $$\widehat{D}=(2|c_1^1,\hdots,c_{n_1}^1)\cup(3|c_1^2,\hdots,c_{n_2}^2)\in KD(D_0).$$ We construct $D_{t+1}$ for $0\le t\le |\col[\leftarrow]{D_0}{r_1}|+1$ from $\widehat{D}$ in the following way. Let $k = \left| \col[\leftarrow]{D_0}{r_1}\right|$ and $\col[\leftarrow]{D_0}{r_1}=\left\{c_{i_1}<c_{i_2}<\dots<c_{i_k}\right\}$. For $0\le t\le k$, we define
$$n(t)=\begin{cases}
            |\{c~|~c_{i_{k-t}}<c\in \col[]{D_0}{r_1}\cup\col[]{D_0}{r_1,r_2}\}|, & 0\le t<k \\
            |\{c~|~c_{i_{1}}\le c\in \col[]{D_0}{r_1}\cup\col[]{D_0}{r_1,r_2}\}|, & t=k.
    \end{cases}$$
Then, for $0\le t\le k$, to form $D_{t+1}$ from $\widehat{D}$, we successively apply
\begin{enumerate}
    \item[1)] $n(t)$ Kohnert moves at row 2,
     \item[2)] $|\col[]{D_0}{r_2}\cup\col[]{D_0}{r_1,r_2}|$ Kohnert moves at row 3, and
    \item[3)] $|\col[\rightarrow]{D_0}{r_2}|$ Kohnert moves at row 2. 
\end{enumerate}
Note that our procedure for forming $D_{t+1}$ from $\widehat{D}$ proceeds by
\begin{enumerate}
    \item[1)] moving the rightmost $t$ cells in columns of $\col[\leftarrow]{\widehat{D}}{2}$ and all cells in row 2 to the right of such cells down to row 1,
    \item[2)] moving all cells in columns $\col[]{\widehat{D}}{3}$ down to row 2 and arranging for all cells of columns $\col[]{\widehat{D}}{2,3}$ to occupy rows 1 and 2, and finally
    \item[3)] moving all cells in row 2 and columns $\col[\rightarrow]{\widehat{D}}{3}$ down to row 1.
\end{enumerate}
By construction, $D_{t+1}$ has exactly $t$ cells in columns of $\col[\leftarrow]{D_0}{r_1}$ occupying row 1, and, considering Lemma~\ref{lem:2rminrest}, $D_{t+1}\in Min(D_0)$. The result follows.
\end{proof}

As a consequence of Theorem~\ref{thm:2rnummin}, we are immediately led to the following characterization for Kohnert posets associated with two-row diagrams which are bounded.

\begin{theorem}\label{thm:2rmin}
Let $D_0=(r_1|c_1^1,\hdots,c_{n_1}^1)\cup(r_2|c_1^2,\hdots,c_{n_2}^2)$ with $r_1<r_2$. Then $\mathcal{P}(D_0)$ is bounded if and only if either $r_1=1$, or $r_1>1$ and $|\col[\leftarrow]{D_0}{r_1}|=0$.
\end{theorem}

\begin{remark}\label{rem:2rrank}
    Let $D_0$ be a two-row diagram and $D_{min}$ be the minimal element of $\mathcal{P}(D_0)$ for which all cells in columns of $\col[\leftarrow]{D_0}{r_1}$ are in row 1. Then, as a consequence of the proof of Theorem~\ref{thm:2rnummin}, we have that $b(D_0)=rowsum(D_{min})$; this can also be rewritten as $$b(D_0)=|\col[]{D_0}{r_1}|+(1+\delta_{|Col(D_0;r_1,r_2)|>0})|\col[\leftarrow]{D_0}{r_2}|+|\col[\rightarrow]{D_0}{r_2}|+3|\col[]{D_0}{r_1,r_2}|.$$ 
\end{remark}

Next, we characterize when the Kohnert posets of two-row diagrams are ranked. Our first result considers the case when the initial diagram has cells in the first row.

\begin{lemma}\label{lem:2rgn1}
Let $D_0=(1|c_1^1,\hdots,c^1_{n_1})\cup(r_2|c_1^2,\hdots,c_{n_2}^2)$ with $r_2>1$. Then $\mathcal{P}(D_0)$ is ranked with a rank function given by
 $$rk(D)=rowsum(D)-b(D_0)$$ 
 for $D\in KD(D_0)$.
\end{lemma}
\begin{proof}
By Theorem~\ref{thm:2rmin}, $\mathcal{P}(D_0)$ has a unique minimal element $D_{min}$. Moreover, by the definition of $b(D_0)$, we have $rk(D_{min})=0$ and $rk(D)>0$ for all other $D\in KD(D_0)$. 

Now, let $D\in KD(D_0)$ with $D\neq D_{min}$. Then $rk(D)>0$ and we can apply a Kohnert move to some row $r>1$ in $D$, resulting in a new diagram $D'$. Note that in $D$, all blocks have one cell in the first row. Consequently, there is a column $c$ such that $(r,c)\in D$, $(r-1,c)\notin D$, and $D'=D\ldownarrow^{(r,c)}_{(r-1,c)}$. Thus, $rowsum(D')=rowsum(D)-1$ so that $rk(D')=rk(D)-1$; that is, the rank function decreases by one every time we apply a nontrivial Kohnert move. The result follows.
\end{proof}

Next, we want to find restrictions for the case when $r_1>1$.

\begin{lemma}\label{lem:2rgn2}
Let $D_0=(r_1|c_1^1,\hdots,c^1_{n_1})\cup(r_2|c_1^2,\hdots,c_{n_2}^2)$ with $1<r_1<r_2$. Suppose that either
\begin{itemize}
    \item[\textup{(a)}] $|\col[\rightarrow]{D_0}{r_1}\cup \col[\rightarrow]{D_0}{r_2}|>0$, i.e., $D_0$ has a nonempty column to the right of the rightmost block; or
    \item[\textup{(b)}] $|\col[\rightarrow]{D_0}{r_1}\cup \col[\rightarrow]{D_0}{r_2}|=0$ and $|\col[]{D_0}{r_1,r_2}|>1$, i.e., $D_0$ has more than one block and no cells to the right of the rightmost block.
\end{itemize}
Then $\mathcal{P}(D_0)$ is not ranked.
\end{lemma}
\begin{proof}
We analyze each case separately, finding diagrams in $KD(D_0)$ that contain a subdiagram of one of the two forms described in Corollary~\ref{cor:ranked}.
\begin{itemize}
\item[(a)] 
Let $c=\max~\col[]{D_0}{r_1,r_2}$ and let $c'=\min \col[\rightarrow]{D_0}{r_1}\cup \col[\rightarrow]{D_0}{r_2}$. Without lost of generality, assume that $c'\in \col[\rightarrow]{D_0}{r_2}$; the case $c'\in \col[\rightarrow]{D_0}{r_1}$ following via a similar argument. We claim that there exists $D\in KD(D_0)$ which contains a subdiagram of the form described in Corollary~\ref{cor:ranked}~(a). 
By Lemma~\ref{lem:2rbase}, $\widehat{D}=(2|c_1^1,\hdots,c^1_{n_1})\cup(3|c_1^2,\hdots,c_{n_2}^2)\in KD(D_0)$. Form $D$ from $\widehat{D}$ by successively applying
\begin{enumerate}
    \item[1)] $|\col[\rightarrow]{D_0}{r_1}|$ Kohnert moves at row 2,
    \item[2)] $|\col[\rightarrow]{D_0}{r_2}|-1$ Kohnert moves at row $3$,
    \item[3)] $|\col[\rightarrow]{D_0}{r_2}|-1$ Kohnert moves at row 2, and finally
    \item[4)] one Kohnert move at row 3.
\end{enumerate}
Note that steps 1, 2, and 3 together move all cells in columns of $\col[\rightarrow]{D_0}{r_1}$ and all but the leftmost cell in columns of $\col[\rightarrow]{D_0}{r_2}$ down to row 1; and step 4 moves the leftmost cell in columns of $\col[\rightarrow]{D_0}{r_2}$ down to row 2. Now, by construction, $D$ contains a subdiagram of the form described in Corollary~\ref{cor:ranked} with $r^*=1$, $c_1=c$, and $c_2=c'$, establishing the claim. Applying Corollary~\ref{cor:ranked}, the result follows.

\item[(b)] Let $c,c'$ be the two largest values in $\col[]{D_0}{r_1,r_2}$ with $c<c'$. We claim that there exists $D\in KD(D_0)$ which contains a subdiagram of the form described in Corollary~\ref{cor:ranked}~(b). By Lemma~\ref{lem:2rbase}, $\widehat{D}=(2|c_1^1,\hdots,c^1_{n_1})\cup(3|c_1^2,\hdots,c_{n_2}^2)\in KD(D_0)$. Form $D$ from $\widehat{D}$ by applying $$|\{\tilde{c}\in \col[\rightarrow]{D}{r_1}~|~c<\tilde{c}<c'\}|+1$$ Kohnert moves at row 2. Note that we form $D$ by moving all cells in row 2 and strictly to the left of column $c$ down to row 1. Now, by construction, $D$ contains a subdiagram of the form described in Corollary~\ref{cor:ranked}~(b) with $r^*=1$, $c_1=c$, and $c_2=c'$, establishing the claim. Applying Corollary~\ref{cor:ranked}, the result follows.
\end{itemize}
\end{proof}

It turns out that the necessary conditions of Lemma~\ref{lem:2rgn2} for a Kohnert poset associated with a two-row diagram to be ranked are, in fact, also sufficient.

\begin{theorem}\label{thm:2rrank}
Let $D_0=(r_1|c_1^1,\hdots,c^1_n)\cup(r_2|c_1^2,\hdots,c_m^2)$. Then  $\mathcal{P}(D_0)$ is ranked if and only if either
\begin{itemize}
    \item[\textup{(a)}] $r_1=1$ or
    \item[\textup{(b)}] $r_1>1$, $|\col[\rightarrow]{D_0}{r_1}\cup\col[\rightarrow]{D_0}{r_2}|=0$, and $|\col[]{D_0}{r_1,r_2}|\le 1$.
\end{itemize}
Moreover, for $D\in KD(D_0)$, a rank function is given by
$$rk(D)=rowsum(D)-b(D_0).$$
\end{theorem}
\begin{proof}
The forward direction follows by a combination of Lemmas~\ref{lem:2rgn1} and~\ref{lem:2rgn2}. For the backward direction, note that if $r_1=1$, then the result follows by Lemma~\ref{lem:2rgn1}. Thus, assume that $r_1>1$. Note that, considering Remark~\ref{rem:2rrank}, if we let $D_{min}\in Min(D_0)$ be the minimal element such that all cells contained in columns $\col[\leftarrow]{D_0}{r_1}$ are in the first row, then $rk(D_{min})=0$ and $rk(D)>0$ for all other $D\in KD(D_0)$. To establish the result, it suffices to show that 
\begin{enumerate}
    \item if $D_1,D_2\in KD(D_0)$ satisfy $D_2\prec D_1$, then $rk(D_1)>rk(D_2)$; and
    \item if $rk(D_1)-rk(D_2)>1$, then $D_2\prec D_1$ is not a cover relation.
\end{enumerate}
\noindent
The fact that $D_2\prec D_1$ implies $rk(D_1)>rk(D_2)$ follows immediately from the definition of $rk$ and the fact that nontrivial Kohnert moves result in cells moving to lower rows. Thus, for a contradiction, assume that there exists $D_1,D_2\in KD(D_0)$ such that $D_2\precdot D_1$ and $rk(D_1)-rk(D_2)>1$. Since $D_2\precdot D_1$, the diagrams $D_1$ and $D_2$ are related by a single Kohnert move, say at row $r$. Assume that $D_2=D_1\ldownarrow^{(r,c)}_{(r',c)}$. As $D_0$ is a two-row diagram and $rk(D_1)-rk(D_2)>1$, it follows that $(r-1,c)\in D_1$ and $r'=r-2$. Thus, $(r,c)$ belongs to the unique block of $D_1$. By assumption, columns $\tilde{c}>c$ are empty in $D_1$. Consequently, $\tilde{D}:=D_1\ldownarrow^{(r-1,c)}_{(r-2,c)}\in KD(D_0)$ satisfies $$D_2=\tilde{D}\ldownarrow^{(r,c)}_{(r-1,c)}\prec \tilde{D}\prec D_1,$$ i.e., $D_2\prec D_1$ is not a cover relation, a contradiction. Therefore, $rk$ forms a rank function on $\mathcal{P}(D_0)$ and the result follows.
\end{proof}

Combining Theorems~\ref{thm:2rmin} and Theorem~\ref{thm:2rrank} we are led to the following.
\begin{corollary}
    Suppose that $D_0$ is a two-row diagram. If $\mathcal{P}(D_0)$ is ranked, then $\mathcal{P}(D_0)$ is also bounded. Moreover, the converse is not true in general.
\end{corollary}

\section{Key diagrams}\label{sec:cd}

In this section, we consider Kohnert posets associated with certain diagrams defined by weak compositions, called ``key diagrams''. As mentioned in the introduction, these are the diagrams whose Kohnert polynomials are Demazure characters.

Let $\mathbf{a}=(a_1,\hdots,a_n)\in \mathbb{Z}_{\ge 0}^n$ be a weak composition. The \textbf{key diagram} corresponding to $\mathbf{a}$, denoted $\mathbb{D}(\mathbf{a})$, is the diagram containing $a_i$ left justified cells in row $i$ for $1\le i\le n$; that is
$$\mathbb{D}(\mathbf{a})=\bigcup_{i=1}^n\{(i,j)~|~1\le j\le a_i\}.$$

Given a weak composition $\mathbf{a}$, $\mathcal{P}\left(\mathbb{D}(\mathbf{a})\right)$ has a unique minimal element as a direct consequence of Proposition~\ref{prop:nummingen}.

\begin{theorem}\label{thm:uniqminkey}
For any weak composition $\mathbf{a}=(a_1,a_2,\hdots,a_n)$, the poset $\mathcal{P}(\mathbb{D}(\mathbf{a}))$ is bounded.
\end{theorem}

In fact, we can describe the unique minimal element. Intuitively, following the proof of Proposition~\ref{prop:nummingen}, the unique minimal element is the diagram obtained by lowering all the cells in $\mathbb{D}(\mathbf{a})$ to the bottom and so is actually the key diagram of a particular weak composition. In Example~\ref{ex:minimalKey}, we illustrate a key diagram along with the minimal element of its associated Kohnert poset. Given a weak composition $\mathbf{a}=(a_1,\hdots, a_n)$, let $\texttt{sort}(\mathbf{a})$ be the weakly decreasing sequence obtained by reordering the entries of $\mathbf{a}$. 

\begin{corollary}\label{cor:permutingall}
Given a weak composition $\mathbf{a}=(a_1,\hdots, a_n)$, $\mathbb{D}(\texttt{sort}(\mathbf{a}))$ in the unique minimal element of $\mathcal{P}(\mathbb{D}(\mathbf{a}))$.
\end{corollary}

\begin{example}\label{ex:minimalKey}
    For $\mathbf{a}=(0,3,4,2,3)$, we have that $\texttt{sort}(\mathbf{a})=(4,3,3,2,0)$ and in Figure~\ref{fig:key} we illustrate \textup{(a)} $\mathbb{D}(\mathbf{a})$ together with \textup{(b)} the minimal element $\mathbb{D}(\texttt{sort}(\mathbf{a}))$ of its associated Kohnert poset.
    \begin{figure}[H]
        \centering
        $$\begin{tikzpicture}[scale=0.4]
        \node at (0.5, 1.5) {$\times$};
        \node at (1.5, 1.5) {$\times$};
        \node at (2.5, 1.5) {$\times$};
  \node at (0.5, 2.5) {$\times$};
  \node at (1.5, 2.5) {$\times$};
  \node at (2.5, 2.5) {$\times$};
  \node at (3.5, 2.5) {$\times$};
  \node at (0.5, 3.5) {$\times$};
  \node at (1.5, 3.5) {$\times$};
  \node at (0.5, 4.5) {$\times$};
  \node at (1.5, 4.5) {$\times$};
  \node at (2.5, 4.5) {$\times$};
  \draw (0,6)--(0,0)--(5,0);
  \draw (0,1)--(3,1)--(3,2)--(0,2);
  \draw (1,1)--(1,5);
  \draw (2,1)--(2,2);
  \draw (0,3)--(1,3);
  \draw (0,4)--(2,4)--(2,3)--(1,3);
  \draw (0,5)--(3,5)--(3,4)--(2,4);
  \draw (2,4)--(2,5);
  \draw (2,3)--(4,3)--(4,2)--(3,2);
  \draw (2,2)--(2,3);
  \draw (3,2)--(3,3);
  \node at (2.5, -1) {(a)};
\end{tikzpicture}\quad\quad\quad\quad\quad\quad \begin{tikzpicture}[scale=0.4]
        \node at (0.5, 1.5) {$\times$};
        \node at (1.5, 1.5) {$\times$};
        \node at (2.5, 1.5) {$\times$};
  \node at (0.5, 0.5) {$\times$};
  \node at (1.5, 0.5) {$\times$};
  \node at (2.5, 0.5) {$\times$};
  \node at (3.5, 0.5) {$\times$};
  \node at (0.5, 3.5) {$\times$};
  \node at (1.5, 3.5) {$\times$};
  \node at (0.5, 2.5) {$\times$};
  \node at (1.5, 2.5) {$\times$};
  \node at (2.5, 2.5) {$\times$};
  \draw (0,6)--(0,0)--(5,0);
  \draw (0,1)--(4,1)--(4,0);
  \draw (0,2)--(3,2)--(3,1);
  \draw (0,3)--(3,3)--(3,2);
  \draw (0,4)--(2,4)--(2,3);
  \draw (1,0)--(1,4);
  \draw (2,0)--(2,3);
  \draw (3,0)--(3,1);
  \node at (2.5, -1) {(b)};
\end{tikzpicture}$$
        \caption{(a) Key diagram and (b) associated minimal element}
        \label{fig:key}
    \end{figure}
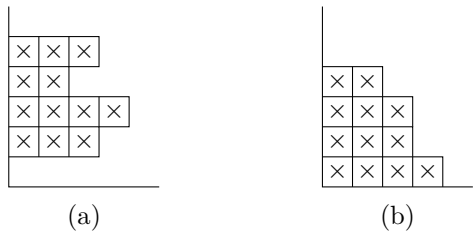
\end{example}

Next, we fully characterize when the Kohnert poset associated with a key diagram $\mathbb{D}(\mathbf{a})$ is ranked. We present such a characterization by first defining a particular family of weak compositions, called ``pure compositions'', and showing that $\mathcal{P}\left(\mathbb{D}(\mathbf{a})\right)$ is ranked if and only if $\mathbf{a}$ is a pure composition. This proof is more complex, and we split it into several steps. 

\begin{definition}\label{def:purecomp}
A weak composition $\mathbf{a}=(a_1,\hdots,a_n)$ is said to be \textbf{pure} if there exist no distinct indices $j_1,\ j_2,\ j_3$, with $1\le j_1<j_2<j_3\le n$, such that $a_{j_1}<a_{j_2}<a_{j_3}$, $a_{j_1}<a_{j_3}<a_{j_2}$, or $a_{j_1}+1<a_{j_2}=a_{j_3}$. Otherwise, we say that $\mathbf{a}$ is \textbf{non-pure}. \textup(See Example~\ref{ex:purecomp} for an example of a pure composition.\textup)
\end{definition}

We dedicate the rest of the section to proving the following result.

\begin{theorem}\label{thm:keyrank}
    Let $\mathbf{a}$ be a weak composition. Then $\mathcal{P}(\mathbb{D}(\mathbf{a}))$ is ranked if and only if $\mathbf{a}$ is a pure composition.
\end{theorem}

The proof of this result will follow from Theorem~\ref{thm:keynecrank} and Lemma~\ref{lem:keysuff} below.

\begin{remark}
    The name ``pure'' composition comes from the fact that, as we show in this section, the corresponding Kohnert posets are pure posets, i.e., all maximal chains are of equal length. In particular, this follows from Lemma~\ref{lem:bdrank} along with Theorems~\ref{thm:uniqminkey} and~\ref{thm:keyrank}.

    Comparing the conditions defining pure compositions with those found in~\textup{\cite{MF}}, we find that pure compositions generate multiplicity-free Demazure characters, i.e., coefficients of monomials are 0 or 1. In~\textup{\cite{MF}}, Demazure characters are referred to as key polynomials.  Note that, although sufficient, being pure is not a necessary condition for a weak composition to generate a multiplicity-free Demazure character.
\end{remark}

\subsection{Properties of pure compositions}

In this section, we show that pure compositions can be decomposed into ``smaller'' pure compositions of four basic types.

First, notice the following direct consequence of the definition of pure composition.
\begin{lemma}\label{lem:purecons}
    Let $\mathbf{a}=(a_1,\hdots,a_n)$ be a pure composition. 
    \begin{enumerate}
        \item[\textup{(a)}] Suppose there exist indices $i$ and $j$ with $1\le i<j\le n$ such that $a_j-a_i=1$. Then for all $k\ge j$, $a_k\le a_j$; that is, if there are two rows of $\mathbb{D}(\mathbf{a})$ such that the higher row has exactly one more cell than the lower row, then all the rows above the higher row have at most as many cells as the higher row. 
        
        \item[\textup{(b)}] Suppose there exist indices $i$ and $j$ with $1\le i<j\le n$ such that $a_j-a_i>1$. Then for all $k\geq j$, $a_k\le a_i$; that is, if there are two rows of $\mathbb{D}(\mathbf{a})$ such that the higher row has at least two cells more than the lower row, then all the rows above the higher row have at most as many cells as the lower row. 
    \end{enumerate}
\end{lemma}

Illustrations of the consequences of Lemma~\ref{lem:purecons}~(a) and (b) are given in Figure~\ref{fig:purecons}~(a) and (b), respectively.

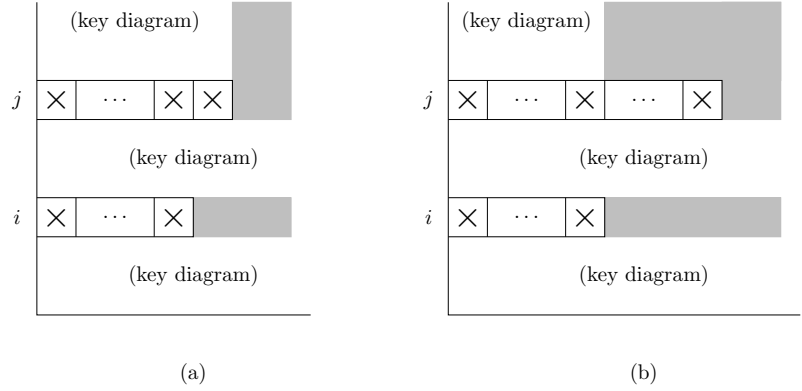
\begin{figure}[H]
    \centering
    $$\scalebox{0.8}{\begin{tikzpicture}[scale=0.65]
    \node at (0.5,5.5) {$\bigtimes$};
    \node at (3.5,5.5) {$\bigtimes$};
    \node at (4.5,5.5) {$\bigtimes$};
    \node at (-0.5,5.5) {$j$};
    \node at (-0.5,2.5) {$i$};
    \node at (2,5.5) {$\hdots$};
    
    \node at (0.5,2.5) {$\bigtimes$};
    \node at (2,2.5) {$\hdots$};
    \node at (3.5,2.5) {$\bigtimes$};
    \node at (2.5,7.5) {(key diagram)};
    \node at (4,4) {(key diagram)};
    \node at (4,1) {(key diagram)};
    \draw (0,8)--(0,0)--(7,0);
    \filldraw[draw=lightgray,fill=lightgray] (4,2) rectangle (6.5,3);
    \filldraw[draw=lightgray,fill=lightgray] (5,5) rectangle (6.5,8);
    \draw (0,2)--(4,2)--(4,3)--(0,3);
    \draw (1,2)--(1,3);
    \draw (3,2)--(3,3);
    \draw (0,5)--(5,5)--(5,6)--(0,6);
    \draw (1,5)--(1,6);
    \draw (3,5)--(3,6);
    \draw (4,5)--(4,6);
    \node at (4,-1.5) {(a)};
\end{tikzpicture}}\quad\quad\quad\quad\scalebox{0.8}{\begin{tikzpicture}[scale=0.65]
    \node at (0.5,5.5) {$\bigtimes$};
    \node at (3.5,5.5) {$\bigtimes$};
    \node at (5,5.5) {$\hdots$};
    \node at (6.5,5.5) {$\bigtimes$};
    \node at (-0.5,5.5) {$j$};
    \node at (-0.5,2.5) {$i$};
    \node at (2,5.5) {$\hdots$};
    
    \node at (0.5,2.5) {$\bigtimes$};
    \node at (2,2.5) {$\hdots$};
    \node at (3.5,2.5) {$\bigtimes$};
    \node at (2,7.5) {(key diagram)};
    \node at (5,4) {(key diagram)};
    \node at (5,1) {(key diagram)};
    \draw (0,8)--(0,0)--(9,0);
    \filldraw[draw=lightgray,fill=lightgray] (4,2) rectangle (8.5,3);
    \filldraw[draw=lightgray,fill=lightgray] (7,5) rectangle (8.5,8);
    \filldraw[draw=lightgray,fill=lightgray] (4,6) rectangle (8.5,8);
    \draw (0,2)--(4,2)--(4,3)--(0,3);
    \draw (1,2)--(1,3);
    \draw (3,2)--(3,3);
    \draw (0,5)--(7,5)--(7,6)--(0,6);
    \draw (1,5)--(1,6);
    \draw (3,5)--(3,6);
    \draw (4,5)--(4,6);
    \draw (6,5)--(6,6);
    \node at (5,-1.5) {(b)};
\end{tikzpicture}}$$
    \caption{Consequences of Lemma~\ref{lem:purecons}~(a) and (b)}
    \label{fig:purecons}
\end{figure}

The next result essentially says that any pure composition can be split into ``smaller'' pure compositions, each of which is of one of four possible forms. For notation, given a weak composition $\mathbf{a}=(a_1,\hdots,a_n)$, ongoing we set $\max(\mathbf{a})=\max\{a_i~|~1\le i\le n\}$ and $\min(\mathbf{a})=\min\{a_i~|~1\le i\le n\}$, i.e., $\max(\mathbf{a})$ (resp., $\min(\mathbf{a})$) is the maximal (resp., minimal) entry of $\mathbf{a}$.

\begin{lemma}\label{lem:purecomp}
    Let $\mathbf{a}=(a_1,\hdots,a_n)$ be a pure composition. Then there exist indices $i_0, \ldots, i_{m+1}$ with $1=i_0<\hdots < i_{m}<i_{m+1}=n+1$ such that if we define $\alpha_j=(a_{i_j},a_{i_j+1}\hdots,a_{i_{j+1}-1})$ for $0\le j\le m$, then $\min(\alpha_{j-1})\ge \max(\alpha_j)$ for $0<j\le m$, and each $\alpha_j$ is of one of the following forms:
    \begin{enumerate}
        \item[\textup{(i)}] $a_{i_j}\ge \hdots\ge a_{i_{j+1}-1}$; that is, $\alpha_j$ is a weakly decreasing sequence. 
        \item[\textup{(ii)}] There exists $p\in\mathbb{Z}_{\ge 0}$ such that $a_{i_j}=p$ and $\{a_{i_j},\hdots,a_{i_{j+1}-1}\}=\{p,p+1\}$; that is, all entries of $\alpha_j$ are $p$ or $p+1$ for some $p\in \mathbb{Z}_{\ge 0}$, the first entry is $p$, and at least one other entry must be equal to $p+1$. 
        \item[\textup{(iii)}] $a_{i_j}\ge\hdots\ge a_{i_{j+1}-2}<a_{i_{j+1}-1}-1$; that is, the entries of $\alpha_j$ are in decreasing order, except the last one which is at least two larger than the penultimate one. 
        \item[\textup{(iv)}] There exist $p\in\mathbb{Z}_{\ge 0}$ and $i_j^*\in\mathbb{Z}_{> 0}$ with $i_j+1<i_j^*<i_{j+1}-1$ such that $a_{i_j}=p$, $\{a_{i_j},\hdots,a_{i_j^*-1}\}=\{p,p+1\}$, $p>a_{i_j^*}\ge \hdots\ge a_{i_{j+1}-2}$, and $a_{i_{j+1}-1}=p+1$.
    \end{enumerate}
\end{lemma}

\begin{remark}
   To guarantee uniqueness of such decompositions it suffices to assume that the subcompositions of each type are taken to be as large as possible. 
\end{remark}

Before giving the proof of Lemma~\ref{lem:purecomp}, we provide an example of a pure composition along with a decomposition into subcompositions of the form described therein.

\begin{example}\label{ex:purecomp}
    Consider the pure composition $$\mathbf{a}=(25,22,18,18,17,16,14,19,13,10,8,8,7,8,8,7,8,7,5,5,6,5,6,5,4,4,6,2,2,2,3,3,2,3,2,3,3).$$ The decomposition into subcompositions of the form described in Lemma~\ref{lem:purecomp} is given by $$\alpha_1=(25,22,18,18,17,16,14,19),\quad\quad\alpha_2=(13,10,8,8),\quad\quad\alpha_3=(7,8,8,7,8,7),$$ $$\alpha_4=(5,5,6,5,6,5,4,4,6),\quad\quad\text{and}\quad\quad\alpha_5=(2,2,2,3,3,2,3,2,3,3).$$ 
    Note that $\alpha_1$ is of type \textup{(iii)}, $\alpha_2$ is of type \textup{(i)}, $\alpha_3$ and $\alpha_5$ are of type \textup{(ii)}, and $\alpha_4$ is of type \textup{(iv)}. We illustrate in Example~\ref{ex:compositionexplanation} how the above decomposition can be found following the procedure described in the proof of Lemma~\ref{lem:purecomp}.
\end{example}

\begin{proof}
Note that the weak compositions described in (i)--(iv) are pure compositions. We refer to these labels as the types of the pure compositions $\alpha_j$. 

To establish the result, we outline an approach for identifying choices of ``smaller'' pure compositions $\alpha_j$. As input, we take the pure composition $\mathbf{a} = (a_1,\dots,a_n)$ for the first iteration, and the sequence $\mathbf{\tilde{a}}$, described below, for the next iteration if necessary. 

If $n=1$ or there exists no index $k_1$ such that $a_{k_1-1}< a_{k_1}$, then we take $\alpha_1 = \mathbf{a}$, which is of type (i), and the process finishes. Otherwise, if $n>1$ and there exists a least index $k_1>1$ such that $a_{k_1-1}< a_{k_1}$, then there are five cases: 
\bigskip

\noindent
\textbf{Case 1:} $a_{k_1}-a_{k_1-1}>1$. If $k_1=n$, then we take $\alpha_1=\mathbf{a}$, which is of type (iii) and the process finishes. Otherwise, note that by Lemma~\ref{lem:purecons}~(b), $$\max\{a_k~|~k> k_1\}\le a_{k_1-1}=\min\{a_k~|~k\le k_1\}.$$ Thus, if $k_1<n$, then we take $i_1=k_1+1$ so that $\alpha_1=(a_1,\hdots,a_{k_1})$, which is of type (iii), and repeat the process with the input sequence $\mathbf{\tilde{a}} = (a_{k_1+1}, \ldots, a_n)$.
\bigskip

\noindent
\textbf{Case 2:} $a_{k_1}-a_{k_1-1}=1$ and there exists no $k$ such that $1\le k\le n$ and $a_{k}\notin \{a_{k_1-1},a_{k_1}=a_{k_1-1}+1\}$. Then, we take $\alpha_1=(a_1,\hdots,a_n)$, which is of type (ii), and the process finishes. 
\bigskip

\noindent
\textbf{Case 3:} $a_{k_1}-a_{k_1-1}=1$ and there exists $k_0$ such that $1\le k_0<k_1-1$ and $a_{k_0}\ge a_{k_1}$. Assume that $k_0$ is maximal with the aforementioned properties. Note that, by our choice of $k_0$ along with Lemma~\ref{lem:purecons}~(a), $$\max\{a_k~|~k> k_0\}\le a_{k_1}\le a_{k_0}=\min\{a_k~|~k\le k_0\}.$$ Thus, in this case, we take $i_1=k_0+1$ so that $\alpha_1=(a_1,\hdots,a_{k_0})$, which is of type (i), and repeat the process with the input sequence $\mathbf{\tilde{a}} = (a_{k_0+1}, \ldots, a_n)$.
\bigskip

\noindent
\textbf{Case 4:} $a_{k_1}-a_{k_1-1}=1$, $a_k=a_{k_1-1}$ for all $k<k_1$, and there exists $k_2$ such that $k_1<k_2\le n$, $a_{k_2}<a_{k_1-1}$, and $a_{k}\le a_{k_1-1}$ for all $k> k_2$. Assume that $k_2$ is the least index with the aforementioned properties. Note that, by our choice of $k_2$, $$\max\{a_k~|~k\ge k_2\}\le a_{k_1-1}=\min\{a_k~|~k<k_2\}.$$ Thus, in this case, we take $i_1=k_2$ so that $\alpha_1=(a_1,\hdots,a_{k_2-1})$, which is of type (ii), and we repeat the process with the input sequence $\mathbf{\tilde{a}} = (a_{k_2}, \ldots, a_n)$.
\bigskip

\noindent
\textbf{Case 5:} $a_{k_1}-a_{k_1-1}=1$, $a_k=a_{k_1-1}$ for all $k<k_1$, and there exists $k_2$ and $k_3$ such that $k_1<k_2<k_3\le n$, $a_{k_2}<a_{k_1-1}$, and $a_{k_3}=a_{k_1}$. Assume that $k_2$ is minimal with the aforementioned properties. We claim that $a_k\ge a_{k+1}$ for $k_2\le k<k_3-1$. To see this, assume otherwise; that is, assume that there exists a least $k^*$ such that $k_2<k^*<k_3$ and $a_{k^*-1}<a_{k^*}$. If $a_{k^*}\le a_{k_2}$, then $a_{k^*-1}<a_{k^*}<a_{k_3}$ which is a contradiction since $\mathbf{a}$ is pure. Thus, $a_{k^*}>a_{k_2}$. If $a_{k^*}-a_{k_2}=1$, then by Lemma~\ref{lem:purecons}~(a) we must have $$a_k\le a_{k^*}=a_{k_2}+1<a_{k_1-1}+1=a_{k_1}=a_{k_3}$$ for all $k>k^*$, including $k=k_3$. Otherwise, by Lemma~\ref{lem:purecons}~(b), we must have $$a_k\le a_{k_2}<a_{k_1-1}<a_{k_1}=a_{k_3}$$ for all $k>k^*$, including $k=k_3$. In either case, we are led to a contradiction, establishing the claim. Consequently, $a_{k_3-1}=\min\{a_k~|~k\le k_3\}$. Now, if $k_3=n$, then we take $\alpha_1=\mathbf{a}$, which is of type (iv), and the procedure finishes. Otherwise, applying Lemma~\ref{lem:purecons}~(b), we find that $$\max\{a_k~|~k> k_3\}\le a_{k_3-1}=\min\{a_k~|~k\le k_3\}.$$ Therefore, if $k_3<n$, then we take $i_1=k_3+1$ so that $\alpha_1=(a_1,\hdots,a_{k_3})$, which is of type (iv), and we repeat the process with the input sequence $\mathbf{\tilde{a}}=(a_{k_3+1},\hdots,a_n)$.
\bigskip

\noindent
Note that in each case the procedure either finishes with a pure composition of one of the four types or removes an initial sequence corresponding to one of the four types and starts over with the weak composition that remains. Consequently, as weak compositions only have a finite number of entries, the procedure must terminate in a finite number of steps with the desired decomposition.
\end{proof}

\begin{example}\label{ex:compositionexplanation}
    Let $\mathbf{a}$ be the pure composition of Example~\ref{ex:purecomp}. To decompose $\mathbf{a}$ into subcompositions of the form described in Lemma~\ref{lem:purecomp} as outlined in the proof above, first note the $\mathbf{a}$ is neither weakly decreasing nor does it consist of a single entry. Thus, we must find the first ascent, i.e., the first instance where an entry is less than the one that follows it. As the first ascent occurs between the seventh and eighth entries, i.e., 14 and 19, respectively, whose difference is greater than 1, we fall into Case 1 of the proof above so that $$\alpha_1=(25,22,18,18,17,16,14,19)$$ and we start the process over with $$\mathbf{\tilde{a}}=(13,10,8,8,7,8,8,7,8,7,5,5,6,5,6,5,4,4,6,2,2,2,3,3,2,3,2,3,3).$$ Now, as in the previous iteration, we start by finding the first ascent. In this case, the first ascent occurs between the fifth and sixth entries, i.e., 7 and 8, respectively, whose difference is exactly 1. Note that prior to the ascent, there is an entry larger than or equal to 8. Consequently, we are in Case 3 of the proof above. Finding the last such entry from left to right, we take $$\alpha_2=(13,10,10,8,8)$$ and we start the process over with $$\mathbf{\tilde{a}}=(7,8,8,7,8,7,5,5,6,5,6,5,4,4,6,2,2,2,3,3,2,3,2,3,3).$$ The first ascent in the current input composition $\mathbf{\tilde{a}}$ occurs between the first and second entries, i.e., 7 and 8, respectively, whose difference is 1. Note that there are no prior entries greater than or equal to 8, but that there does exist an entry less than 7 later in $\mathbf{\tilde{a}}$. Moreover, the first such entry less than 7 is the seventh in 6, and all entries following are also less than 7. Consequently, we are in Case 4 of the proof above so that $$\alpha_3=(7,8,8,7,8,7)$$ and we start the process over with $$\mathbf{\tilde{a}}=(5,5,6,5,6,5,4,4,6,2,2,2,3,3,2,3,2,3,3).$$ The first ascent in the current version of $\mathbf{\tilde{a}}$ occurs between the second and third entries, i.e., 5 and 6, respectively, whose difference is 1. As there are no prior entries greater than or equal to 6, while there is an entry smaller than 5 followed by one equal to 6 occurring after the first ascent, we must be in Case 5 of the proof above. Consequently, $$\alpha_4=(5,5,6,5,6,5,4,4,6)$$ and we start the process over with $$\mathbf{\tilde{a}}=(2,2,2,3,3,2,3,2,3,3).$$ Finally, the first ascent of the current version of $\mathbf{\tilde{a}}$ occurs between the third and fourth entries, i.e., 2 and 3, respectively, whose difference is 1. As all entries of $\mathbf{\tilde{a}}$ are 2 or 3, it follows that we are in Case 2 of the proof above. Consequently, $$\alpha_5=\mathbf{\tilde{a}}=(2,2,2,3,3,2,3,2,3,3)$$ and we are done.
\end{example}

Ongoing, it proves convenient to have a characterization of non-pure compositions in terms of containment of patterns of one of two forms, instead of three.

\begin{lemma}\label{lem:non-pure}
Let $\mathbf{a}=(a_1,\hdots,a_n)$ be a weak composition.  Then $\mathbf{a}$ is non-pure if and only if there exist indices $j_1,\ j_2,\ j_3$ with $1\le j_1<j_2<j_3\le n$ such that 
    \begin{itemize}
    \item[\textup{(a)}] $a_{j_1}<a_{j_2}<a_{j_3}$ or
    \item[\textup{(b)}] $a_{j_1}+1< a_{j_2}$ and $a_{j_1}<a_{j_3}$.
\end{itemize}
\end{lemma} 

\subsection{Proof of the necessary conditions}
In this section, we establish the following result.

\begin{theorem}\label{thm:keynecrank}
    If $\mathbf{a}=(a_1,\hdots,a_n)$ is a non-pure composition, then $\mathcal{P}(\mathbb{D}(\mathbf{a}))$ is not ranked.
\end{theorem}

We prove Theorem~\ref{thm:keynecrank} following these steps:
\begin{itemize}
\item[1)] In Lemma~\ref{lem:keyconsec}, we establish the analogous result for non-pure compositions containing one of the patterns described in Lemma~\ref{lem:non-pure} where the entries are consecutive, i.e., for $j_1,\ j_2=j_1+1,\ j_3=j_1+2$. 

\item[2)] In Lemma~\ref{lem:permute}, we show that given a key diagram $\mathbb{D}(\mathbf{a})$, the key diagram obtained from $\mathbb{D}(\mathbf{a})$ by permuting a given row with a shorter row below is also an element of $\mathcal{P}(\mathbb{D}(\mathbf{a}))$.

\item[3)] Finally, given a non-pure composition $\mathbf{a}$, we show that we can always apply Lemma~\ref{lem:permute} to find a non-pure composition of the form considered in Lemma~\ref{lem:keyconsec} whose key diagram is contained in $\mathcal{P}(\mathbb{D}(\mathbf{a}))$, establishing the result.
\end{itemize}

As outlined above, first we establish the analogue of Theorem~\ref{thm:keynecrank} for non-pure compositions containing one of the patterns described in Lemma~\ref{lem:non-pure} where the entries are consecutive.

\begin{lemma}\label{lem:keyconsec}
Let $\mathbf{a}=(a_1,\hdots,a_n)$ be a non-pure composition for which there exists an index $i$ such that $1\le i<n-2$ and either
\begin{itemize}
    \item[\textup{(a)}] $a_i<a_{i+1}<a_{i+2}$ or
    \item[\textup{(b)}] $a_i+1< a_{i+1}$ and $a_i<a_{i+2}$.
\end{itemize}
Then $\mathcal{P}(\mathbb{D}(\mathbf{a}))$ is not ranked.
\end{lemma}
\begin{proof} We consider non-pure compositions containing patterns of the form (a) and (b) separately.
\begin{enumerate}
    \item[(a)] Let $D$ denote the diagram formed from $\mathbb{D}(\mathbf{a})$ by applying $a_{i+2}-a_{i+1}$ Kohnert moves at row $i+2$ followed by $a_{i+2}-a_{i+1}-1$ Kohnert moves at row $i+1$; that is, $D$ is the diagram formed from $\mathbb{D}(\mathbf{a})$ by lowering the cells of row $i+2$ that are not above cells in row $i+1$ down to row $i+1$, and then further lowering all those cells except the leftmost one down to row $i$. Illustrations of the pertinent portions of (left) $\mathbb{D}(\mathbf{a})$ and (right) $D$ can be found in Figure~\ref{fig:lemkeyconseca}, where a region filled with $\dagger~\dagger~\dagger$ is either empty or filled with cells.

    \begin{figure}[H]
        \centering
        $$\scalebox{0.75}{\begin{tikzpicture}[scale=0.65]
    \node at (0.5,3.5) {$\bigtimes$};
    \node at (3.5,3.5) {$\bigtimes$};
    \node at (5,3.5) {$\hdots$};
    \node at (6.5,3.5) {$\bigtimes$};
    \node at (-0.75,4.5) {$i+2$};
    \node at (-0.75,3.5) {$i+1$};
    \node at (-0.75,2.5) {$i$};
    \node at (2,3.5) {$\hdots$};

    \node at (0.5,4.5) {$\bigtimes$};
    \node at (3.5,4.5) {$\bigtimes$};
    \node at (5,4.5) {$\hdots$};
    \node at (6.5,4.5) {$\bigtimes$};
    \node at (2,4.5) {$\hdots$};
    \node at (7.5,4.5) {$\bigtimes$};
    \node at (9,4.5) {$\dagger~\dagger~\dagger$};
    
    \node at (0.5,2.5) {$\bigtimes$};
    \node at (2,2.5) {$\hdots$};
    \node at (3.5,2.5) {$\bigtimes$};
    \node at (5,7) {(key diagram)};
    \node at (5,1) {(key diagram)};
    \draw (0,8)--(0,0)--(12,0);
    \filldraw[draw=lightgray,fill=lightgray] (4,2) rectangle (11.5,3);
    \filldraw[draw=lightgray,fill=lightgray] (7,3) rectangle (11.5,4);
    \filldraw[draw=lightgray,fill=lightgray] (10,4) rectangle (11.5,5);
    \draw (0,2)--(4,2)--(4,3)--(0,3);
    \draw (1,2)--(1,3);
    \draw (3,2)--(3,3);
    \draw (0,3)--(7,3)--(7,4)--(0,4);
    \draw (1,3)--(1,4);
    \draw (3,3)--(3,4);
    \draw (4,3)--(4,4);
    \draw (6,3)--(6,4);

    \draw (0,4)--(10,4)--(10,5)--(0,5);
    \draw (1,4)--(1,5);
    \draw (3,4)--(3,5);
    \draw (4,4)--(4,5);
    \draw (6,4)--(6,5);
    \draw (7,4)--(7,5);
    \draw (8,4)--(8,5);
    \node at (6,-1) {$\mathbb{D}(\mathbf{a})$};
\end{tikzpicture}}\quad\quad\begin{tikzpicture}
    \node at (0,0){};
    \node at (0,2) {$\rightarrow$};
\end{tikzpicture}\quad\quad\scalebox{0.75}{\begin{tikzpicture}[scale=0.65]
    \node at (0.5,3.5) {$\bigtimes$};
    \node at (3.5,3.5) {$\bigtimes$};
    \node at (5,3.5) {$\hdots$};
    \node at (6.5,3.5) {$\bigtimes$};
    \node at (7.5,3.5) {$\bigtimes$};
    \node at (6.5,4.5) {$\bigtimes$};
    \node at (-0.75,4.5) {$i+2$};
    \node at (-0.75,3.5) {$i+1$};
    \node at (-0.75,2.5) {$i$};
    \node at (2,3.5) {$\hdots$};

    \node at (0.5,4.5) {$\bigtimes$};
    \node at (3.5,4.5) {$\bigtimes$};
    \node at (5,4.5) {$\hdots$};
    \node at (6.5,4.5) {$\bigtimes$};
    \node at (2,4.5) {$\hdots$};
    \node at (9.5,4.5) {$\bigtimes$};
    \node at (8,4.5) {$\hdots$};
    
    \node at (0.5,2.5) {$\bigtimes$};
    \node at (2,2.5) {$\hdots$};
    \node at (3.5,2.5) {$\bigtimes$};
    \node at (9,2.5) {$\dagger~\dagger~\dagger$};
    \node at (5,7) {(key diagram)};
    \node at (5,1) {(key diagram)};
    \draw (0,8)--(0,0)--(12,0);
    \filldraw[draw=lightgray,fill=lightgray] (10,2) rectangle (11.5,3);
    \filldraw[draw=lightgray,fill=lightgray] (8,3) rectangle (11.5,4);
    \filldraw[draw=lightgray,fill=lightgray] (7,4) rectangle (11.5,5);
    \filldraw[draw=lightgray,fill=lightgray] (4,2) rectangle (8,3);
    \draw (0,2)--(4,2)--(4,3)--(0,3);
    \draw (1,2)--(1,3);
    \draw (3,2)--(3,3);
    \draw (0,3)--(7,3)--(7,4)--(0,4);
    \draw (1,3)--(1,4);
    \draw (3,3)--(3,4);
    \draw (4,3)--(4,4);
    \draw (6,3)--(6,4);
    \draw (7,3)--(8,3)--(8,4)--(7,4);
    \draw (8,2)--(10,2)--(10,3)--(8,3)--(8,2);

    \draw (0,4)--(7,4)--(7,5)--(0,5);
    \draw (1,4)--(1,5);
    \draw (3,4)--(3,5);
    \draw (4,4)--(4,5);
    \draw (6,4)--(6,5);
    \node at (6,-1) {$D$};
\end{tikzpicture}}$$
        \caption{Lemma~\ref{lem:keyconsec}~(a)}
        \label{fig:lemkeyconseca}
    \end{figure}

    Note that
    \begin{itemize}
        \item $(i+2,a_{i+1}),(i+1,a_{i+1}),(i+1,a_{i+1}+1)\in D$,
        \item $(i+2,\tilde{c})\notin D$ for $\tilde{c}>a_{i+1}$,
        \item $(i+1,\tilde{c})\notin D$ for $\tilde{c}>a_{i+1}+1$, and
        \item $(i,a_{i+1}),(i,a_{i+1}+1)\notin D$.
    \end{itemize}
    Thus, we apply Corollary~\ref{cor:ranked}~(a) with $r^*=i$, $c_1=a_{i+1}$, and $c_2=a_{i+1}+1$ to conclude that $\mathcal{P}(\mathbb{D}(\mathbf{a}))$ is not ranked.

    \item[(b)] Note that if $a_{i+2}>a_{i+1}$, then $\mathbf{a}$ is of the form considered in (a). There remain two cases depending on the relationship between $a_{i+1}$ and $a_{i+2}$.
    \bigskip

    \noindent
    \textbf{Case 1:} If $a_{i+2}=a_{i+1}$, then let $D$ denote the diagram formed from $\mathbb{D}(\mathbf{a})$ by applying a single Kohnert move at row $i+1$, i.e., $D= \mathbb{D}(\mathbf{a}) \ldownarrow^{(i+1,a_{i+1})}_{(i,a_{i+1})}$. Illustrations of the pertinent portions of (left) $\mathbb{D}(\mathbf{a})$ and (right) $D$ can be found in Figure~\ref{fig:lemkeyconsecb1}. 
    
    \begin{figure}[H]
        \centering
        $$\scalebox{0.75}{\begin{tikzpicture}[scale=0.65]
    \node at (0.5,3.5) {$\bigtimes$};
    \node at (3.5,3.5) {$\bigtimes$};
    \node at (5,3.5) {$\hdots$};
    \node at (6.5,3.5) {$\bigtimes$};
    \node at (7.5,3.5) {$\bigtimes$};
    \node at (-0.75,4.5) {$i+2$};
    \node at (-0.75,3.5) {$i+1$};
    \node at (-0.75,2.5) {$i$};
    \node at (2,3.5) {$\hdots$};

    \node at (0.5,4.5) {$\bigtimes$};
    \node at (3.5,4.5) {$\bigtimes$};
    \node at (5,4.5) {$\hdots$};
    \node at (6.5,4.5) {$\bigtimes$};
    \node at (2,4.5) {$\hdots$};
    \node at (7.5,4.5) {$\bigtimes$};
    
    \node at (0.5,2.5) {$\bigtimes$};
    \node at (2,2.5) {$\hdots$};
    \node at (3.5,2.5) {$\bigtimes$};
    \node at (5,7) {(key diagram)};
    \node at (5,1) {(key diagram)};
    \draw (0,8)--(0,0)--(10,0);
    \filldraw[draw=lightgray,fill=lightgray] (4,2) rectangle (9.5,3);
    \filldraw[draw=lightgray,fill=lightgray] (8,3) rectangle (9.5,4);
    \filldraw[draw=lightgray,fill=lightgray] (8,4) rectangle (9.5,5);
    \draw (0,2)--(4,2)--(4,3)--(0,3);
    \draw (1,2)--(1,3);
    \draw (3,2)--(3,3);
    \draw (0,3)--(8,3)--(8,4)--(0,4);
    \draw (1,3)--(1,4);
    \draw (3,3)--(3,4);
    \draw (4,3)--(4,4);
    \draw (6,3)--(6,4);
    \draw (7,3)--(7,4);

    \draw (0,4)--(8,4)--(8,5)--(0,5);
    \draw (1,4)--(1,5);
    \draw (3,4)--(3,5);
    \draw (4,4)--(4,5);
    \draw (6,4)--(6,5);
    \draw (7,4)--(7,5);
    \draw (8,4)--(8,5);
    \node at (5,-1) {$\mathbb{D}(\mathbf{a})$};
\end{tikzpicture}}\quad\quad\begin{tikzpicture}
    \node at (0,0){};
    \node at (0,2) {$\rightarrow$};
\end{tikzpicture}\quad\quad\scalebox{0.75}{\begin{tikzpicture}[scale=0.65]
    \node at (0.5,3.5) {$\bigtimes$};
    \node at (3.5,3.5) {$\bigtimes$};
    \node at (5,3.5) {$\hdots$};
    \node at (6.5,3.5) {$\bigtimes$};
    \node at (7.5,2.5) {$\bigtimes$};
    \node at (-0.75,4.5) {$i+2$};
    \node at (-0.75,3.5) {$i+1$};
    \node at (-0.75,2.5) {$i$};
    \node at (2,3.5) {$\hdots$};

    \node at (0.5,4.5) {$\bigtimes$};
    \node at (3.5,4.5) {$\bigtimes$};
    \node at (5,4.5) {$\hdots$};
    \node at (6.5,4.5) {$\bigtimes$};
    \node at (2,4.5) {$\hdots$};
    \node at (7.5,4.5) {$\bigtimes$};
    
    \node at (0.5,2.5) {$\bigtimes$};
    \node at (2,2.5) {$\hdots$};
    \node at (3.5,2.5) {$\bigtimes$};
    \node at (5,7) {(key diagram)};
    \node at (5,1) {(key diagram)};
    \draw (0,8)--(0,0)--(10,0);
    \filldraw[draw=lightgray,fill=lightgray] (4,2) rectangle (7,3);
    \filldraw[draw=lightgray,fill=lightgray] (8,2) rectangle (9.5,3);
    \filldraw[draw=lightgray,fill=lightgray] (7,3) rectangle (9.5,4);
    \filldraw[draw=lightgray,fill=lightgray] (8,4) rectangle (9.5,5);
    \draw (0,2)--(4,2)--(4,3)--(0,3);
    \draw (1,2)--(1,3);
    \draw (3,2)--(3,3);
    \draw (0,3)--(7,3)--(7,4)--(0,4);
    \draw (1,3)--(1,4);
    \draw (3,3)--(3,4);
    \draw (4,3)--(4,4);
    \draw (6,3)--(6,4);
    \draw (7,2)--(8,2)--(8,3)--(7,3)--(7,2);

    \draw (0,4)--(8,4)--(8,5)--(0,5);
    \draw (1,4)--(1,5);
    \draw (3,4)--(3,5);
    \draw (4,4)--(4,5);
    \draw (6,4)--(6,5);
    \draw (7,4)--(7,5);
    \draw (8,4)--(8,5);
    \node at (5,-1) {$D$};
\end{tikzpicture}}$$
        \caption{Lemma~\ref{lem:keyconsec}~(b) Case 1}
        \label{fig:lemkeyconsecb1}
    \end{figure}
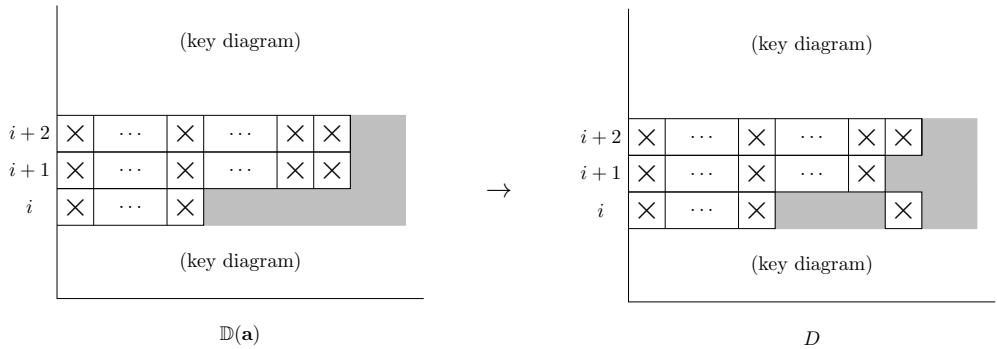

    Note that
    \begin{itemize}
        \item $(i+2,a_{i+2}),(i+2,a_{i+2}-1),(i+1,a_{i+2}-1),(i,a_{i+2})\in D$,
        \item $(i+2,\tilde{c})\notin D$ for $\tilde{c}>a_{i+2}$,
        \item $(i+1,\tilde{c})\notin D$ for $\tilde{c}>a_{i+2}-1$,
        \item $(i,\tilde{c})\notin D$ for $\tilde{c}>a_{i+2}$, and
        \item $(i,a_{i+2}-1)\notin D$.
    \end{itemize}
    Thus, we apply Corollary~\ref{cor:ranked}~(b) with $r^*=i$, $c_1=a_{i+2}-1$, and $c_2=a_{i+2}$ to conclude that $\mathcal{P}(\mathbb{D}(\mathbf{a}))$ is not ranked. 
    \bigskip

    \noindent
    \textbf{Case 2:} If $a_{i+2}<a_{i+1}$, then let $D$ denote the diagram formed from $\mathbb{D}(\mathbf{a})$ by applying $a_{i+1}-a_{i+2}-1$ Kohnert moves at row $i+1$; that is, $D$ is the diagram formed from $\mathbb{D}(\mathbf{a})$ by moving all but the leftmost cell in row $i+1$ which has no cell above it in row $i+2$ down to row $i$. Illustrations of the pertinent portions of (left) $\mathbb{D}(\mathbf{a})$ and (right) $D$ can be found in Figure~\ref{fig:lemkeyconsecb2}, where a region filled with $\dagger~\dagger~\dagger$ is either empty or filled with cells.

    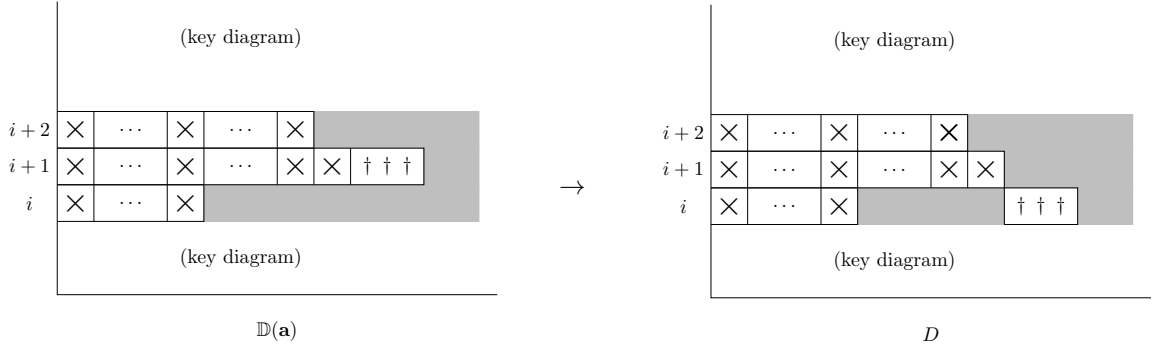
\begin{figure}[H]
        \centering
        $$\scalebox{0.75}{\begin{tikzpicture}[scale=0.65]
    \node at (0.5,4.5) {$\bigtimes$};
    \node at (3.5,4.5) {$\bigtimes$};
    \node at (5,4.5) {$\hdots$};
    \node at (6.5,4.5) {$\bigtimes$};
    \node at (-0.75,4.5) {$i+2$};
    \node at (-0.75,3.5) {$i+1$};
    \node at (-0.75,2.5) {$i$};
    \node at (2,4.5) {$\hdots$};

    \node at (0.5,3.5) {$\bigtimes$};
    \node at (3.5,3.5) {$\bigtimes$};
    \node at (5,3.5) {$\hdots$};
    \node at (6.5,3.5) {$\bigtimes$};
    \node at (2,3.5) {$\hdots$};
    \node at (7.5,3.5) {$\bigtimes$};
    \node at (9,3.5) {$\dagger~\dagger~\dagger$};
    
    \node at (0.5,2.5) {$\bigtimes$};
    \node at (2,2.5) {$\hdots$};
    \node at (3.5,2.5) {$\bigtimes$};
    \node at (5,7) {(key diagram)};
    \node at (5,1) {(key diagram)};
    \draw (0,8)--(0,0)--(12,0);
    \filldraw[draw=lightgray,fill=lightgray] (4,2) rectangle (11.5,3);
    \filldraw[draw=lightgray,fill=lightgray] (10,3) rectangle (11.5,4);
    \filldraw[draw=lightgray,fill=lightgray] (7,4) rectangle (11.5,5);
    \draw (0,2)--(4,2)--(4,3)--(0,3);
    \draw (1,2)--(1,3);
    \draw (3,2)--(3,3);
    \draw (0,4)--(7,4)--(7,5)--(0,5);
    \draw (1,4)--(1,5);
    \draw (3,4)--(3,5);
    \draw (4,4)--(4,5);
    \draw (6,4)--(6,5);

    \draw (0,3)--(10,3)--(10,4)--(0,4);
    \draw (1,3)--(1,4);
    \draw (3,3)--(3,4);
    \draw (4,3)--(4,4);
    \draw (6,3)--(6,4);
    \draw (7,3)--(7,4);
    \draw (8,3)--(8,4);
    \node at (6, -1) {$\mathbb{D}(\mathbf{a})$};
\end{tikzpicture}}
\quad\quad
\begin{tikzpicture}
    \node at (0,0){};
    \node at (0,2) {$\rightarrow$};
\end{tikzpicture}
\quad\quad\
\scalebox{0.75}{\begin{tikzpicture}[scale=0.65]
    \node at (0.5,3.5) {$\bigtimes$};
    \node at (3.5,3.5) {$\bigtimes$};
    \node at (5,3.5) {$\hdots$};
    \node at (6.5,3.5) {$\bigtimes$};
    \node at (7.5,3.5) {$\bigtimes$};
    \node at (6.5,4.5) {$\bigtimes$};
    \node at (-0.75,4.5) {$i+2$};
    \node at (-0.75,3.5) {$i+1$};
    \node at (-0.75,2.5) {$i$};
    \node at (2,3.5) {$\hdots$};

    \node at (0.5,4.5) {$\bigtimes$};
    \node at (3.5,4.5) {$\bigtimes$};
    \node at (5,4.5) {$\hdots$};
    \node at (6.5,4.5) {$\bigtimes$};
    \node at (2,4.5) {$\hdots$};
    \node at (9.5,4.5) {$\bigtimes$};
    \node at (8,4.5) {$\hdots$};
    
    \node at (0.5,2.5) {$\bigtimes$};
    \node at (2,2.5) {$\hdots$};
    \node at (3.5,2.5) {$\bigtimes$};
    \node at (9,2.5) {$\dagger~\dagger~\dagger$};
    \node at (5,7) {(key diagram)};
    \node at (5,1) {(key diagram)};
    \draw (0,8)--(0,0)--(12,0);
    \filldraw[draw=lightgray,fill=lightgray] (10,2) rectangle (11.5,3);
    \filldraw[draw=lightgray,fill=lightgray] (8,3) rectangle (11.5,4);
    \filldraw[draw=lightgray,fill=lightgray] (7,4) rectangle (11.5,5);
    \filldraw[draw=lightgray,fill=lightgray] (4,2) rectangle (8,3);
    \draw (0,2)--(4,2)--(4,3)--(0,3);
    \draw (1,2)--(1,3);
    \draw (3,2)--(3,3);
    \draw (0,3)--(7,3)--(7,4)--(0,4);
    \draw (1,3)--(1,4);
    \draw (3,3)--(3,4);
    \draw (4,3)--(4,4);
    \draw (6,3)--(6,4);
    \draw (7,3)--(8,3)--(8,4)--(7,4);
    \draw (8,2)--(10,2)--(10,3)--(8,3)--(8,2);

    \draw (0,4)--(7,4)--(7,5)--(0,5);
    \draw (1,4)--(1,5);
    \draw (3,4)--(3,5);
    \draw (4,4)--(4,5);
    \draw (6,4)--(6,5);
    \node at (6, -1) {$D$};
\end{tikzpicture}}$$
        \caption{Lemma~\ref{lem:keyconsec}~(b) Case 2}
        \label{fig:lemkeyconsecb2}
    \end{figure}

    Note that
    \begin{itemize}
        \item $(i+2,a_{i+2}),(i+1,a_{i+2}),(i+1,a_{i+2}+1)\in D$,
        \item $(i+2,\tilde{c})\notin D$ for $\tilde{c}>a_{i+2}$,
        \item $(i+1,\tilde{c})\notin D$ for $\tilde{c}>a_{i+2}+1$, and
        \item $(i,a_{i+2}),(i,a_{i+2}+1)\notin D$.
    \end{itemize}
     Thus, we apply Corollary~\ref{cor:ranked}~(a) with $r^*=i$, $c_1=a_{i+2}$, and $c_2=a_{i+2}+1$ to conclude that $\mathcal{P}(\mathbb{D}(\mathbf{a}))$ is not ranked.
\end{enumerate}

\end{proof}

Next, we show that the key diagram obtained by permuting a given row of a key diagram $\mathbb{D}(\mathbf{a})$ with a shorter row below is an element of $\mathcal{P}(\mathbb{D}(\mathbf{a}))$. To aid in exposition, we adopt the following notational convention. Given a weak composition $\mathbf{a}=(a_1,\hdots,a_n)$, we denote by $\mathbf{a}s_{i,j}$ the weak composition obtained from $\mathbf{a}$ by exchanging the entries $a_i$ and $a_j$; that is, $\mathbf{a}s_{i,j} =(a_1,\hdots,a_{i-1}, a_j,a_{i+1}, \hdots, a_{j-1}, a_i, a_{j+1}, \hdots, a_n).$

\begin{lemma}\label{lem:permute}
    Let $\mathbf{a}=(a_1,\hdots,a_n)$ be a weak composition. Suppose that there exist two indices $i,\ j$ such that $1\le i<j\le n$ and $a_i<a_j$. Then $\mathbb{D}(\mathbf{a}s_{i,j})\in KD(\mathbb{D}(\mathbf{a}))$. 
\end{lemma}
\begin{proof}
Intuitively, the idea is that for each additional cell in row $j$, working from right to left, we can apply Kohnert moves to drop the cell through the empty positions below until it reaches row $i$.

To formalize the intuitive idea outlined above, for $l$ satisfying $a_i<l\le a_j$, let $$\{(r,l)~|~i<r<j,~(r,l)\notin \mathbb{D}(\mathbf{a})\}=\{r^l_1<\hdots<r^l_{k_l}\},$$ i.e., the rows $r^l_1,\hdots,r^l_{k_l}$ constitute exactly the rows containing empty positions in column $l$ strictly between rows $i$ and $j$ in $\mathbb{D}(\mathbf{a})$. Now, working from $l=a_j$ down to $a_{i}+1$ in decreasing order, successively apply a single Kohnert move at row $j$ followed by rows $r^l_{k_l}$ down to $r^l_1$ in decreasing order. It is straightforward to verify that the aforementioned formal procedure accomplishes the desired effect.
\end{proof}

We are now ready to prove Theorem~\ref{thm:keynecrank}. 

\begin{proof}[Proof of Theorem~\ref{thm:keynecrank}]
We consider non-pure compositions $\mathbf{a}$ containing patterns of the form (a) and (b) separately. In each case, we show that there exists a key diagram $T\in KD(\mathbb{D}(\mathbf{a}))$ satisfying the hypotheses of Lemma~\ref{lem:keyconsec} so that, considering the proof of Lemma~\ref{lem:keyconsec}, Corollary~\ref{cor:ranked} applies and $\mathcal{P}(\mathbb{D}(\mathbf{a}))$ is not ranked.
Also, in both cases, we assume without loss of generality that $j_1$ is maximal, while $j_2$ is minimal given our choice of $j_1$, and $j_3$ is minimal given our choice of $j_2$.
\begin{enumerate}
\item[(a)] Note that, by our assumption on $j_1$, for all $j_1<i<j_2$ we have $a_i>a_{j_1}$. Thus, applying Lemma~\ref{lem:permute}, $\mathbb{D}(\mathbf{a}s_{j_1,j_2-1})\in KD(\mathbb{D}(\mathbf{a}))$. Now, by our assumption on $j_3$, for all $j_2<i<j_3$ we have $a_i<a_{j_3}$. Therefore, applying Lemma~\ref{lem:permute} again, we find that $\mathbb{D}(\mathbf{a}s_{j_1,j_2-1}s_{j_2+1,j_3})\in KD(\mathbb{D}(\mathbf{a}))$. As the values $a_{j_1},a_{j_2},a_{j_3}$ are consecutive and occur in the original ordering in $\mathbf{a}s_{j_1,j_2-1}s_{j_2+1,j_3}$, the result follows as discussed above.

\item[(b)] We need only consider the case not covered by (a). Thus, there are no indices $1\le i_1<i_2<i_3\le n$ such that $a_{i_1}<a_{i_2}<a_{i_3}$. Consequently, $a_{j_1}+1<a_{j_2}$ and $a_{j_1}<a_{j_3}\le a_{j_2}$. Note that, by our assumption on $j_1$, for all $j_1<i<j_2$ we have $a_i>a_{j_1}$. Thus, applying Lemma~\ref{lem:permute}, $\mathbb{D}(\mathbf{a}s_{j_1,j_2-1})\in KD(\mathbb{D}(\mathbf{a}))$. Now, by our assumption on $j_3$, for all $j_2<i<j_3$ we have $a_i\le a_{j_1}<a_{j_3}$. Therefore, applying Lemma~\ref{lem:permute} again, we find that $\mathbb{D}(\mathbf{a}s_{j_1,j_2-1}s_{j_2+1,j_3})\in KD(\mathbb{D}(\mathbf{a}))$. As the values $a_{j_1},a_{j_2},a_{j_3}$ are consecutive and occur in the original ordering in $\mathbf{a}s_{j_1,j_2-1}s_{j_2+1,j_3}$, the result follows as discussed above.
\end{enumerate}
\end{proof}

\subsection{Proof of the sufficient conditions}

In this section, we finish the proof of Theorem~\ref{thm:keyrank}.

To complete the proof of Theorem~\ref{thm:keyrank}, it remains to show that if $\mathbf{a}$ is a pure composition, then $\mathcal{P}(\mathbb{D}(\mathbf{a}))$ is ranked. We accomplish this in Lemma~\ref{lem:keysuff} below by showing that a rank function is provided by $rk(D)=rowsum(D)-b(\mathbb{D}(\mathbf{a}))$ for $D\in KD(\mathbb{D}(\mathbf{a}))$.

\begin{remark}
    Note that for a key diagram $\mathbb{D}(\mathbf{a})$, we have $b(\mathbb{D}(\mathbf{a}))=rowsum(\mathbb{D}(\texttt{sort}(\mathbf{a})))$.
\end{remark}

\begin{lemma}\label{lem:keysuff}
    If $\mathbf{a}$ is a pure composition, then $\mathcal{P}(\mathbb{D}(\mathbf{a}))$ is ranked. Moreover, for $D\in \mathbb{D}(\mathbf{a})$, a rank function is given by $rk(D)=rowsum(D)-b(\mathbb{D}(\mathbf{a}))$.
\end{lemma}

\begin{proof}
Note that, by the definition of Kohnert move, if $D_1,D_2\in KD(\mathbb{D}(\mathbf{a}))$ with $D_2\prec D_1$, then $rk(D_2)<rk(D_1)$ since $rowsum(D_2) < rowsum(D_1)$. Moreover, by the definition of the rank function, $rk(D)=0$ for $D\in KD(\mathbb{D}(\mathbf{a}))$ if and only if $D=\mathbb{D}(\texttt{sort}(\mathbf{a}))$; and $rk(D)>0$ for all other $D\in KD(\mathbb{D}(\mathbf{a}))$.

Now, it remains to show that if $D_1,D_2\in KD(\mathbb{D}(\mathbf{a}))$ satisfy $D_2\precdot D_1$, then $rk(D_2)=rk(D_1)-1$. For a contradiction, assume that $rk(D_2)<rk(D_1)-1$. Thus, there exist some $r>1$ and $1< k\le r-1$ such that $D_2$ is formed from $D_1$ by applying a single Kohnert move at row $r$ which results in the cell in position $(r,c)$ moving to position $(r-k,c)$, i.e., $D_2 = D_1 \ldownarrow^{(r,c)}_{(r-k,c)}$. Note that, considering the definition of Kohnert move, this implies that $(\tilde{r},c)\in D_1$ for all $r-k<\tilde{r}<r$. We claim that there exists a cell $(r',c)\in D_1$ with $r-k<r'<r$ which is rightmost in its row. If true, then the claim would imply that $D_2$ can also be formed from $D_1$ by applying a Kohnert move at row $r'$ followed by another at row $r$, contradicting our assumption that $D_2\precdot D_1$.

To establish the claim, we show that each cell $(\tilde{r},c)\in D_1$ for $r-k<\tilde{r}<r$ is, in fact, rightmost in its row. Assume otherwise. Then there exist $r'$ with $r-k<r'<r$ and $c'>c$ such that $(r,c),(r',c),(r',c')\in D_1$; that is, since $(r-k,c)\notin D_1$, the diagram $D_1$ has two cells in column $c$ with a common empty position below and a column $c'$ to the right of $c$ which contains a cell in the same row as the lower of the two aforementioned cells of column $c$. Now, in Lemma~\ref{lem:purecomp} we showed that $\mathbf{a}$ can be partitioned into subcompositions $\alpha_j$ for some $1\le j\le m$ where each is of one of four types and $\min(\alpha_j)\ge \max(\alpha_{j+1})$ for $1\le j<m$. Note that $\min(\alpha_j)\ge \max(\alpha_{j+1})$ for $1\le j<m$ implies that $\min(\alpha_k)\ge \max(\alpha_{j+1})$ for all $1\le k\le j\le m$. Consequently, at the diagram level, the cells corresponding to a given $\alpha_j$ for $1\le j\le m$ cannot be moved by Kohnert moves into rows corresponding to any other $\alpha_k$ with $1\le k<j$ as all such rows contain at least $\max(\alpha_j)$ left-justified cells, i.e., there are no empty positions in the appropriate columns to move into. Thus, the three cells $(r,c),(r',c),(r',c')\in D_1$ must all originally belong to rows of $\mathbb{D}(\mathbf{a})$ corresponding to a single $\alpha_j$ of one of the four types listed in Lemma~\ref{lem:purecomp}. However, only weak compositions of type (iv) have diagrams for which there is a column containing more than one cell sharing an empty space below, and all columns to the right must be empty. Thus, there cannot exist $c'>c$ satisfying $(r',c')\in D_1$, and the claim follows.
\end{proof}

\section{Checkered diagrams}\label{sec:chd}

In this section, we consider Kohnert posets associated with checkered diagrams, which resemble chess boards.  There are two types of checkered diagrams, $Ch_n^1$ and $Ch_n^2$ for $n\ge 1$, depending on whether the bottom left position contains a cell or not. 

Formally, we define the two families of checkered diagrams as follows.
\begin{align*} 
Ch_n^1&=\left\{(2i-1,2j-1)~\left|~1\le i,j\le \left\lceil\frac{n}{2}\right\rceil\right.\right\}
\cup
\left\{(2i,2j)~\left|~1\le i,j\le \left\lfloor\frac{n}{2}\right\rfloor\right.\right\}, \\
Ch_n^2&=\left\{(2i-1,2j)~\left|~1\le i\le \left\lceil\frac{n}{2}\right\rceil\right.,~ 
1\le j\le \left\lfloor\frac{n}{2}\right\rfloor\right\}
\cup
\left\{(2i,2j-1)~\left|~1\le i\le \left\lfloor\frac{n}{2}\right\rfloor\right.,~ 
1\le j \le \left\lceil \frac{n}{2} \right\rceil\right\}.
\end{align*}
\noindent
In Figure~\ref{fig:checkered}, we illustrate $Ch^1_n$ and $Ch^2_n$ for $n=1,2,3,$ and $4$. 

\begin{figure}[H]
    \centering
    $$\begin{tikzpicture}[scale=0.4]
  \node at (0.5, 0.5) {$\bigtimes$};
  \draw (0,2)--(0,0)--(2,0);
  \draw (0,1)--(1,1)--(1,0);
  \node at (1, -1) {$Ch^1_1$};
\end{tikzpicture}
\quad\quad\quad\quad
\begin{tikzpicture}[scale=0.4]
  \node at (1.5, 1.5) {$\bigtimes$};
  \node at (0.5, 0.5) {$\bigtimes$};
  \draw (0,3)--(0,0)--(3,0);
  \draw (1,0)--(1,1)--(0,1);
  \draw (1,1)--(2,1)--(2,2)--(1,2)--(1,1);
  \node at (1.5, -1) {$Ch^1_2$};
\end{tikzpicture}
\quad\quad\quad\quad
\begin{tikzpicture}[scale=0.4]
  \node at (1.5, 1.5) {$\bigtimes$};
  \node at (0.5, 0.5) {$\bigtimes$};
  \node at (0.5, 2.5) {$\bigtimes$};
  \node at (2.5, 0.5) {$\bigtimes$};
  \node at (2.5, 2.5) {$\bigtimes$};
  \draw (0,4)--(0,0)--(4,0);
  \draw (1,0)--(1,1)--(0,1);
  \draw (0,2)--(1,2)--(1,3)--(0,3);
  \draw (1,1)--(2,1)--(2,2)--(1,2)--(1,1);
  \draw (3,0)--(3,1)--(2,1)--(2,0);
  \draw (2,2)--(3,2)--(3,3)--(2,3)--(2,2);
  \node at (2, -1) {$Ch^1_3$};
\end{tikzpicture}
\quad\quad\quad\quad
\begin{tikzpicture}[scale=0.4]
  \node at (1.5, 1.5) {$\bigtimes$};
  \node at (1.5, 3.5) {$\bigtimes$};
  \node at (0.5, 0.5) {$\bigtimes$};
  \node at (0.5, 2.5) {$\bigtimes$};
  \node at (3.5, 1.5) {$\bigtimes$};
  \node at (3.5, 3.5) {$\bigtimes$};
  \node at (2.5, 0.5) {$\bigtimes$};
  \node at (2.5, 2.5) {$\bigtimes$};
  \draw (0,5)--(0,0)--(5,0);
  \draw (1,0)--(1,1)--(0,1);
  \draw (0,2)--(1,2)--(1,3)--(0,3);
  \draw (1,1)--(2,1)--(2,2)--(1,2)--(1,1);
  \draw (1,3)--(2,3)--(2,4)--(1,4)--(1,3);
  \draw (3,0)--(3,1)--(2,1)--(2,0);
  \draw (2,2)--(3,2)--(3,3)--(2,3)--(2,2);
  \draw (3,1)--(4,1)--(4,2)--(3,2)--(3,1);
  \draw (3,3)--(4,3)--(4,4)--(3,4)--(3,3);
  \node at (2.5, -1) {$Ch^1_4$};
\end{tikzpicture}$$
$$\begin{tikzpicture}[scale=0.4]
  \draw (0,2)--(0,0)--(2,0);
  \node at (1, -1) {$Ch^2_1$};
\end{tikzpicture}
\quad\quad\quad\quad
\begin{tikzpicture}[scale=0.4]
  \node at (0.5, 1.5) {$\bigtimes$};
  \node at (1.5, 0.5) {$\bigtimes$};
  \draw (0,3)--(0,0)--(3,0);
  \draw (0,1)--(1,1)--(1,2)--(0,2);
  \draw (1,0)--(1,1)--(2,1)--(2,0);
  \node at (1.5, -1) {$Ch^2_2$};
\end{tikzpicture}
\quad\quad\quad\quad
\begin{tikzpicture}[scale=0.4]
  \node at (0.5, 1.5) {$\bigtimes$};
  \node at (1.5, 0.5) {$\bigtimes$};
  \node at (1.5, 2.5) {$\bigtimes$};
  \node at (2.5, 1.5) {$\bigtimes$};
  \draw (0,4)--(0,0)--(4,0);
  \draw (0,1)--(1,1)--(1,2)--(0,2);
  \draw (1,0)--(1,1)--(2,1)--(2,0);
  \draw (1,2)--(2,2)--(2,3)--(1,3)--(1,2);
  \draw (2,1)--(3,1)--(3,2)--(2,2)--(2,1);
  \node at (2, -1) {$Ch^2_3$};
\end{tikzpicture}
\quad\quad\quad\quad
\begin{tikzpicture}[scale=0.4]
  \node at (0.5, 1.5) {$\bigtimes$};
  \node at (0.5, 3.5) {$\bigtimes$};
  \node at (1.5, 0.5) {$\bigtimes$};
  \node at (1.5, 2.5) {$\bigtimes$};
  \node at (2.5, 1.5) {$\bigtimes$};
  \node at (2.5, 3.5) {$\bigtimes$};
  \node at (3.5, 0.5) {$\bigtimes$};
  \node at (3.5, 2.5) {$\bigtimes$};
  \draw (0,5)--(0,0)--(5,0);
  \draw (0,1)--(1,1)--(1,2)--(0,2);
  \draw (0,3)--(1,3)--(1,4)--(0,4);
  \draw (1,0)--(1,1)--(2,1)--(2,0);
  \draw (1,2)--(2,2)--(2,3)--(1,3)--(1,2);
  \draw (2,1)--(3,1)--(3,2)--(2,2)--(2,1);
  \draw (2,3)--(3,3)--(3,4)--(2,4)--(2,3);
  \draw (3,0)--(3,1)--(4,1)--(4,0);
  \draw (3,2)--(4,2)--(4,3)--(3,3)--(3,2);
  \node at (2.5, -1) {$Ch^2_4$};
\end{tikzpicture}$$
    \caption{$Ch^1_n$ and $Ch^2_n$ for $n=1,2,3,$ and $4$}
    \label{fig:checkered}
\end{figure}
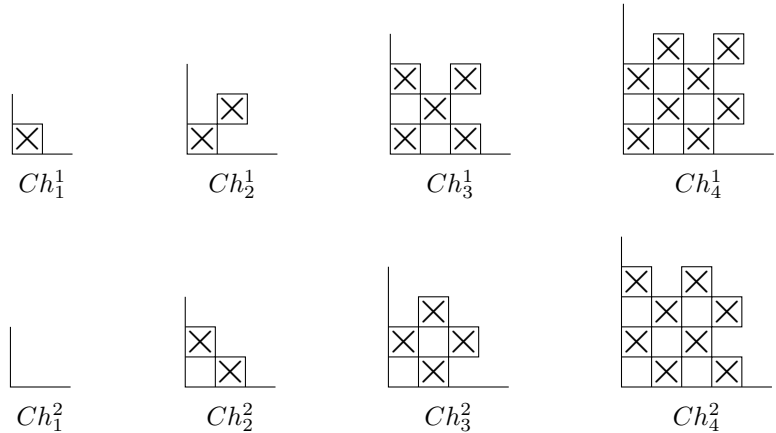

\noindent
As we will see, the results of this section concerning the poset properties of $Ch_n^1$ and $Ch_n^2$ depend only on the parity of $n$ and not the particular family of checkered diagrams. Consequently, to facilitate discussion, throughout the remainder of this section, for $n\ge 1$, we let $D_0(n)$ denote the initial diagram $Ch^1_n$ or $Ch^2_n$.

Unlike in previous sections, for Kohnert posets of checkered diagrams, we begin by considering the property of being ranked, as the characterization is more straightforward compared to the study of minimal elements.

\begin{theorem}\label{thm:checkrank}
The poset $\mathcal{P}(D_0(n))$ is ranked if and only if $n\le 3$. Moreover, $\mathcal{P}(D_0(n))$ is ranked and bounded if and only if $n\le 2$.
\end{theorem}
\begin{proof}
For the backward direction, it can be checked directly that for $n=1,2,$ and $3$ the poset $\mathcal{P}(D_0(n))$ is ranked. The poset $Ch_1^2$ is empty, while the Hasse diagrams of the remaining posets are illustrated in Figure~\ref{fig:checkHasse}.

\begin{figure}[H]
    \centering
    $$\begin{tikzpicture}
	\node (1) at (0, 0) [circle, draw = black, fill = black, inner sep = 0.5mm]{};
	\node (2) at (0, -1) {$\mathcal{P}(Ch_1^1)$};
\end{tikzpicture}\quad\quad \begin{tikzpicture}
    \node (2) at (0, 0.75) [circle, draw = black, fill = black, inner sep = 0.5mm]{};
	\node (1) at (0, 0) [circle, draw = black, fill = black, inner sep = 0.5mm]{};
    \draw (1)--(2);
	\node (3) at (0, -1) {$\mathcal{P}(Ch_2^1)=\mathcal{P}(Ch_2^2)$};
\end{tikzpicture}\quad\quad \begin{tikzpicture}
	\node (1) at (-0.5, 0) [circle, draw = black, fill = black, inner sep = 0.5mm]{};
    \node (2) at (-0.5, 0.75) [circle, draw = black, fill = black, inner sep = 0.5mm]{};
    \node (3) at (-0.5, 1.5) [circle, draw = black, fill = black, inner sep = 0.5mm]{};
    \node (4) at (0.5, 0.75) [circle, draw = black, fill = black, inner sep = 0.5mm]{};
    \node (5) at (0.5, 1.5) [circle, draw = black, fill = black, inner sep = 0.5mm]{};
    \node (6) at (0, 2) [circle, draw = black, fill = black, inner sep = 0.5mm]{};
    \draw (1)--(2);
	\node (7) at (0, -1) {$\mathcal{P}(Ch_3^1)$};
    \draw (1)--(2)--(3)--(6)--(5)--(4);
\end{tikzpicture}\quad\quad\quad \begin{tikzpicture}
	\node (1) at (0.5, 0) [circle, draw = black, fill = black, inner sep = 0.5mm]{};
    \node (2) at (0.5, 0.75) [circle, draw = black, fill = black, inner sep = 0.5mm]{};
    \node (3) at (0.5, 1.5) [circle, draw = black, fill = black, inner sep = 0.5mm]{};
    \node (4) at (-0.5, 0.75) [circle, draw = black, fill = black, inner sep = 0.5mm]{};
    \node (5) at (-0.5, 1.5) [circle, draw = black, fill = black, inner sep = 0.5mm]{};
    \node (6) at (0, 2) [circle, draw = black, fill = black, inner sep = 0.5mm]{};
    \draw (1)--(2);
	\node (7) at (0, -1) {$\mathcal{P}(Ch_3^2)$};
    \draw (1)--(2)--(3)--(6)--(5)--(4)--(3);
\end{tikzpicture}$$
    \caption{Hasse diagrams of $\mathcal{P}(D_0(n))$ for $n=1,2,$ and $3$}
    \label{fig:checkHasse}
\end{figure}
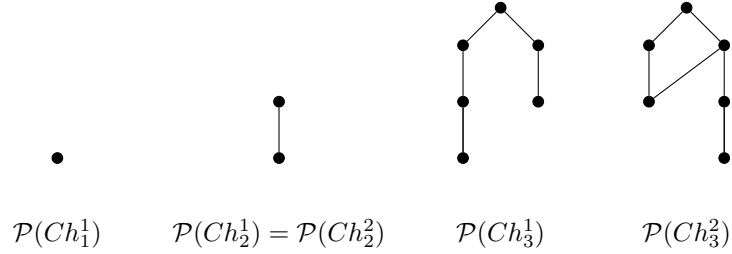

Now, for the forward direction, assuming $n\ge 3$, we show that we are under the assumptions of Corollary~\ref{cor:ranked}~(b). If $D_0(n)=Ch^1_n$ with $n$ odd or $D_0(n)=Ch^2_n$ with $n$ even, then consider the diagram $D$ obtained from $D_0(n)$ by applying a Kohnert move at row 4.  Note that $(2,n-1),(3,n-1),(1,n),(3,n)\in D$ and $(1,n-1),(2,n),(r,c)\notin D$ for $r\ge 1$ and $c>n$. Thus, we apply Corollary~\ref{cor:ranked}~(b) with $r^*=1$, $c_1=n-1$, and $c_2=n$ to conclude that $\mathcal{P}(D_0(n))$ is not ranked. On the other hand, if $D_0(n)=Ch^1_n$ with $n$ even or $D_0(n)=Ch^2_n$ with $n$ odd, then consider the diagram $D'$ obtained from $D_0(n)$ by applying a  Kohnert move at row 2 followed by two Kohnert moves at row 4. Note that $(2,n-2),(3,n-2),(1,n),(3,n)\in D'$ and $(1,n-2),(2,n-1),(2,n),(r,c)\notin D'$ for $r\ge 1$ and $c>n$. Thus, we apply Corollary~\ref{cor:ranked}~(b) with $r^*=1$, $c_1=n-2$, and $c_2=n$ to conclude that $\mathcal{P}(D_0(n))$ is not ranked.

Thus, $\mathcal{P}(D_0(n))$ is ranked if and only if $n\le 3$. Moreover, considering Figure~\ref{fig:checkHasse}, we see that $\mathcal{P}(D_0(n))$ is ranked and bounded if and only if $n\le 2$.
\end{proof}

Moving to boundedness, the number of minimal elements of $\mathcal{P}(D_0(n))$ depends on the parity of $n$, and, in both cases, the formula takes a simple form. 
\begin{theorem}\label{thm:noddcheckered}
    \begin{enumerate}
        \item[\textup{(a)}] If $n$ is even, then $|Min(D_0(n))|=1$.
        \item[\textup{(b)}] If $n$ is odd, then $|Min(D_0(n))|=\binom{2m}{m}$, where $m=\lfloor\frac{n}{2}\rfloor$. 
    \end{enumerate}
\end{theorem}

When $n$ is even, the above result follows as a direct consequence of Proposition~\ref{prop:nummingen} since each nonempty column has exactly $\dfrac{n}{2}$ cells in both cases. 

On the other hand, when $n$ is odd, despite being an elegant result, the proof requires a more detailed study. The rest of the section is dedicated to proving this result. The idea behind the proof is to give a bijection between the minimal elements of $\mathcal{P}(D_0(n))$ and Kohnert diagrams associated with the key diagram defined by the composition 
$$(0,m^m)=(0,\underbrace{m,\hdots,m}_m) \qquad \text{ for } m=\left\lfloor\frac{n}{2}\right\rfloor.$$ 
Ongoing, we set $\mathbb{D}_m=\mathbb{D}((0,m^m))$, which we illustrate for $m=1,2,$ and $3$ in Figure~\ref{fig:squares}.

\begin{figure}[H]
    \centering
    $$\begin{tikzpicture}[scale=0.4]
  \node at (0.5, 1.5) {$\bigtimes$};
  \draw (0,3)--(0,0)--(2,0);
  \draw (0,2)--(1,2)--(1,1)--(0,1);
  \node at (1, -1) {$\mathbb{D}_1$};
\end{tikzpicture}
\quad\quad\quad\quad
\begin{tikzpicture}[scale=0.4]
  \node at (1.5, 1.5) {$\bigtimes$};
  \node at (0.5, 1.5) {$\bigtimes$};
  \node at (1.5, 2.5) {$\bigtimes$};
  \node at (0.5, 2.5) {$\bigtimes$};
  \draw (0,4)--(0,0)--(3,0);
  \draw (0,1)--(1,1);
  \draw (0,2)--(1,2);
  \draw (1,1)--(2,1)--(2,2)--(1,2)--(1,1);
  \draw (0,3)--(2,3)--(2,2);
  \draw (1,3)--(1,2);
  \node at (1.5, -1) {$\mathbb{D}_2$};
\end{tikzpicture}
\quad\quad\quad\quad
\begin{tikzpicture}[scale=0.4]
  \node at (2.5, 1.5) {$\bigtimes$};
  \node at (1.5, 1.5) {$\bigtimes$};
  \node at (0.5, 1.5) {$\bigtimes$};
  \node at (2.5, 2.5) {$\bigtimes$};
  \node at (1.5, 2.5) {$\bigtimes$};
  \node at (0.5, 2.5) {$\bigtimes$};
  \node at (2.5, 3.5) {$\bigtimes$};
  \node at (1.5, 3.5) {$\bigtimes$};
  \node at (0.5, 3.5) {$\bigtimes$};
  \draw (0,5)--(0,0)--(4,0);
  \draw (0,1)--(1,1);
  \draw (0,2)--(1,2);
  \draw (1,1)--(2,1)--(2,2)--(1,2)--(1,1);
  \draw (0,3)--(2,3)--(2,2);
  \draw (1,3)--(1,2);
  \draw (2,1)--(3,1)--(3,4)--(0,4);
  \draw (1,3)--(1,4);
  \draw (2,3)--(2,4);
  \draw (3,3)--(3,4);
  \draw (2,2)--(3,2);
  \draw (2,3)--(3,3);
  \node at (1.5, -1) {$\mathbb{D}_3$};
\end{tikzpicture}$$
    \caption{$\mathbb{D}_m$ for $m=1,2,3$}
    \label{fig:squares}
\end{figure}
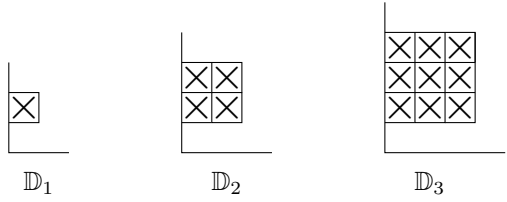

To start, we determine the number of Kohnert diagrams associated with $\mathbb{D}_m$, i.e., we compute $|KD(\mathbb{D}_m)|$. Let's introduce some notation. Consider a diagram $T$ consisting of $m^2$ cells contained in rows $1,2,\ldots, m+1$ and columns $1,2,\ldots,m$ with $m$ cells in each nonempty column for some $m\ge 1$. Note that each nonempty column of $T$ has exactly one empty position contained in rows 1 through $m+1$, and, consequently, we can encode the diagram $T$ as follows. We define the \textbf{empty-row sequence} of $T$ as the sequence 
$\texttt{er}(T):=(r_1,\dots, r_m)$ where $1\le r_i\le m+1$, $1\le i\le m$, and $(r_i,i)\notin T$. Then $T$ is completely determined by $\texttt{er}(T)$. In the particular case of $T \in KD(\mathbb{D}_m)$ for $m\ge 1$, we have the following characterization.

\begin{lemma}\label{lem:sqform}
For $m\ge 1$, let $T$ be a diagram consisting of $m^2$ cells contained in rows $1,2,\ldots, m+1$ and columns $1,2,\ldots,m$ with $m$ cells in each nonempty column. Then $T\in KD(\mathbb{D}_m)$ if and only if the sequence $\texttt{er}(T)$ is weakly increasing.
\end{lemma}
\begin{proof}
For the forward direction, note that for the maximal element $\mathbb{D}_m$ of $\mathcal{P}(\mathbb{D}_m)$, the empty-row sequence $\texttt{er}(\mathbb{D}_m)=(1,\dots,1)$ is weakly increasing. Thus, it suffices to show that for diagrams of $KD(\mathbb{D}_m)$ the property of having a weakly increasing empty-row sequence is preserved by Kohnert moves. Take $T_1\in KD(\mathbb{D}_m)$ such that the sequence $\texttt{er}(T_1)$ is weakly increasing and assume that $T_2\in KD(\mathbb{D}_m)$ can be obtained from $T_1$ by applying a single Kohnert move at row $r$, say $T_2 = T_1\ldownarrow^{(r,c)}_{(r',c)}$. Note that, by the definition of Kohnert move, $r'<r$ and $(r,\tilde{c})\notin T_1$ for $c<\tilde{c}\le m$. Consequently, the empty-row sequence of $T_1$ is of the form $$\texttt{er}(T_1) = (r_1, \hdots , r_{c-1} , r' , r , \hdots , r),$$ where $r_1 \le \hdots \le r_{c-1} \le r' < r$, and the empty-row sequence of $T_2$ is of the form $$\texttt{er}(T_2) = (r_1, \hdots, r_{c-1}, r, r, \hdots, r),$$ where $r_1 \le \dots \le r_{c-1} < r$; that is $\texttt{er}(T_2)$ is weakly increasing and the forward direction follows. 

Now, for the backward direction, suppose we have a weakly increasing sequence $(r_1,\dots,r_m)$, where $1\le r_i\le m+1$ for $1\le i\le m$. Starting from $\mathbb{D}_m$, we can construct $T\in KD(\mathbb{D}_m)$ such that $\texttt{er}(T) = (r_1,\dots,r_m)$ using the following procedure:
\begin{enumerate}
    \item[] Step 1: Apply $m$ Kohnert moves at row $r_1$. 
    \item[] Step $j$: If $r_j = r_{j-1}$, then continue; otherwise, apply $n-j+1$ Kohnert moves at row $r_j$. 
\end{enumerate}
The result follows.
\end{proof}

Using Lemma~\ref{lem:sqform}, we are able to compute $|KD(\mathbb{D}_m)|$.

\begin{corollary}\label{cor:numberdiagDm}
For $m\ge 1$, $|KD(\mathbb{D}_m)|=\binom{2m}{m}$.
\end{corollary}
\begin{proof}
By Lemma~\ref{lem:sqform}, $\texttt{er}(T)$ defines a bijection between diagrams $T\in KD(\mathbb{D}_m)$ and weakly increasing sequences $(r_1,\dots,r_m)$ where $1\le r_i\le m+1$ for $1\le i\le m$. Decreasing each entry in such a sequence by one results in a weakly increasing sequence $(r_1,\dots,r_m)$ where $0\le r_i\le m$ for $1\le i\le m$, i.e., an integer partition contained inside of an $m\times m$ rectangle. As the number of such integer partitions is well-known to be $\binom{2m}{m}$ (see \cite[Proposition 1.7.3]{EC1}), the result follows.
\end{proof}

Returning to diagrams $D_0(n)$ for $n$ odd, the following result allows us to relate diagrams of $Min(D_0(n))$ with diagrams of $KD(\mathbb{D}_m)$ for $m=\lfloor\frac{n}{2}\rfloor$. In particular, combining Lemmas~\ref{lem:sqform} and~\ref{lem:Ch}, it follows that removing certain columns of $T\in Min(D_0(n))$ results in a unique $T'\in KD(\mathbb{D}_m)$ for $m=\lfloor\frac{n}{2}\rfloor$. Consequently, considering Corollary~\ref{cor:numberdiagDm}, it is enough to show that the aforementioned map is a bijection to establish Theorem~\ref{thm:noddcheckered} when $n$ is odd. Before stating the technical result, we introduce some notation. Given a diagram $T$, let $r_{i,j}(T)$ denote the row occupied by the $j^{th}$ cell from bottom to top in column $i$ of $T$. 

\begin{lemma}\label{lem:Ch}
Consider a diagram $T\in KD(D_0(n))$ with $n$ odd, and a cell $(r_{i,j}(T),i)\in T$ for some $1\le i \le n-2$ and $j\ge 1$. Then the number of cells weakly below row $r_{i,j}(T)$ in column $i+2$ of $T$ is at least $j$.
\end{lemma}

\begin{proof} 
Note that the diagram $D_0(n)$ satisfies the property stated in the lemma, which we refer to as \textit{Property $(\ast)$}. For a contradiction, suppose that there exists a diagram $D\in KD(D_0(n))$ not satisfying Property $(\ast)$. Considering that $D_0(n)$ has Property $(\ast)$, we assume without loss of generality that there exists another diagram $D'\in KD(D_0(n))$ such that $D'$ has Property $(\ast)$ and $D$ is obtained from $D'$ by applying a Kohnert move, say at row $r$. 

Since $D$ does not have Property $(\ast)$, there exists a column $1\le i\le n-2$ and a $j\ge 1$ such that $(r_{i,j}(D),i)\in D$, but $D$ has strictly less than $j$ cells weakly below row $r_{i,j}(D)$ in column $i+2$. Moreover, since $D$ is obtained from $D'$ by applying a Kohnert move at row $r$ and $D'$ has the Property $(\ast)$, that Kohnert move must affect a cell in column $i$; that is, $D = D' \ldownarrow^{(r,i)}_{(r',i)}$ with $r'<r$. To see this, consider the following two observations.
\begin{itemize}
    \item[1)] Suppose there is a cell in a column other than $i$ or $i+2$ that is moved to a strictly lower row in forming $D$ from $D'$. Then the cells in columns $i$ and $i+2$ are unchanged and, thus, Property $(\ast)$ would be preserved in these columns.
    \item[2)] Suppose there is a cell in column $i+2$ that is moved to a strictly lower row in forming $D$ from $D'$. Then the number of cells weakly below a given cell in column $i$ can only increase, which means that Property $(\ast)$ would once again be preserved in these columns.
\end{itemize}
Now, since $D = D' \ldownarrow^{(r,i)}_{(r',i)}$, we have that  $(r'',i)\in D'\cap D$ for $r'< r''<r$. By our assumption, there exists $r^{\ast}$ such that $(r^{\ast},i)=(r_{i,j^{\ast}}(D),i)\in D$ and there are less than $j^{\ast}$ cells weakly below this cell in column $i+2$. Note that $r'\le r^{\ast}<r$, since a cell in position $$(r'',i)=(r_{i,j''}(D),i)=(r_{i,j''}(D'),i)$$ for $1\le r''<r'$ or $r<r''$ has the same number of cells weakly below it in column $i+2$ in $D$ as in $D'$. Now, since $D'$ satisfies Property $(\ast)$, $(r^*,i)=(r_{i,j^{\ast}-1}(D'),i)\in D'$ must have at least $j^{\ast}-1$ cells weakly below it in column $i+2$. As $(r^*,i)=(r_{i,j^{\ast}}(D),i)\in D$ has less than $j^{\ast}$ cells weakly below it in column $i$, it follows that $(r^*,i)=(r_{i,j^{\ast}-1}(D'),i)\in D'$ must have exactly $j^{\ast}-1$ cells weakly below it in column $i+2$. Thus, $(r'',i)=(r_{i,j''}(D'),i)\in D'$ for $r^{\ast}\le r''\le r$ must have exactly $j''$ cells weakly below it in column $i+2$. However, this implies that $(r'',i+2)\in D'$ for  $r^{\ast}< r''\le r$. In particular, $(r,i+2)\in D'$ so that applying a Kohnert move at row $r$ of $D'$ cannot affect the cell in position $(r,i)$, contradicting our assumption that $D = D' \ldownarrow^{(r,i)}_{(r',i)}$. The result follows.
\end{proof}

Finally, we are ready to finish the proof of  Theorem~\ref{thm:noddcheckered}.
\begin{proof}[Proof of Theorem~\ref{thm:noddcheckered}(b)]

Recall that $m= \left\lfloor\dfrac{n}{2}\right\rfloor$. To establish the result, we define a bijection $$\varphi:\ Min(D_0(n)) \longrightarrow KD(\mathbb{D}_m)$$ and conclude that, by Corollary~\ref{cor:numberdiagDm}, $|Min(D_0(n))| = |KD(\mathbb{D}_m)| = \binom{m}{2}$. Before defining $\varphi$, we note a few key properties of the diagrams in $Min(D_0(n))$. 
\begin{enumerate}
    \item[(i)] By the definition of the checkered diagram $D_0(n)$, each nonempty column has either $m$ or $m+1$ cells. In particular, for $Ch_n^1$, the odd columns have $m+1$ cells and the even columns have $m$ cells, while for $Ch_n^2$, the odd columns have $m$ cells and the even columns have $m+1$ cells. 
    
    \item[(ii)] Since there are at most $m+1$ cells in any column of $T\in Min(D_0(n))$, applying an argument similar to that used in the proof of Lemma~\ref{lem:2rminrest}~(a), it follows that all cells of $T$ must lie weakly below row $m+1$. In particular, the columns with $m+1$ cells have all cells bottom justified and no empty positions weakly below row $m+1$. Moreover, if $D_0(n)= Ch^2_n$, then one can also conclude that the $m$ cells in column $n$ are also bottom justified.

    \item[(iii)] By Lemma~\ref{lem:Ch}, the empty-row sequence that one can associate with the columns of $T\in Min(D_0(n))$ containing exactly $m$ cells (from left to right) is a weakly increasing sequence of entries in $\{1,\hdots,m+1\}$. 
\end{enumerate}
We now define $\varphi$ and $\varphi^{-1}$.
\begin{itemize}
    \item[] $\varphi$: Given $T\in Min(D_0(n))$, we define $D:=\varphi(T)$ to be the diagram obtained by removing the columns containing $m+1$ cells, along with the last nonempty column if $D_0(n)=Ch_n^2$, and then left justifying the remaining nonempty columns. More formally, if $D_0(n)=Ch_n^1$, then $(r,i)\in D$ if and only if $(r,2i)\in T$ for $1\le i\le m$; while if $D_0(n)=Ch_n^2$, then $(r,i)\in D$ if and only if $(r,2i-1)\in T$ for $1\le i\le m$.

    \item[] $\varphi^{-1}$: Given $D\in KD(\mathbb{D}_m)$ and $n$, the definition of $T:= \varphi^{-1}(D)$ depends on whether $D_0(n)=Ch_n^1$ or $Ch_n^2$. If $D_0(n)=Ch_n^1$, then we define $T$ to be the diagram obtained by adding a column with $m+1$ bottom justified cells to the immediate left of each nonempty column of $D$ and to the right of the rightmost nonempty column of $D$; that is, $(r,2i-1)\in T$ for $1\le r\le m+1$ and $1\le i\le m$, and $(r,2i)\in T$ if and only if $(r,i)\in D$ for $1\le i\le m$. On the other hand, if $D_0(n)=Ch_n^2$, then we define $T$ to be the diagram obtained by adding a column with $m+1$ bottom justified cells to the immediate right of each nonempty column of $D$ as well as a column with $m$ bottom justified cells in column $n$; that is, $(r,2i)\in T$ for $1\le r\le m+1$ and $1\le i\le m$, $(r,n)\in T$ for $1\le r\le m$, and $(r,2i-1)\in T$ if and only if $(r,i)\in D$ for $1\le i\le m$.
       
\end{itemize}
In Example~\ref{ex:bijection}, we illustrate several examples of $\varphi$ and $\varphi^{-1}$.

Considering observations $(i)$ -- $(iii)$ above along with Lemma~\ref{lem:sqform}, for $T\in Min(D_0(n))$ we have $\varphi(T)\in KD(\mathbb{D}_m)$. As for $\varphi^{-1}$, given $D\in KD(\mathbb{D}_m)$ it is clear that the diagram $T=\varphi^{-1}(D)$ is a minimal diagram; that is, $T$ is fixed by all Kohnert moves. To see that $T\in Min(D_0(n))$, note that one can form $T$ from $D_0(n)$ by applying the following sequence of Kohnert moves.
\begin{itemize}
    \item[] Step 1: If row 1 of $D$ contains $i_1$ empty positions, perform $$\begin{cases}
        m-i_1, & if~D_0(n)=Ch_n^1 \\
        m-i_1+1, &if~D_0(n)=Ch_n^2
    \end{cases}$$
    Kohnert moves at row $2$.

    \item[] Step 2: Perform $$\begin{cases}
        m+1, & if~D_0(n)=Ch_n^1 \\
        m, & if~D_0(n)=Ch_n^2
    \end{cases}$$ 
    Kohnert moves at row 3 and then, working from rows 4 up to $n$, apply enough Kohnert moves to drop all cells down one row.
    
    \item[] Step $2j-1$: If row $j$ of $D$ contains $i_j$ empty spaces and rows $1$ through $j-1$ contain $e$ empty cells in total, perform $$\begin{cases}
        m-i_j-e, & if~D_0(n)=Ch_n^1 \\
        m-i_j-e+1, &~if D_0(n)=Ch_n^2
    \end{cases}$$
    Kohnert moves at row $j+1$. 

    \item[] Step $2j$: Perform $$\begin{cases}
        m+1, & if~D_0(n)=Ch_n^1 \\
        m, & if~D_0(n)=Ch_n^2
    \end{cases}$$
    Kohnert moves at row $j+2$ and then, working from rows $j+3$ up to $n$, apply enough Kohnert moves to drop all cells down one row.
\end{itemize}
Thus, for $D\in KD(\mathbb{D}_m)$ we have $\varphi^{-1}(D)\in Min(D_0(n))$. Moreover, by definition, $\varphi(\varphi^{-1}(D))=D$ for $D\in KD(\mathbb{D}_m)$ and $\varphi^{-1}(\varphi(T))=T$ for $T\in Min(D_0(n))$. The result follows.
\end{proof}

Considering the proof of Theorem~\ref{thm:noddcheckered}, we have the following immediate corollary.
\begin{corollary}~
    \begin{enumerate}
        \item[\textup{(a)}] If $D_0(n)=Ch_n^1$, then $T\in Min(D_0(n))$ if and only if all cells of $T$ are contained in rows weakly below $\left\lceil\frac{n}{2}\right\rceil$ and the unique empty positions below cells in nonempty even columns of $T$ occur in weakly increasing rows from left to right.
        \item[\textup{(b)}] If $D_0(n)=Ch_n^2$, then $T\in Min(D_0(n))$ if and only if all cells of $T$ are contained in rows weakly below $\left\lceil\frac{n}{2}\right\rceil$, the cells of column $n$ are bottom justified, and the unique empty positions below cells in all other nonempty odd columns of $T$ occur in weakly increasing rows from left to right.
    \end{enumerate}
    See Figure~\ref{fig:bijection} for illustrations of the minimal elements of $Ch_5^1$ \textup(left\textup) and $Ch_5^2$ \textup(right\textup). 
\end{corollary}

As a consequence of Theorem~\ref{thm:noddcheckered}, we get the following characterization for Kohnert posets associated with checkered diagrams which are bounded.

\begin{theorem}\label{thm:checkbound}
    The poset $\mathcal{P}(D_0(n))$ is bounded if and only if $D_0(n)=Ch_1^1$ or $n$ is even.
\end{theorem}

We finish this section with an example illustrating the bijections defined in the proof of Theorem~\ref{thm:noddcheckered}.
\begin{example}\label{ex:bijection}
    In Figure~\ref{fig:bijection} we illustrate the bijections $\varphi$ and $\varphi^{-1}$ defined in the proof of Theorem~\ref{thm:noddcheckered} for $Ch_5^1$ and $Ch_5^2$. The diagrams on the left are the elements of $Min(Ch_5^1)$, the diagrams in the middle are the corresponding elements of $KD(\mathbb{D}_2)$, and the elements on the right are the corresponding elements of $Min(Ch_5^2)$.
    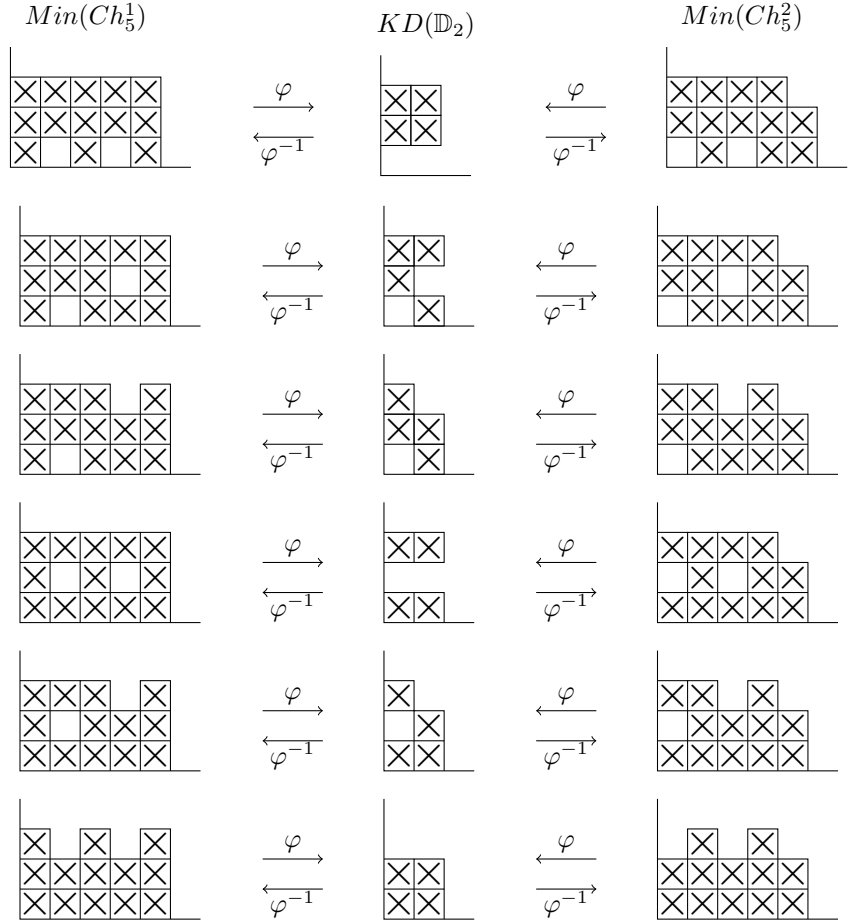
\begin{figure}[H]
        \centering
        $$\begin{tikzpicture}[scale=0.4]
\node at (2.5,5) {$Min(Ch^1_5)$};
  \node at (0.5, 0.5) {$\bigtimes$};
  \node at (0.5, 1.5) {$\bigtimes$};
  \node at (0.5, 2.5) {$\bigtimes$};
  \node at (2.5, 0.5) {$\bigtimes$};
  \node at (2.5, 1.5) {$\bigtimes$};
  \node at (2.5, 2.5) {$\bigtimes$};
  \node at (4.5, 0.5) {$\bigtimes$};
  \node at (4.5, 1.5) {$\bigtimes$};
  \node at (4.5, 2.5) {$\bigtimes$};
  \draw (0,4)--(0,0)--(6,0);
  \draw (0,3)--(1,3)--(1,0);
  \draw (0,1)--(1,1);
  \draw (0,2)--(1,2);
  \draw (2,0)--(2,3)--(3,3)--(3,0);
  \draw (2,1)--(3,1);
  \draw (2,2)--(3,2);
  \draw (4,0)--(4,3)--(5,3)--(5,0);
  \draw (4,1)--(5,1);
  \draw (4,2)--(5,2);
  \node at (1.5, 1.5) {$\bigtimes$};
  \node at (1.5, 2.5) {$\bigtimes$};
  \draw (1,3)--(2,3);
  \draw (1,2)--(2,2);
  \draw (1,1)--(2,1);
  \node at (3.5, 1.5) {$\bigtimes$};
  \node at (3.5, 2.5) {$\bigtimes$};
  \draw (3,3)--(4,3);
  \draw (3,2)--(4,2);
  \draw (3,1)--(4,1);
\end{tikzpicture}\quad\quad \begin{tikzpicture}[scale=0.4]
  \node at (0,0) {};
  \node at (1, 2.5) {$\varphi$};
   \node at (1, 0.5) {$\varphi^{-1}$};
  \draw[->] (0,2)--(2,2);
  \draw[->] (2,1)--(0,1);
\end{tikzpicture}\quad\quad \begin{tikzpicture}[scale=0.4]
\node at (1.5,5) {$KD(\mathbb{D}_2)$};
  \node at (0.5, 1.5) {$\bigtimes$};
  \node at (0.5, 2.5) {$\bigtimes$};
   \node at (1.5, 1.5) {$\bigtimes$};
  \node at (1.5, 2.5) {$\bigtimes$};
  \draw (0,4)--(0,0)--(3,0);
  \draw (0,1)--(1,1)--(1,2)--(0,2);
  \draw (0,2)--(1,2)--(1,3)--(0,3);
  \draw (1,1)--(2,1)--(2,2)--(1,2)--(1,1);
  \draw (1,2)--(2,2)--(2,3)--(1,3)--(1,2);
\end{tikzpicture}\quad\quad \begin{tikzpicture}[scale=0.4]
  \node at (0,0) {};
  \node at (1, 2.5) {$\varphi$};
   \node at (1, 0.5) {$\varphi^{-1}$};
  \draw[->] (2,2)--(0,2);
  \draw[->] (0,1)--(2,1);
\end{tikzpicture}\quad\quad  \begin{tikzpicture}[scale=0.4]
\node at (2.5,5) {$Min(Ch^2_5)$};
  \node at (1.5, 0.5) {$\bigtimes$};
  \node at (1.5, 1.5) {$\bigtimes$};
  \node at (1.5, 2.5) {$\bigtimes$};
  \node at (3.5, 0.5) {$\bigtimes$};
  \node at (3.5, 1.5) {$\bigtimes$};
  \node at (3.5, 2.5) {$\bigtimes$};
  \node at (4.5, 0.5) {$\bigtimes$};
  \node at (4.5, 1.5) {$\bigtimes$};
  \draw (0,4)--(0,0)--(6,0);
  \draw (1,0)--(1,3)--(2,3)--(2,0);
  \draw (1,1)--(2,1);
  \draw (1,2)--(2,2);
  \draw (3,0)--(3,3)--(4,3)--(4,0);
  \draw (3,1)--(4,1);
  \draw (3,2)--(4,2);
  \draw (4,2)--(5,2)--(5,0);
  \draw (4,1)--(5,1);
  \node at (0.5, 1.5) {$\bigtimes$};
  \node at (0.5, 2.5) {$\bigtimes$};
  \draw (0,3)--(1,3);
  \draw (0,2)--(1,2);
  \draw (0,1)--(1,1);
  \node at (2.5, 1.5) {$\bigtimes$};
  \node at (2.5, 2.5) {$\bigtimes$};
  \draw (2,3)--(3,3);
  \draw (2,2)--(3,2);
  \draw (2,1)--(3,1);
\end{tikzpicture}$$ $$\begin{tikzpicture}[scale=0.4]
  \node at (0.5, 0.5) {$\bigtimes$};
  \node at (0.5, 1.5) {$\bigtimes$};
  \node at (0.5, 2.5) {$\bigtimes$};
  \node at (2.5, 0.5) {$\bigtimes$};
  \node at (2.5, 1.5) {$\bigtimes$};
  \node at (2.5, 2.5) {$\bigtimes$};
  \node at (4.5, 0.5) {$\bigtimes$};
  \node at (4.5, 1.5) {$\bigtimes$};
  \node at (4.5, 2.5) {$\bigtimes$};
  \draw (0,4)--(0,0)--(6,0);
  \draw (0,3)--(1,3)--(1,0);
  \draw (0,1)--(1,1);
  \draw (0,2)--(1,2);
  \draw (2,0)--(2,3)--(3,3)--(3,0);
  \draw (2,1)--(3,1);
  \draw (2,2)--(3,2);
  \draw (4,0)--(4,3)--(5,3)--(5,0);
  \draw (4,1)--(5,1);
  \draw (4,2)--(5,2);
  \node at (1.5, 1.5) {$\bigtimes$};
  \node at (1.5, 2.5) {$\bigtimes$};
  \draw (1,3)--(2,3);
  \draw (1,2)--(2,2);
  \draw (1,1)--(2,1);
  \node at (3.5, 0.5) {$\bigtimes$};
  \node at (3.5, 2.5) {$\bigtimes$};
  \draw (3,3)--(4,3);
  \draw (3,2)--(4,2);
  \draw (3,1)--(4,1);
\end{tikzpicture}\quad\quad \begin{tikzpicture}[scale=0.4]
  \node at (0,0) {};
  \node at (1, 2.5) {$\varphi$};
   \node at (1, 0.5) {$\varphi^{-1}$};
  \draw[->] (0,2)--(2,2);
  \draw[->] (2,1)--(0,1);
\end{tikzpicture}\quad\quad \begin{tikzpicture}[scale=0.4]
  \node at (0.5, 1.5) {$\bigtimes$};
  \node at (0.5, 2.5) {$\bigtimes$};
   \node at (1.5, 0.5) {$\bigtimes$};
  \node at (1.5, 2.5) {$\bigtimes$};
  \draw (0,4)--(0,0)--(3,0);
  \draw (0,1)--(1,1)--(1,2)--(0,2);
  \draw (0,2)--(1,2)--(1,3)--(0,3);
  \draw (1,0)--(2,0)--(2,1)--(1,1)--(1,0);
  \draw (1,2)--(2,2)--(2,3)--(1,3)--(1,2);
\end{tikzpicture}\quad\quad \begin{tikzpicture}[scale=0.4]
  \node at (0,0) {};
  \node at (1, 2.5) {$\varphi$};
   \node at (1, 0.5) {$\varphi^{-1}$};
  \draw[->] (2,2)--(0,2);
  \draw[->] (0,1)--(2,1);
\end{tikzpicture}\quad\quad  \begin{tikzpicture}[scale=0.4]
  \node at (1.5, 0.5) {$\bigtimes$};
  \node at (1.5, 1.5) {$\bigtimes$};
  \node at (1.5, 2.5) {$\bigtimes$};
  \node at (3.5, 0.5) {$\bigtimes$};
  \node at (3.5, 1.5) {$\bigtimes$};
  \node at (3.5, 2.5) {$\bigtimes$};
  \node at (4.5, 0.5) {$\bigtimes$};
  \node at (4.5, 1.5) {$\bigtimes$};
  \draw (0,4)--(0,0)--(6,0);
  \draw (1,0)--(1,3)--(2,3)--(2,0);
  \draw (1,1)--(2,1);
  \draw (1,2)--(2,2);
  \draw (3,0)--(3,3)--(4,3)--(4,0);
  \draw (3,1)--(4,1);
  \draw (3,2)--(4,2);
  \draw (4,2)--(5,2)--(5,0);
  \draw (4,1)--(5,1);
  \node at (0.5, 1.5) {$\bigtimes$};
  \node at (0.5, 2.5) {$\bigtimes$};
  \draw (0,3)--(1,3);
  \draw (0,2)--(1,2);
  \draw (0,1)--(1,1);
  \node at (2.5, 0.5) {$\bigtimes$};
  \node at (2.5, 2.5) {$\bigtimes$};
  \draw (2,3)--(3,3);
  \draw (2,2)--(3,2);
  \draw (2,1)--(3,1);
\end{tikzpicture}$$
$$\begin{tikzpicture}[scale=0.4]
  \node at (0.5, 0.5) {$\bigtimes$};
  \node at (0.5, 1.5) {$\bigtimes$};
  \node at (0.5, 2.5) {$\bigtimes$};
  \node at (2.5, 0.5) {$\bigtimes$};
  \node at (2.5, 1.5) {$\bigtimes$};
  \node at (2.5, 2.5) {$\bigtimes$};
  \node at (4.5, 0.5) {$\bigtimes$};
  \node at (4.5, 1.5) {$\bigtimes$};
  \node at (4.5, 2.5) {$\bigtimes$};
  \draw (0,4)--(0,0)--(6,0);
  \draw (0,3)--(1,3)--(1,0);
  \draw (0,1)--(1,1);
  \draw (0,2)--(1,2);
  \draw (2,0)--(2,3)--(3,3)--(3,0);
  \draw (2,1)--(3,1);
  \draw (2,2)--(3,2);
  \draw (4,0)--(4,3)--(5,3)--(5,0);
  \draw (4,1)--(5,1);
  \draw (4,2)--(5,2);
  \node at (1.5, 1.5) {$\bigtimes$};
  \node at (1.5, 2.5) {$\bigtimes$};
  \draw (1,3)--(2,3);
  \draw (1,2)--(2,2);
  \draw (1,1)--(2,1);
  \node at (3.5, 0.5) {$\bigtimes$};
  \node at (3.5, 1.5) {$\bigtimes$};
  \draw (3,2)--(4,2);
  \draw (3,1)--(4,1);
\end{tikzpicture}\quad\quad \begin{tikzpicture}[scale=0.4]
  \node at (0,0) {};
  \node at (1, 2.5) {$\varphi$};
   \node at (1, 0.5) {$\varphi^{-1}$};
  \draw[->] (0,2)--(2,2);
  \draw[->] (2,1)--(0,1);
\end{tikzpicture}\quad\quad \begin{tikzpicture}[scale=0.4]
  \node at (0.5, 1.5) {$\bigtimes$};
  \node at (0.5, 2.5) {$\bigtimes$};
   \node at (1.5, 0.5) {$\bigtimes$};
  \node at (1.5, 1.5) {$\bigtimes$};
  \draw (0,4)--(0,0)--(3,0);
  \draw (0,1)--(1,1)--(1,2)--(0,2);
  \draw (0,2)--(1,2)--(1,3)--(0,3);
  \draw (1,0)--(2,0)--(2,1)--(1,1)--(1,0);
  \draw (1,1)--(2,1)--(2,2)--(1,2)--(1,1);
\end{tikzpicture}\quad\quad \begin{tikzpicture}[scale=0.4]
  \node at (0,0) {};
  \node at (1, 2.5) {$\varphi$};
   \node at (1, 0.5) {$\varphi^{-1}$};
  \draw[->] (2,2)--(0,2);
  \draw[->] (0,1)--(2,1);
\end{tikzpicture}\quad\quad  \begin{tikzpicture}[scale=0.4]
  \node at (1.5, 0.5) {$\bigtimes$};
  \node at (1.5, 1.5) {$\bigtimes$};
  \node at (1.5, 2.5) {$\bigtimes$};
  \node at (3.5, 0.5) {$\bigtimes$};
  \node at (3.5, 1.5) {$\bigtimes$};
  \node at (3.5, 2.5) {$\bigtimes$};
  \node at (4.5, 0.5) {$\bigtimes$};
  \node at (4.5, 1.5) {$\bigtimes$};
  \draw (0,4)--(0,0)--(6,0);
  \draw (1,0)--(1,3)--(2,3)--(2,0);
  \draw (1,1)--(2,1);
  \draw (1,2)--(2,2);
  \draw (3,0)--(3,3)--(4,3)--(4,0);
  \draw (3,1)--(4,1);
  \draw (3,2)--(4,2);
  \draw (4,2)--(5,2)--(5,0);
  \draw (4,1)--(5,1);
  \node at (0.5, 1.5) {$\bigtimes$};
  \node at (0.5, 2.5) {$\bigtimes$};
  \draw (0,3)--(1,3);
  \draw (0,2)--(1,2);
  \draw (0,1)--(1,1);
  \node at (2.5, 0.5) {$\bigtimes$};
  \node at (2.5, 1.5) {$\bigtimes$};
  \draw (2,2)--(3,2);
  \draw (2,1)--(3,1);
\end{tikzpicture}$$
$$\begin{tikzpicture}[scale=0.4]
  \node at (0.5, 0.5) {$\bigtimes$};
  \node at (0.5, 1.5) {$\bigtimes$};
  \node at (0.5, 2.5) {$\bigtimes$};
  \node at (2.5, 0.5) {$\bigtimes$};
  \node at (2.5, 1.5) {$\bigtimes$};
  \node at (2.5, 2.5) {$\bigtimes$};
  \node at (4.5, 0.5) {$\bigtimes$};
  \node at (4.5, 1.5) {$\bigtimes$};
  \node at (4.5, 2.5) {$\bigtimes$};
  \draw (0,4)--(0,0)--(6,0);
  \draw (0,3)--(1,3)--(1,0);
  \draw (0,1)--(1,1);
  \draw (0,2)--(1,2);
  \draw (2,0)--(2,3)--(3,3)--(3,0);
  \draw (2,1)--(3,1);
  \draw (2,2)--(3,2);
  \draw (4,0)--(4,3)--(5,3)--(5,0);
  \draw (4,1)--(5,1);
  \draw (4,2)--(5,2);
  \node at (1.5, 0.5) {$\bigtimes$};
  \node at (1.5, 2.5) {$\bigtimes$};
  \draw (1,3)--(2,3);
  \draw (1,2)--(2,2);
  \draw (1,1)--(2,1);
  \node at (3.5, 0.5) {$\bigtimes$};
  \node at (3.5, 2.5) {$\bigtimes$};
  \draw (3,3)--(4,3);
  \draw (3,2)--(4,2);
  \draw (3,1)--(4,1);
\end{tikzpicture}\quad\quad \begin{tikzpicture}[scale=0.4]
  \node at (0,0) {};
  \node at (1, 2.5) {$\varphi$};
   \node at (1, 0.5) {$\varphi^{-1}$};
  \draw[->] (0,2)--(2,2);
  \draw[->] (2,1)--(0,1);
\end{tikzpicture}\quad\quad \begin{tikzpicture}[scale=0.4]
  \node at (0.5, 0.5) {$\bigtimes$};
  \node at (0.5, 2.5) {$\bigtimes$};
   \node at (1.5, 0.5) {$\bigtimes$};
  \node at (1.5, 2.5) {$\bigtimes$};
  \draw (0,4)--(0,0)--(3,0);
  \draw (0,2)--(2,2)--(2,3)--(0,3);
  \draw (1,2)--(1,3);
  \draw (0,1)--(2,1)--(2,0);
  \draw (1,0)--(1,1);
\end{tikzpicture}\quad\quad \begin{tikzpicture}[scale=0.4]
  \node at (0,0) {};
  \node at (1, 2.5) {$\varphi$};
   \node at (1, 0.5) {$\varphi^{-1}$};
  \draw[->] (2,2)--(0,2);
  \draw[->] (0,1)--(2,1);
\end{tikzpicture}\quad\quad  \begin{tikzpicture}[scale=0.4]
  \node at (1.5, 0.5) {$\bigtimes$};
  \node at (1.5, 1.5) {$\bigtimes$};
  \node at (1.5, 2.5) {$\bigtimes$};
  \node at (3.5, 0.5) {$\bigtimes$};
  \node at (3.5, 1.5) {$\bigtimes$};
  \node at (3.5, 2.5) {$\bigtimes$};
  \node at (4.5, 0.5) {$\bigtimes$};
  \node at (4.5, 1.5) {$\bigtimes$};
  \draw (0,4)--(0,0)--(6,0);
  \draw (1,0)--(1,3)--(2,3)--(2,0);
  \draw (1,1)--(2,1);
  \draw (1,2)--(2,2);
  \draw (3,0)--(3,3)--(4,3)--(4,0);
  \draw (3,1)--(4,1);
  \draw (3,2)--(4,2);
  \draw (4,2)--(5,2)--(5,0);
  \draw (4,1)--(5,1);
  \node at (0.5, 0.5) {$\bigtimes$};
  \node at (0.5, 2.5) {$\bigtimes$};
  \draw (0,3)--(1,3);
  \draw (0,2)--(1,2);
  \draw (0,1)--(1,1);
  \node at (2.5, 0.5) {$\bigtimes$};
  \node at (2.5, 2.5) {$\bigtimes$};
  \draw (2,3)--(3,3);
  \draw (2,2)--(3,2);
  \draw (2,1)--(3,1);
\end{tikzpicture}$$
$$\begin{tikzpicture}[scale=0.4]
  \node at (0.5, 0.5) {$\bigtimes$};
  \node at (0.5, 1.5) {$\bigtimes$};
  \node at (0.5, 2.5) {$\bigtimes$};
  \node at (2.5, 0.5) {$\bigtimes$};
  \node at (2.5, 1.5) {$\bigtimes$};
  \node at (2.5, 2.5) {$\bigtimes$};
  \node at (4.5, 0.5) {$\bigtimes$};
  \node at (4.5, 1.5) {$\bigtimes$};
  \node at (4.5, 2.5) {$\bigtimes$};
  \draw (0,4)--(0,0)--(6,0);
  \draw (0,3)--(1,3)--(1,0);
  \draw (0,1)--(1,1);
  \draw (0,2)--(1,2);
  \draw (2,0)--(2,3)--(3,3)--(3,0);
  \draw (2,1)--(3,1);
  \draw (2,2)--(3,2);
  \draw (4,0)--(4,3)--(5,3)--(5,0);
  \draw (4,1)--(5,1);
  \draw (4,2)--(5,2);
  \node at (1.5, 0.5) {$\bigtimes$};
  \node at (1.5, 2.5) {$\bigtimes$};
  \draw (1,3)--(2,3);
  \draw (1,2)--(2,2);
  \draw (1,1)--(2,1);
  \node at (3.5, 0.5) {$\bigtimes$};
  \node at (3.5, 1.5) {$\bigtimes$};
  \draw (3,2)--(4,2);
  \draw (3,1)--(4,1);
\end{tikzpicture}\quad\quad \begin{tikzpicture}[scale=0.4]
  \node at (0,0) {};
  \node at (1, 2.5) {$\varphi$};
   \node at (1, 0.5) {$\varphi^{-1}$};
  \draw[->] (0,2)--(2,2);
  \draw[->] (2,1)--(0,1);
\end{tikzpicture}\quad\quad \begin{tikzpicture}[scale=0.4]
  \node at (0.5, 0.5) {$\bigtimes$};
  \node at (0.5, 2.5) {$\bigtimes$};
   \node at (1.5, 0.5) {$\bigtimes$};
  \node at (1.5, 1.5) {$\bigtimes$};
  \draw (0,4)--(0,0)--(3,0);
  \draw (0,2)--(1,2)--(1,3)--(0,3);
  \draw (1,1)--(1,2)--(2,2)--(2,1);
  \draw (0,1)--(2,1)--(2,0);
  \draw (1,0)--(1,1);
\end{tikzpicture}\quad\quad \begin{tikzpicture}[scale=0.4]
  \node at (0,0) {};
  \node at (1, 2.5) {$\varphi$};
   \node at (1, 0.5) {$\varphi^{-1}$};
  \draw[->] (2,2)--(0,2);
  \draw[->] (0,1)--(2,1);
\end{tikzpicture}\quad\quad  \begin{tikzpicture}[scale=0.4]
  \node at (1.5, 0.5) {$\bigtimes$};
  \node at (1.5, 1.5) {$\bigtimes$};
  \node at (1.5, 2.5) {$\bigtimes$};
  \node at (3.5, 0.5) {$\bigtimes$};
  \node at (3.5, 1.5) {$\bigtimes$};
  \node at (3.5, 2.5) {$\bigtimes$};
  \node at (4.5, 0.5) {$\bigtimes$};
  \node at (4.5, 1.5) {$\bigtimes$};
  \draw (0,4)--(0,0)--(6,0);
  \draw (1,0)--(1,3)--(2,3)--(2,0);
  \draw (1,1)--(2,1);
  \draw (1,2)--(2,2);
  \draw (3,0)--(3,3)--(4,3)--(4,0);
  \draw (3,1)--(4,1);
  \draw (3,2)--(4,2);
  \draw (4,2)--(5,2)--(5,0);
  \draw (4,1)--(5,1);
  \node at (0.5, 0.5) {$\bigtimes$};
  \node at (0.5, 2.5) {$\bigtimes$};
  \draw (0,3)--(1,3);
  \draw (0,2)--(1,2);
  \draw (0,1)--(1,1);
  \node at (2.5, 0.5) {$\bigtimes$};
  \node at (2.5, 1.5) {$\bigtimes$};
  \draw (2,2)--(3,2);
  \draw (2,1)--(3,1);
\end{tikzpicture}$$
$$\begin{tikzpicture}[scale=0.4]
  \node at (0.5, 0.5) {$\bigtimes$};
  \node at (0.5, 1.5) {$\bigtimes$};
  \node at (0.5, 2.5) {$\bigtimes$};
  \node at (2.5, 0.5) {$\bigtimes$};
  \node at (2.5, 1.5) {$\bigtimes$};
  \node at (2.5, 2.5) {$\bigtimes$};
  \node at (4.5, 0.5) {$\bigtimes$};
  \node at (4.5, 1.5) {$\bigtimes$};
  \node at (4.5, 2.5) {$\bigtimes$};
  \draw (0,4)--(0,0)--(6,0);
  \draw (0,3)--(1,3)--(1,0);
  \draw (0,1)--(1,1);
  \draw (0,2)--(1,2);
  \draw (2,0)--(2,3)--(3,3)--(3,0);
  \draw (2,1)--(3,1);
  \draw (2,2)--(3,2);
  \draw (4,0)--(4,3)--(5,3)--(5,0);
  \draw (4,1)--(5,1);
  \draw (4,2)--(5,2);
  \node at (1.5, 0.5) {$\bigtimes$};
  \node at (1.5, 1.5) {$\bigtimes$};
  \draw (1,2)--(2,2);
  \draw (1,1)--(2,1);
  \node at (3.5, 0.5) {$\bigtimes$};
  \node at (3.5, 1.5) {$\bigtimes$};
  \draw (3,2)--(4,2);
  \draw (3,1)--(4,1);
\end{tikzpicture}\quad\quad \begin{tikzpicture}[scale=0.4]
  \node at (0,0) {};
  \node at (1, 2.5) {$\varphi$};
   \node at (1, 0.5) {$\varphi^{-1}$};
  \draw[->] (0,2)--(2,2);
  \draw[->] (2,1)--(0,1);
\end{tikzpicture}\quad\quad \begin{tikzpicture}[scale=0.4]
  \node at (0.5, 0.5) {$\bigtimes$};
  \node at (0.5, 1.5) {$\bigtimes$};
   \node at (1.5, 0.5) {$\bigtimes$};
  \node at (1.5, 1.5) {$\bigtimes$};
  \draw (0,4)--(0,0)--(3,0);
  \draw (0,2)--(1,2);
  \draw (1,1)--(1,2)--(2,2)--(2,1);
  \draw (0,1)--(2,1)--(2,0);
  \draw (1,0)--(1,1);
\end{tikzpicture}\quad\quad \begin{tikzpicture}[scale=0.4]
  \node at (0,0) {};
  \node at (1, 2.5) {$\varphi$};
   \node at (1, 0.5) {$\varphi^{-1}$};
  \draw[->] (2,2)--(0,2);
  \draw[->] (0,1)--(2,1);
\end{tikzpicture}\quad\quad  \begin{tikzpicture}[scale=0.4]
  \node at (1.5, 0.5) {$\bigtimes$};
  \node at (1.5, 1.5) {$\bigtimes$};
  \node at (1.5, 2.5) {$\bigtimes$};
  \node at (3.5, 0.5) {$\bigtimes$};
  \node at (3.5, 1.5) {$\bigtimes$};
  \node at (3.5, 2.5) {$\bigtimes$};
  \node at (4.5, 0.5) {$\bigtimes$};
  \node at (4.5, 1.5) {$\bigtimes$};
  \draw (0,4)--(0,0)--(6,0);
  \draw (1,0)--(1,3)--(2,3)--(2,0);
  \draw (1,1)--(2,1);
  \draw (1,2)--(2,2);
  \draw (3,0)--(3,3)--(4,3)--(4,0);
  \draw (3,1)--(4,1);
  \draw (3,2)--(4,2);
  \draw (4,2)--(5,2)--(5,0);
  \draw (4,1)--(5,1);
  \node at (0.5, 0.5) {$\bigtimes$};
  \node at (0.5, 1.5) {$\bigtimes$};
  \draw (0,2)--(1,2);
  \draw (0,1)--(1,1);
  \node at (2.5, 0.5) {$\bigtimes$};
  \node at (2.5, 1.5) {$\bigtimes$};
  \draw (2,2)--(3,2);
  \draw (2,1)--(3,1);
\end{tikzpicture}$$
        \caption{$\varphi$ and $\varphi^{-1}$ for $Min(Ch^1_5)$ (left) and $Min(Ch^2_5)$ (right)}
        \label{fig:bijection}
    \end{figure}
\end{example}

\section{Conclusion}

In this article, we initiated an investigation into the posets naturally associated with collections of Kohnert diagrams introduced in \cite{KP2}. In particular, we focused on characterizing when such posets are bounded and/or ranked. While we were able to establish some general results (see Proposition~\ref{prop:nummingen} and Theorems~\ref{thm:ranked1} and~\ref{thm:ranked2}) as well as results for certain special families of diagrams (see Theorems~\ref{thm:obpcmin},~\ref{thm:obpcrank},~\ref{thm:2rmin},~\ref{thm:2rrank},~\ref{thm:uniqminkey},~\ref{thm:keyrank},~\ref{thm:checkrank}, and~\ref{thm:checkbound}), a complete characterization seems to be much more complex. In fact, we suspect that a general characterization of boundedness may not exist for Kohnert posets. However, we believe that a general result concerning rankedness does exist. Indeed, it seems as though the necessary conditions provided by Theorems~\ref{thm:ranked1} and~\ref{thm:ranked2}, or perhaps slight generalizations, are also sufficient; that is, we conjecture the following.
\begin{conj}
    There exists a finite number of special families of subdiagrams $\mathcal{F}$ such that given any diagram $D_0$, $\mathcal{P}(D_0)$ is ranked if and only if there is no $D\in KD(D_0)$ such that $D$ contains a subdiagram from $\mathcal{F}$. 
\end{conj}
\noindent
The main obstacle to establishing such a result is determining a rank function since the common choice of $rowsum(D)-b(D_0)$ used throughout this article does not work in general, e.g., $D_0=\{(2,1),(3,1),(1,2),(2,2)\}$.

This article represents the first attempt to study the ``not-so-well-behaved'' structure of Kohnert posets in general. Therefore, there are many other interesting questions remaining. In particular, the third author (N. M.) is studying the shellability of Kohnert posets~\cite{NNC23}, and the fourth author (E. P.) is working on identifying those Kohnert posets that are lattices~\cite{E24}.


\printbibliography

\end{document}